\documentclass[twoside,10pt]{article}

\usepackage[utf8]{inputenc}
\usepackage[T1]{fontenc}
\usepackage{times}
\usepackage{amsmath,amsfonts,amssymb,amsthm,euscript,pifont}
\usepackage{mathrsfs}
\usepackage{authblk}
\usepackage[pagebackref,colorlinks=true]{hyperref}
\hypersetup{
    linkcolor = {blue}, 
    citecolor = {red}
    }
\usepackage{fullpage}
\usepackage{xspace}
\usepackage{float}
\restylefloat{table}
\usepackage{graphicx} 
\usepackage{caption}
\usepackage{subcaption}
\usepackage{multirow}
\usepackage{color}
\usepackage{xcolor,colortbl}

\newcommand{\eps}{\varepsilon}
\renewcommand{\P}{\mathbb{P}}
\newcommand{\F}{\mathcal{F}}
\newcommand{\R}{\mathbb{R}}
\newcommand{\er}{\mathbb{R}}
\newcommand{\N}{\mathbb{N}}
\newcommand{\Q}{\mathbb{Q}}
\newcommand{\E}{\mathbb{E}}
\newcommand{\e}{\mathcal{E}}

\usepackage{dsfont}
\newcommand{\ind}[1]{\mathds{1}_{#1}}

\newcommand{\sgn}{\text{sgn}}
\def\erfc{\mathop{\mbox{\normalfont erfc}}\nolimits}
\newcommand{\stalpha}{\Theta_{\alpha}}
\newcommand{\stonehalf}{\Theta_{\text{\tiny{$\tfrac{1}{2}$}}}}
\newcommand{\lo}{l_\sigma(\alpha)}
\newcommand{\X}{\overline{X}}
\newcommand{\hX}{\widehat{X}}
\newcommand{\Xe}[1]{\overline{X}_{\eta(#1)}}
\newcommand{\Z}[1]{\overline{Z}_{#1}}

\newcommand{\dt}{\Delta t}
\newcommand{\ds}{\Delta s}
\newcommand{\dW}{\Delta W}
\newcommand{\dS}{D }

\newcommand{\xx}{\bar{x}(\alpha)}
\newcommand{\XX}{(X_t,0\leq t\leq T)}
\newcommand{\XXb}{(\X_t,0\leq t\leq T)}
\newcommand{\XXh}{(\hX_t,0\leq t\leq T)}

\newcommand{\GG}{(\g_{t},0\leq t\leq T)}
\newcommand{\g}{\Gamma}
\newcommand{\bb}{{\beta}}

\newcommand{\dmax}{\Delta_{\max}}

\newtheorem{theo}{Theorem}[section]
\newtheorem{lem}[theo]{Lemma}
\newtheorem{proposition}[theo]{Proposition}
\newtheorem{hypothesis}[theo]{Hypothesis}
\newtheorem{corollary}[theo]{Corollary}

\newtheorem{rem}[theo]{Remark} 

\numberwithin{equation}{section}
\usepackage{color}

\begin{document}
\title{Strong convergence of the symmetrized {M}ilstein scheme for some {CEV}-like {SDE}s}

\author[1]{Mireille Bossy\thanks{email: mireille.bossy@inria.fr}}
\author[2]{H\'ector Olivero \thanks{email: holivero@dim.uchile.cl. Financially supported by: Proyecto Mecesup UCH0607, the Direcci\'on de Postgrado y Post\'itulo de la Vicerrector\'ia de Asuntos Acad\'emicos de la Universidad de Chile, the Instituto Franc\'es de Chile - Embajada de Francia en Chile, and the Center for Mathematical Modeling CMM.
}   }
\affil[1]{TOSCA Laboratory, INRIA Sophia Antipolis -- M\'editerran\'ee, France}
\affil[2]{Departamento de Ingenier\'ia Matem\'atica, Universidad de Chile, Chile}
\date{\today}

\maketitle
\begin{abstract}
In this paper we study the rate of convergence of a symmetrized version of the Milstein scheme applied to the solution of the  one dimensional SDE 
\begin{equation*}
X_t = x_0 + \int_{0}^t{b(X_s)ds}+\int_{0}^t{\sigma |X_s|^\alpha dW_s}, \;x_0>0,\;\sigma>0,\; \alpha\in[\tfrac{1}{2},1).
\end{equation*}
Assuming $b(0)/\sigma^2$ big enough, and $b$ smooth, we prove a strong rate of convergence of order one, recovering the classical result of Milstein for SDEs with smooth  diffusion coefficient.  In contrast with other recent results, our proof does not relies on Lamperti transformation, and it can be applied to a wide class of drift functions. On the downside, our hypothesis on the critical parameter value $b(0)/\sigma^2$ is more restrictive than others available in the literature.  Some numerical experiments and comparison with various other schemes complement our theoretical analysis that also applies for the simple projected Milstein scheme with same convergence rate. 
\end{abstract}

\section{Introduction and main result}
The Milstein scheme was introduced by Milstein in \cite{Milstein:1974uq} for  one dimensional  Stochastic Differential Equations (SDEs) having smooth diffusion coefficient. Introducing an appropriated correction term,   this scheme has better convergence rate for the strong error than  the classical Euler-Maruyama scheme. 
Typically, when the drift and diffusion coefficient of one dimensional SDE are twice continuously differentiable  with bounded derivatives, the Milstein scheme is of order one for strong error (see eg. Talay \cite{talay-96}) instead of one-half for the Euler-Maruyama scheme. 
This well-know fact produces remarks on blogs, internet forums, and software packages that sometimes recommend to use the  Milstein scheme  for  constant elasticity of variance (CEV) models in finance, or its extension with stochastic volatility as SABR model, (see e.g Delbaen and Shirakawa \cite{Delbaen-Shirakawa_02} and  Lions and Musiela \cite{Lions20071} for a discussion on the (weak) existence of such models); CEV are popular stochastic volatility models of the form
\[ dX_t = \mu X_t dt + \sigma X_t^\gamma dW_t\]
with $0<\gamma<1$.  
But the interesting fact in this story is that the rate of convergence of  the Milstein scheme, for such family of processes with $0<\gamma<1$ is not yet well studied, to the best of our knowledge. 

In this paper we establish  a rate of convergence result for a symmetrized version of the Milstein scheme applied to the solution of the one dimensional SDE
\begin{equation}\label{exactProcess}
X_t = x_0 + \int_{0}^t{b(X_s)ds}+\int_{0}^t{\sigma |X_s|^\alpha dW_s},
\end{equation}
where $x_0>0$, $\sigma>0$ and $\tfrac{1}{2}\leq \alpha < 1$. Of course  Equation \eqref{exactProcess} does not satisfies the hypothesis to apply the classical result of Milstein \cite{Milstein:1974uq}. In particular, the diffusion coefficient is only H\"older continuous whereas the classical hypothesis is to have a $\mathcal{C}^2$ diffusion coefficient. 

The main picture of our convergence rate result is that Milstein scheme stays of order one in the case of Equation \eqref{exactProcess}, but some attention must be paid to the values of $b(0)$, $\alpha$ and $\sigma$. 

There exist in the literature other strategies for the discretization of the solution to \eqref{exactProcess}. There are some results based on the Lamperti transformation of the equation,  for example, by Alfonsi \cite{Alfonsi:2005aa, Alfonsi:2013aa}, and by  Chassagneux, Jacquier and Mihaylov \cite{Chassagneux:2015aa}.  And also, there  some are results where the equation \eqref{exactProcess} is discretized directly, as in  Berkaoui, Bossy and Diop \cite{BERKAOUI:2008fk} or in Kahl and Jack\"{e}l \cite{Kahl:2006aa}. In the numerical experiments section,  we compare the symmetrized Milstein scheme with a selection of schemes proposed in the aforementioned references.  We also experiment the symmetrized Milstein scheme in a multilevel Monte Carlo application and we compare with other schemes.  

\medskip

In the whole paper, we work under the following basis-hypothesis:
\begin{hypothesis}\label{hypothesisH0}
The power parameter  $\alpha$ in the diffusion coefficient of Equation \eqref{exactProcess} belongs to $[\tfrac{1}{2},1)$. The drift coefficient  $b$ is Lipschitz with constant $K>0$, and is such that  $b(0)>0$. 
\end{hypothesis}

Hypothesis  \ref{hypothesisH0} is a classical assumption to ensure a unique strong solution valued in $\er^+$. We assume  it in all the forthcoming results of the paper,  without recall it explicitly.  
To state the convergence  result (see Theorem \ref{mainTheorem}), another Hypothesis \ref{hypo:2} will be added and discussed, that in particular  constrains the values $\alpha$,  $b(0)$ and $\sigma$. 

\subsection{The symmetrized Milstein scheme}\label{sec:defSchemeAndMainResult}

To complete our task we follow the ideas of Berkaoui, Bossy and Diop in \cite{BERKAOUI:2008fk} who analyze the rate of convergence of the strong error for the symmetrized Euler scheme applied to Equation \eqref{exactProcess}. Although, whereas they utilize an argument of change of time, we consider first a weighted $L^p(\Omega)$-error for which we prove a convergence result, and then we utilize this result to prove the convergence of the actual $L^p(\Omega)$-error.

We consider  $x_0>0$, $T>0$,  and $N\in\N$. We define the constant step size $\dt=T/N$ and $t_k=k\dt$. Over this discretization of the interval $[0,T]$ we define the Symmetrized Milstein Scheme (SMS) $(\X_{t_k},k=0,\ldots,N)$ by

$$ \X_{t_{k}} =\left\{ \begin{array}{ll}x_0,\mbox{ for }k=0, \\
 \left|\X_{t_{k-1}}+b(\X_{t_{k-1}})\dt+\sigma\X_{t_{k-1}}^\alpha(W_{t_k}-W_{t_{k-1}})
+\dfrac{\alpha\sigma^2}{2}\X_{t_{k-1}}^{2\alpha-1}\left[(W_{t_k}-W_{t_{k-1}})^2 -\dt \right]\right|,  \\
\qquad\mbox{ for }k =1,\ldots,N.
 \end{array}
\right.
$$
In the following, we use the time continuous version of the SMS, $(\X_t, 0\leq t\leq T)$ satisfying 
\begin{align}\label{defXbarContinuous}
\X_t = \Big|\Xe{t}+b(\Xe{t})(t-\eta(t))  &+\sigma\Xe{t}^\alpha(W_t-W_{\eta(t)})+\frac{\alpha\sigma^2}{2}\Xe{t}^{2\alpha-1}\left[(W_t-W_{\eta(t)})^2 -(t-\eta(t)) \right]\Big|,
\end{align}
where $\eta(t)=\sup_{k\in\{1,\ldots,N\}}\{t_k:t_k\leq t\}$. We also introduce the increment process $(\Z{t}, 0\leq t\leq T)$ defined by
\begin{align}\label{defZ}
\Z{t} = \Xe{t}+b(\Xe{t})(t-\eta(t))&+\sigma\Xe{t}^\alpha(W_t-W_{\eta(t)})+\frac{\alpha\sigma^2}{2}\Xe{t}^{2\alpha-1}\left[(W_t-W_{\eta(t)})^2 -(t-\eta(t)) \right],
\end{align}
so that  $\X_t=|\Z{t}|$. Thanks to Tanaka's Formula, the semi-martingale decomposition of $\X_t$ is given by
\begin{align}\label{XbarIsSemiMartingale}
\X_t = x_0+ \int_0^t{\sgn(\Z{s})b(\Xe{s})ds} + L_t^0(\X) +\int_0^t\sgn(\Z{s})\left[\sigma\Xe{s}^\alpha+\alpha\sigma^2\Xe{s}^{2\alpha-1}(W_s-W_{\eta(s)})\right]dW_s
\end{align}
where $\sgn(x)  = 1 - 2 \ind{[x\leq 0]}$.

\subsubsection*{Moment upper bound estimations for $X$ and $\X$}
We summarize some facts about the process $\XX$, the proofs of which can be found in  Bossy and Diop~\cite{BOSSY:2013fk}.
 
\begin{lem}\label{lem:LemmaMomentsOfX}
For any $q\geq 1$, there exists a positive constant $C$ depending on $q$, but also on the parameters $b(0)$, $K$, $\sigma$, $\alpha$ and $T$ such that, for any $x_0 >  0$, 
\begin{equation}\label{boundedMomentsOfX}
\E\Big[\sup_{0\leq t\leq T}X_t^{2q} \Big]\leq C(1+x_0^{2q}).
\end{equation}
When $\tfrac{1}{2}<\alpha<1$, for any $q>0$,
\begin{equation} \label{boundedNegativeMomentsOfX}
\sup_{0\leq t\leq T} \E\left[ X_t^{-q} \right]\leq C(1+x_0^{-q}).
\end{equation}
When $\alpha=\tfrac{1}{2}$, for any $q$ such that $1<q<\tfrac{2b(0)}{\sigma^2} -1$,
\begin{equation} \label{boundedNegativeMomentsOfXcir}
\sup_{0\leq t\leq T} \E\left[ X_t^{-q} \right]\leq Cx_0^{-q}.
\end{equation}
\end{lem}

\begin{lem}\label{LemmaExponentialMomentOfX}
Let $\XX$ be the solution of \eqref{exactProcess} with $\tfrac{1}{2} < \alpha < 1$. For all $\mu\geq0$, there exists a positive constant $C(T,\mu)$, increasing in $\mu$ and $T$, depending also on $b,\; \sigma,\;\alpha$ and $x_0$ such that
\begin{equation}\label{eqExponentialMomentsExactProcess}
\E\exp\left( \mu \int_{0}^T{\frac{ds}{X_s^{2(1-\alpha)}}} \right)\leq C(T,\mu).
\end{equation}
When $\alpha=\tfrac{1}{2}$, the inequality \eqref{eqExponentialMomentsExactProcess} holds if $ b(0) > \tfrac{\sigma^2}{2}$ and $\mu\leq \tfrac{\sigma^2}{8}(\tfrac{2b(0)}{\sigma^2}-1)^2$.
\end{lem}
Notice that the condition $b(0) > \sigma^2/2$ is also imposed by the Feller test in the case $\alpha =\tfrac{1}{2}$ for the strict positivity of $X$, that allows to rewrite Equation \eqref{exactProcess} as 
\begin{align*}
X_t = x_0 + \int_{0}^t{b(X_s)ds}+\int_{0}^t \sigma \sqrt{X_s} dW_s.
\end{align*}

Using the semimartingale representation \eqref{XbarIsSemiMartingale}, we prove the following Lemma regarding the existence of moments of any order for $\X_t$. 
\begin{lem}\label{lem:finitenessOfTheMomentsOfXbar}
For any $q\geq 1$, there exists a positive constant $C$ depending on $q$, but also on the parameters $b(0)$, $K$, $\sigma$, $\alpha$ and $T$ such that for any $x_0 >  0$, 
$$\E\Big[\sup_{t\in[0,T]}\X_t^{2q} \Big]\leq C(1+x_0^{2q}).$$  
\end{lem}
The proof of this lemma is based on the  Lipschitz property of $b$ and classical combination of It{\^o} formula and Young Inequality. For the sake of completeness, we give a short proof in the Appendix. 

\subsection{Strong rate of convergence}
The main result of this works is the strong convergence at rate one of the SMS $\X$ to the exact process $X$. The convergence holds in $L^{p}$ for $p\geq 1$.  To state it, we add to Hypothesis \ref{hypothesisH0} the following. \\
For any $x$ in $\er^+$, we denote $\lceil x \rceil$  the rounded up integer. 

\begin{hypothesis}\label{hypo:2}

\item{(i)} Let $p\geq1$. To control the $L^p(\Omega)$-norm of the error, if $\alpha>\tfrac{1}{2}$ we assume $b(0)>2\alpha(1-\alpha)^2\sigma^2$. Whereas for $\alpha = \tfrac{1}{2}$ we assume  $b(0)> 3( 2[p\lor2]+1)\sigma^2/2$.

\item{(ii)} The drift coefficient  $b$ is of class $\mathcal{C}^2(\R)$, and $b''$ has polynomial growth.
\end{hypothesis}

We now state our main theorem. 
To lighten the notation, we consider  for $\alpha\in(\tfrac{1}{2},1)$ 
\begin{align}\label{def:b_alpha}
\left\{
\begin{aligned}
b_\sigma(\alpha)&:=b(0)-2(1-\alpha)^2\alpha\sigma^2,\\
K(\alpha)& := K + \frac{\alpha\sigma^2}{2}(2\alpha-1)[2(1-\alpha)]^{-\frac{2(1-\alpha)}{2\alpha-1}}
\end{aligned}
\right.
\end{align}
and we extend this definitions to $\alpha=\tfrac{1}{2}$ taking limits. So, $b_\sigma(\tfrac{1}{2})=\lim_{\alpha\to\tfrac{1}{2}}b_\sigma(\alpha)=b(0)-\sigma^2/4$, and $K(\tfrac{1}{2})=\lim_{\alpha\to\tfrac{1}{2}}K(\alpha)=K$.  Notice that  $\lim_{\alpha\to1}K(\alpha)=K+\sigma^2/2$,  and since $K(\alpha)$ is continuous on $(\tfrac{1}{2},1)$, we have that $K(\alpha)$ is bounded. This is especially important in the definition of $\dmax(\alpha)$ bellow, because tells us that $\alpha \mapsto \dmax(\alpha)$ is strictly positive and bounded on $[\tfrac{1}{2},1)$.

\begin{theo}\label{mainTheorem}Assume Hypotheses \ref{hypothesisH0} and \ref{hypo:2}.  Define a maximum step size $\dmax(\alpha)$ as 
\begin{equation}\label{eq:def_dt0_alpha}
\dmax(\alpha) =  \dfrac{x_0}{(1-\sqrt{\alpha})b_\sigma(\alpha)} \wedge  \left\{\begin{array}{ll}
 \dfrac{1}{4\alpha K(\alpha)},\quad\mbox{ for } \alpha \in(\tfrac{1}{2},1)\\
 \dfrac{1}{4K}  \wedge x_0,\quad\mbox {for  }\alpha=\tfrac{1}{2}. 
\end{array}
\right.
\end{equation}
Let $\XX$ be the process defined on \eqref{exactProcess} and $\XXb$ the  symmetrized Milstein scheme given in \eqref{defXbarContinuous}. 
Then for  any $p\geq1$ that allows Hypotheses \ref{hypo:2}, there exists  a constant $C$ depending on $p$, $T$, $b(0)$, $\alpha$, $\sigma$, $K$,  and $x_0$, but not on $\dt$, such that for all $\dt\leq\dmax(\alpha)$, 
\begin{align}\label{eq:convergence}
\sup_{0\leq t\leq T}\Big( \E\left[|X_t-\X_t|^{p}\right]\Big)^{\frac{1}{p}}\leq C\dt.
\end{align}
\end{theo}

\paragraph{About Hypothesis \ref{hypo:2}. }   Notice that for $\alpha>\tfrac{1}{2}$, 
 Assumption $(i)$ does not depend on $p$ and becomes easier to fulfill as $\alpha$ increases.  On the other hand, for $\alpha=\tfrac{1}{2}$, Assumption $(i)$ depends on $p$ in a unpleasant manner. However, as we will see later in Section~\ref{sec:numExp} (see Table~ \ref{table:ConditionsOverParametersDifferentSchemesAlpha1/2}), this kind of dependence in $p$ is expected, and similar conditions are asked in the literature for other approximation schemes in order to obtain similar rate of convergence results. 

Also, notice that $(i)$ is a sufficient condition: in the numerical experiments we still observe a rate of convergence of order one for parameters that do not satisfy it, but we also observe that  for parameters such that $b(0)\ll\sigma^2$, although the convergence occurs, it does in a sublinear fashion. 

On the other hand, Assumption {$(ii)$} is the classical requirement for the strong convergence of the Milstein scheme. As we will see later in the proof of the main theorem, with the help of the Itô formula, this hypothesis let us conclude that 
$$\E\left[ |X_s -\X_s|^{2p-1}\big(b(X_{\eta(s)})-b(X_s)\big)\right]\leq C\left( \sup_{u\leq s} \E\left[|X_u -\X_u|^{2p}\right] + \dt^{2p}\right)$$ 
instead of 
$$\E\left[   |X_s -\X_s|^{2p-1}\big(b(X_{\eta(s)})-b(X_s)\big)\right]\leq C\left(\sup_{u\leq s}\E\left[|X_u -\X_u|^{2p}\right] + \dt^{p}\right), $$ 
which is the classical bound obtained for the Euler-Maruyama scheme under a Lipschitz condition for a drift $b$.

\medskip
The rest of the paper is organized as follow.  In Section \ref{sec:PreliminaryResults} we state some preliminary results on the scheme which will be building blocks in the proof of  Theorem \ref{mainTheorem}.   Section \ref{sec:theProof} is devoted to the proof of the convergence rate. 
 The main idea is first to introduce a weight process in the  $L^{4p}(\Omega)$-error. Wet get the rate of convergence for this weighted error process, and we use this intermediate bound to control the $L^{2p}(\Omega)$-error, from where we finally control the  $L^p(\Omega)$-error. Also, as a byproduct we obtain the order one convergence of the Projected Milstein Scheme (see \ref{sec:ConvergencePMS}). In  section \ref{sec:numExp}, we display some numerical experiments to show the effectiveness of the theoretical rate of convergence of the scheme, but also to test  Hypotheses \ref{hypo:2}-$(i)$ on a set of parameters.  In this section we also shows how the inclusion of the SMS in a Multilevel Monte Carlo framework could help to optimize  the computational  time of weak approximation of  assets valuation. In Section \ref{sec:ProofForPreliminaryResultsOnXbar} we present the proof of the preliminary results on the scheme. Finally we have included in a small appendix a couple of proofs to make this paper self contained. 

\section{Some preliminary results for $\X$}\label{sec:PreliminaryResults}

This short section is devoted to state some results about the behavior of $\X$, their proofs are postponed to Section~\ref{sec:ProofForPreliminaryResultsOnXbar}. All these results hold under  Hypothesis \ref{hypo:2}-$(i)$ which is in fact stronger than what we need here. So, we present the next lemmas with their minimal hypotheses (still assuming Hypothesis \ref{hypothesisH0}). 

\begin{lem}[Local error] \label{lem:localErrorOfTheScheme} For any {$x_0 >  0$}, for any {$p\geq 1$}, there exists a positive constant $C$, depending on $p$, $T$, the parameters of the model $b(0)$, $K$, $\sigma$, $\alpha$, but not on $\dt$ such that 
\[\sup_{0\leq t\leq T}\E\left[ |\X_t -\Xe{t}|^{2p}  \right]\leq C\dt^p.\]
\end{lem}

By construction the scheme $\X$ is nonnegative, but a key point of the convergence proof resides in the analysis of the behavior of $\X$ or $\Z{}$ visiting the point 0.  
The next Lemma shows that although $\Z{t}$ is not always positive, the probability of $\Z{t}$ being negative is actually very small under suitable  hypotheses.

\begin{lem}\label{lem:ZbarRemainsPositive}
For  $\alpha\in[\tfrac{1}{2},1)$, if $b(0)>2\alpha(1-\alpha)^2\sigma^2$,  and $\dt\leq \tfrac{1}{2 K(\alpha)}$, then there exists a positive constant $\gamma$, depending on the parameters of the model, but not on $\dt$, such that
\begin{equation}\label{eq:ProbabilityOfReflection}
\sup_{k=0,\ldots,N-1}\P\left( \inf_{t_k< s\leq t_{k+1}}\Z{s}\leq0 \right)\leq \exp\left( -\frac{\gamma}{\dt}\right).
\end{equation}
In particular,
\begin{equation}\label{eq:ZbarRemainsPositive}
\sup_{0\leq t\leq T}\P\left( \Z{t}\leq0 \right)\leq C\exp\left( - \frac{\gamma}{\dt}\right).
\end{equation}
\end{lem}

To prove Lemma \ref{lem:ZbarRemainsPositive}, it is necessary to establish before the following one, which although technical, gives some intuition about the difference between the SMS and the Symmetrized Euler scheme presented in \cite{BERKAOUI:2008fk}.  
\begin{lem}\label{lem:ZbarPositiveSmallx}
For $\alpha\in[\tfrac{1}{2},1)$,  if $b(0)>2\alpha(1-\alpha)^2\sigma^2$, we set $\xx =  \tfrac{b_\sigma(\alpha)}{K(\alpha)}>0$. Then for all $t\in[0,T]$, and  for all $\rho\in(0,1]$, 
$$\P\left[\Z{t}\leq (1-\rho )b_\sigma(\alpha)\dt ,\; \Xe{t}<{\rho} \xx \right. ]=0.$$
\end{lem}
Roughly speaking, from this lemma we see that when $\Z{\eta(t)}>0$, $\Z{t}$ becomes negative only when 
\[|\Z{t}-\Z{\eta(t)}|>{\rho}\xx.\]
But observe that only the left-hand side of this inequality  depends on $\dt$,  and its expectation decreases to zero proportionally  to $\sqrt{\dt}$, according to Lemma~\ref{lem:localErrorOfTheScheme}. 

Now imposing $\dt$ small enough, we prove an explicit bound for the local time moment of $\X$. 
\begin{lem}\label{lem:Control2ndMomentLocalTime} For $\alpha\in[\tfrac{1}{2},1)$, if $b(0)>2\alpha(1-\alpha)^2\sigma^2$ and $\dt\leq \tfrac{1}{2K(\alpha)} \wedge \tfrac{x_0}{(1-\sqrt{\alpha})b_\sigma(\alpha)}$, then there exist  positive constants $C$ and  $\gamma>0$ depending on $\alpha$, $b(0)$, $K$, and $\sigma$ but not in $\dt$ such that
$$\E\left(L_T^0(\X)^2\right)\leq C\frac{1}{\sqrt{\dt}}\exp\left(  \frac{-\gamma}{2\dt}\right).$$
\end{lem}

We end this section with another key preliminary result, which is the convergence rate of order 1 for the {\it corrected} local error.  Although the classical local error is of order 1/2, as stated in  Lemma~\ref{lem:localErrorOfTheScheme}, the local error seen by the diffusion coefficient function, corrected with the Milstein term stays  of order 1. 
\begin{lem}[Corrected local error process]\label{lem:CorrectedLocalError}
Let us fix $p\geq1$, and $\alpha\in[\tfrac{1}{2},1)$. For $\alpha>\tfrac{1}{2}$, assume $b(0)>2\alpha(1-\alpha)^2\sigma^2$, whereas for  $\alpha = \tfrac{1}{2}$,  assume  ${b(0)}>3(2p+1)\sigma^2/2.$ 
Then, there exists $C>0$, depending on the parameters of the model but not in $\dt$, such that for all $\dt\leq~\dmax(\alpha)$, the Corrected Local Error satisfies
$$\sup_{0\leq t\leq T}\E\left[ \left|  \sigma \X_t^\alpha-\sigma \Xe{t}^\alpha-\alpha\sigma^2\Xe{t}^{2\alpha-1}(W_t-W_{\eta(t)}) \right|^{2p}\right]\leq C\dt^{2p}.$$
\end{lem}

\section{Proof of the main Theorem \ref{mainTheorem}}\label{sec:theProof}

The proof of Theorem \ref{mainTheorem} is built in several steps. First, we work with the $L^{2p}(\Omega)$-norm of the error, for $p\geq1$, then at the last step of the proof we go back to the $L^p(\Omega)$-norm for $p\geq1$.

In what follows we denote 
$$\e_t:= \X_t-X_t$$
and 
$$\Sigma_t := \sgn(\Z{t})\left[\sigma\Xe{t}^\alpha+\alpha\sigma^2\Xe{t}^{2\alpha-1}(W_t-W_{\eta(t)})\right]-\sigma X_t^\alpha$$
so that 
$$d\e_t  = \left(\sgn(\Z{t})b(\Xe{t})-b(X_t)\right)dt + dL_t^0(\X)+ \Sigma_t dW_t.$$
Also, to make the notation lighter, we will denote the  Corrected Local Error by
$$D_t(\X) := \sigma \X_t^\alpha-\sigma \Xe{t}^\alpha-\alpha\sigma^2\Xe{t}^{2\alpha-1}(W_t-W_{\eta(t)}).$$

\subsection{The Weighted Error}

Before to prove the main theorem, we  establish in the Proposition \ref{LemmaScaledError} the convergence of a weighted  error. For $p\geq1$, let us consider $(\beta_t,0\leq t\leq T)$, defined by 
\begin{align}\label{def:beta}
\bb_t = 2p\|b'\|_\infty+2p(4p-1)+\frac{8\alpha^2p(4p-1)\sigma^2}{X_t^{2(1-\alpha)}},
\end{align}
and the Weight Process $\GG$ defined by
\begin{equation}\label{defScaleProcessTilde}
\Gamma_t = \exp\left( - \int_{0}^t{\bb_s ds} \right).
\end{equation}

The  Weight Process is adapted, almost surely positive, and bounded by $1$. Its paths are non increasing and hence has bounded variation, and  also satisfies 
$$d\Gamma_t = -\bb_t\Gamma_tdt.$$ 

The process $\left(\Gamma_t\right)_{t\geq0}$  can be seen as an integrating factor in the sense of linear first order ODE (see for example \cite{Simmons:1972aa}). When we apply the It\^o's Lemma to  $\Gamma^{ 2}_t\e^{4p}_t$, instead of $\e^{4p}_t$ alone, we can remove a very annoying term that appears in the righthand side bound (see the proof of Lemma \ref{lem:firstStepInScaledError}).  The exponential weight  in a $L^p(\Omega)$-norm is a useful tool to obtain a priori bound (as an example, for the existence and uniqueness of the solution of a Backward SDE, it is introduced a norm with exponential weight, such that the operator associated  to the BSDE is contractive under this new norm (see proof of Theorem 1.2 in \cite{Pardoux:1998aa}). 
In the same way, here we introduce this exponential weight to get the following a priori error-bound, which will allow us to prove Theorem \ref{mainTheorem}.

\begin{proposition}[Weighted Error]\label{LemmaScaledError} 
Under the hypotheses of Theorem \ref{mainTheorem}, for $p\geq1$ and $\alpha\in[\tfrac{1}{2},1)$, there exists a constant $C$ not depending on $\dt$ such that for all $\dt\leq \dmax(\alpha)$
\begin{equation}\label{eqLemmaScaledError}
\sup_{0\leq t\leq T}\E\left[ \Gamma^{2}_t\e^{4p}_t \right]\leq C\dt^{4p}.
\end{equation}
\end{proposition}

\begin{rem}
\item{(i)} \label{rem:hypothesisForB0norm2p}
Since we are going to work first with the $L^{2p}(\Omega)$-norm of the error,  when $\alpha = \tfrac{1}{2}$, Hypothesis \ref{hypo:2}-$(i)$ becomes
$$b(0)>\frac{3(2[2p\lor 2]+1)\sigma^2}{2}=\frac{3(4p+1)\sigma^2}{2},$$
in particular
\begin{equation}\label{hypothesisForB0norm2p}
\left( \frac{2b(0)}{\sigma^2}-1\right)> 12p+2.
\end{equation}
\item{(ii)} \label{rem:integrabilityOfBeta} From Lemma \ref{lem:LemmaMomentsOfX}, the process $\bb$ has polynomial moments of any order for $\alpha>\tfrac{1}{2}$, and when $\alpha=\tfrac{1}{2}$,  there exists $C$ such that $\E[\bb_t^q]  <C,$ for all $1<q<2b(0)/\sigma^2-1$. From the previous point in this Remark, it follows that the process $\bb$ has moments at least up to order $12p+2$.

\item{(iii)} From the definition $\beta$ in \eqref{def:beta}, and due to Lemma \ref{LemmaExponentialMomentOfX}, there exists a constant $C$ such that $\E [ \Gamma_T^{-q}]<C$ for all $q>0$ when $\alpha>\tfrac{1}{2}$, whereas for $\alpha=\tfrac{1}{2}$,  the $q$-th negative moment of the weight process is finite, as soon as
$$ 2p(4p-1)\sigma^2 \, q \leq \frac{\sigma^2}{8}\left( \frac{2b(0)}{\sigma^2}-1\right)^2. $$
Notice that, thanks to point $(i)$ in this Remark, a sufficient condition such that this last inequality holds is 
$$ 16p(4p-1)\, q \leq \left(12p+2\right)^2. $$
\end{rem}

We cut the proof of Proposition \ref{LemmaScaledError}  in two technical lemmas.
\begin{lem}\label{lem:firstStepInScaledError}
Under the hypotheses of Theorem \ref{mainTheorem},  for $p\geq1$ and $\alpha\in[\tfrac{1}{2},1)$, there exists a constant $C$ not depending on $\dt$ such that for all $\dt\leq \dmax(\alpha)$
\begin{align}\label{eq:firstStepInScaledError}
 \E[ \Gamma^{2}_t\e^{4p}_t] \leq 4p \int_{0}^t{\E\left[ \Gamma_s^{2}\e^{4p-1}_s\left(b(X_{\eta(s)})-b(X_s)\right)\right]ds} 
+4p\|b'\|_\infty \int_{0}^t{\sup_{0\leq u\leq s}\E[\Gamma_{u}^{2}\e_{u}^{4p}] ds}  + C\dt^{4p}.
\end{align}
\end{lem}

\begin{lem}\label{lem:driftScaledError}
Under the hypotheses of Theorem \ref{mainTheorem},  for $p\geq1$ and $\alpha\in[\tfrac{1}{2},1)$, there exists a constant $C$ not depending on $\dt$ such that for all $\dt\leq \dmax(\alpha)$, and for any $s\in[0,T]$
\begin{equation}\label{eq:boundDriftTerm}
| \E[\Gamma^{2}_s\e^{4p-1}_s (b(X_{\eta(s)})-b(X_s))] | \leq C \sup_{u\leq s}{\E [\Gamma^{2}_u\e^{4p}_u] } + C\dt^{4p}. 
\end{equation}

\end{lem}
As we will see soon, we prove \eqref{eq:boundDriftTerm} with the help of the It\^o's formula applied to $b$, and here is where we need $b$ of class $\mathcal{C}^2$ required in Hypotheses \ref{hypo:2}-$(ii)$.

\begin{proof}[Proof of Proposition \ref{LemmaScaledError}]
Thank to Lemmas \ref{lem:firstStepInScaledError} and \ref{lem:driftScaledError}, we have 
$$\E [\Gamma^{2}_t\e^{4p}_t] \leq  C\int_{0}^t{\sup_{u\leq s}{\E[\Gamma^{2}_u\e^{4p}_u]} ds}  + C\dt^{4p}. $$
and since the right-hand side is increasing, it follows
$$\sup_{s\leq t}\E[ \Gamma^{2}_s\e^{4p}_s ] \leq C\int_{0}^t{\sup_{u\leq s}{\E[\Gamma^{2}_u\e^{4p}_u]}     ds} +C\dt^{4p}, $$
from where we conclude the result thanks to Gronwall's Inequality.
\end{proof}

Now we present the proof of the technical lemmas.

\begin{proof}[Proof of Lemma \ref{lem:firstStepInScaledError}]
By the integration by parts formula, 
\begin{align*}
\E[\Gamma^{2}_t\e^{4p}_t]&=\;
4p\E\left[ \int_{0}^t{\Gamma_s^{2}\e^{4p-1}_s\left\{\sgn(\Z{s})b(\Xe{s})-b(X_s)\right\}ds}\right]  \\
&\quad+2p(4p-1)\E\left[\int_{0}^t{\Gamma_s^{2}\e^{4p-2}_s \Sigma_s^2ds}\right]\\
&\quad+ 4p\E\left[ \int_{0}^t{\Gamma_s^{2}\e^{4p-1}_s dL_s^0(\X) }\right]-\E\left[ \int_{0}^t{2\bb_s\Gamma_s^{2}\e^{4p}_sds}\right].
\end{align*}
Thanks to Lemma \ref{lem:Control2ndMomentLocalTime} and the control in the moments of the exact process in Lemma \ref{lem:LemmaMomentsOfX} we have
\begin{align*}
\E\left[ \int_{0}^t{\Gamma_s^{2}\e_s^{4p-1}dL_s^0(\X)}\right] \leq \E\left[ \int_{0}^t{|X_s^{4p-1}|dL_s^0(\X)}\right]
\leq 
\sqrt{\E\left[ \sup_{0\leq s\leq T} X_s^{8p-2}\right] \E\left[ L_T^0(\X)^2\right]  }\leq C\dt^{4p}.
\end{align*}
On the other hand, with $\sgn(x)=1-2\ind{\{x<0\}}$,  calling $\dW_s=W_s-W_{\eta(s)}$, we get for all $0\leq s\leq t$
\begin{equation*}
 \Sigma_s^2 \leq 2\left[\sigma X_s^\alpha - \sigma\X_s^{\alpha}\right]^2 
+ 2\left[\sigma\X_s^{\alpha} - \sigma\Xe{s}^\alpha-\alpha\sigma^2\Xe{s}^{2\alpha-1}\dW_s\right]^2
+ R_s^{\Sigma}\ind{ \{\Z{s}<0\}},
\end{equation*}
where we put aside all the terms multiplied by $\ind{\{\Z{s}<0\}}$ in 
\begin{align*}
R_s^{\Sigma} := & 4\left[\sigma\Xe{s}^\alpha+\alpha\sigma^2\Xe{s}^{2\alpha-1}\dW_s\right]\\
& \times \left\{ \left[\sigma\Xe{s}^\alpha+\alpha\sigma^2\Xe{s}^{2\alpha-1}\dW_s\right] +\left[\sigma X_s^\alpha - \sigma\X_s^{\alpha}\right]
+\left[\sigma\X_s^{\alpha} - \sigma\Xe{s}^\alpha-\alpha\sigma^2\Xe{s}^{2\alpha-1}\dW_s\right]\right\}.
\end{align*}
So, from the previous computations, the Lipschitz property of $b$, and Young's Inequality, we conclude
\begin{align*}
\E[ \Gamma^{2}_t\e^{4p}_t]&\leq
4p \int_{0}^t{\E\left[\Gamma_s^{2}\e^{4p-1}_s\left(b(X_{\eta(s)})-b(X_s)\right)\right] ds}  \\
&\quad{+4p\|b'\|_\infty \int_{0}^t{\E\left[\Gamma_{\eta(s)}^{2}\e_{\eta(s)}^{4p}\right]ds}+4p\|b'\|_\infty \int_{0}^t{\E[\Gamma_s^{2}\e^{4p}_s]ds}} \\
&\quad{+{4}p(4p-1)\E\left[ \int_{0}^t{\Gamma_s^{2}\e^{4p-2}_s\left(\sigma X_s^\alpha - \sigma\X_s^{\alpha}\right)^2 ds}\right]}\\
&\quad{+{(4p-2)}(4p-1)\E\left[ \int_{0}^t{\Gamma_s^{2}\e^{4p}_sds}\right]}\\
&\quad{+{2}(4p-1)\int_{0}^t{\E[ D_s(\X)^{4p}]ds}}\\
&\quad -\E\left[\int_{0}^t{{2}\bb_s\Gamma_s^{2}\e^{4p}_s ds}\right]+ \int_{0}^t{\E\left[ R_s\ind{\{\Z{s}<0\}}\right]ds}  + C\dt^{4p}.
\end{align*}
where $R_s = 4p(4p-1)\e_s^{4p-2}R_s^{\Sigma} + 8p\e_s^{4p-1}b(\Xe{s})$, and from Lemma  \ref{lem:ZbarRemainsPositive}  we have
$$\E\left[ R_s\ind{\{\Z{s}<0\}}\right]\leq C\dt^{4p}.$$
Since $\dt\leq \dmax(\alpha)$ and Remark \ref{rem:hypothesisForB0norm2p}-$(i)$, we can apply Lemma \ref{lem:CorrectedLocalError} so
$\E [ D_s(\X)^{4p}] \leq C\dt^{4p}$. 
Introducing these estimations in the previous computations, we have
\begin{align*}
\E[ \Gamma^{2}_t\e^{4p}_t ]\leq&
4p \int_{0}^t{\E\left[\Gamma_s^{2}\e^{4p-1}_s\left(b(X_{\eta(s)})-b(X_s)\right)  \right]ds}  \\
&+4p\|b'\|_\infty \int_{0}^t{\E\left[\Gamma_{\eta(s)}^{2}\e_{\eta(s)}^{4p}\right]ds} \\
&+\left(4p\|b'\|_\infty +(4p-2)(4p-1)\right)\E\left[ \int_{0}^t{\Gamma_s^{ 2}\e^{4p}_sds}\right]\\
&+4p(4p-1)\E\left[ \int_{0}^t{\Gamma_s^{2}\e^{4p-2}_s\left(\sigma X_s^\alpha - \sigma\X_s^{\alpha}\right)^2 ds}\right]\\ 
&-\E\left[ \int_{0}^t{2\bb_s\Gamma_s^{ 2}\e^{4p}_sds}\right]
  + C\dt^{4p}.
\end{align*}
Now we use the particular form of the weight process. Since for all $\tfrac{1}{2}\leq \alpha\leq 1$,  for all $x\geq0$, $y\geq0$, 
\begin{equation}\label{theAlgebraicTrick}
|x^\alpha-y^\alpha|(x^{1-\alpha}+y^{1-\alpha})\leq 2\alpha|x-y|,
\end{equation}
we have 
$$\E\int_{0}^t{\Gamma_s^{2}\e^{4p-2}_s\left(\sigma X_s^\alpha - \sigma\X_s^{\alpha}\right)^2 ds}\leq \E \int_{0}^t{\Gamma_s^{2}\e^{4p}_s \frac{4\alpha^2\sigma^2}{ X_s^{2(1-\alpha)}} ds},$$
and then, from the definition of $\bb$ in \eqref{def:beta}, we conclude 
\begin{align*}
\E[\Gamma^{2}_t\e^{4p}_t] &\leq
4p \int_{0}^t{\E\left[\Gamma_s^{2}\e^{4p-1}_s\left(b(X_{\eta(s)})-b(X_s)\right)\right]ds} +4p\|b'\|_\infty \int_{0}^t{\E\left[\Gamma_{\eta(s)}^{2}\e_{\eta(s)}^{4p}\right]ds}  + C\dt^{4p}.
\end{align*}
from where 
\begin{align*}
 \E[ \Gamma^{2}_t\e^{4p}_t] &\leq 4p \int_{0}^t{\E\left[\Gamma_s^{ 2}\e^{4p-1}_s\left(b(X_{\eta(s)})-b(X_s)\right)\right]ds} +4p\|b'\|_\infty \int_{0}^t{\sup_{0\leq u\leq s}\E[\Gamma_{u}^{ 2}\e_{u}^{4p}] ds}  + C\dt^{4p}.
\end{align*}
\end{proof}

\begin{proof}[Proof of Lemma \ref{lem:driftScaledError}]
By integration by parts 
\begin{equation}\label{theBadGuy}
\begin{aligned}
\E\left[ \Gamma^{2}_s\e^{4p-1}_s (b(X_{\eta(s)})-b(X_s))\right] &={-\E\int_{\eta(s)}^s {{2}\Gamma_u\e^{4p-1}_u (b(X_{\eta(s)})-b(X_u)) \bb_u\Gamma_udu} }\\
&\quad +\E\int_{\eta(s)}^s {\Gamma^{2}_ud\left(\e^{4p-1}_u (b(X_{\eta(s)})-b(X_u)) \right)}.   
\end{aligned}
\end{equation}
Applying H\"older's Inequality to the first term in the right-hand side we have
\begin{align*}
\left| \E\int_{\eta(s)}^s {\Gamma^{2}_u\e^{4p-1}_u[b(X_{\eta(s)})-b(X_u)]\bb_udu} \right| 
 &\leq\int_{\eta(s)}^s  \left(\E\left[\Gamma^{2}_u\e^{4p}_u\right] \right)^{1-\frac{1}{4p}}\left(\E\left[ |b(X_{\eta(s)})-b(X_u)| ^{4p}\bb_u^{4p}\right] \right)^{\frac{1}{4p}}du. 
\end{align*}
Recalling the Remark \ref{rem:integrabilityOfBeta}-$(ii)$, we have that $\E[\bb_u^{8p}]$ is finite, so applying Lemma \ref{lem:localErrorOfTheScheme}, 
\begin{align*}
\E\left[ |b(X_{\eta(s)})-b(X_u)|^{4p}\bb_u^{4p}\right]&
\leq \sqrt{\E\left[|b(X_{\eta(s)})-b(X_u)|^{8p}\right]{\E[\bb_u^{8p}] } }\leq C\dt^{2p}.
\end{align*}
Then, 
$$\left| \E\int_{\eta(s)}^s {\Gamma^{2}_u\e^{4p-1}_u[b(X_{\eta(s)})-b(X_u)]\bb_udu} \right| \leq C\left(\sup_{u\leq s}{\E[\Gamma^{2}_u\e^{4p}_u] }\right)^{1-\frac{1}{4p}}\dt^{3/2}. $$
Applying the It\^o's Formula to the second term in the right-hand  side of \eqref{theBadGuy}, and taking expectation we get
\begin{equation}\label{eq:ItoForb}
\begin{split}
\E\int_{\eta(s)}^s \Gamma^{2}_u&d\Big(\e^{4p-1}_u[b(X_{\eta(s)})-b(X_u)]\Big)\\
& = {-  \sigma\E\int_{\eta(s)}^s {\Gamma^{2}_u\e^{4p-1}_u\left(b'(X_u)b(X_u)+ \frac{\sigma^2}{2}b''(X_u) X^{2\alpha}_u  \right)du}}\\
&\quad{+ (4p-1)\E\int_{\eta(s)}^s {\Gamma^{2}_u\e^{4p-2}_u (b(X_{\eta(s)})-b(X_u)) \left\{\sgn(\Z{u})b(\Xe{s})-b(X_u)\right\}du}}\\
 &\quad{+\frac{\sigma^2}{2}(4p-1)(4p-2)\E\int_{\eta(s)}^s {\Gamma^{2}_u\e^{4p-3}_u (b(X_{\eta(s)})-b(X_u))  \Sigma_u^2du}}\\
 &\quad{- (4p-1)\sigma^2\E\int_{\eta(s)}^s {\Gamma^{2}_u\e^{4p-2}_ub'(X_u)X_u^\alpha \Sigma_u du}}\\
&\quad{+\frac{(4p-1)}{2}\E\int_{\eta(s)}^s {\Gamma^{2}_u\e^{4p-2}_u (b(X_{\eta(s)})-b(X_u)) dL_u^0(\X)}}\\
&=: I_1+I_2+I_3 + I_4+I_5.
\end{split}
\end{equation}
By the finiteness of the moment of $X$, the linear growth of $b$, and the polynomial growth of $b''$, applying Holder's inequality, we have
\begin{align*}
|I_1| 
&\leq C\int_{\eta(s)}^s {\left(\E[\Gamma^{2}_u\e^{4p}_u]\right)^{1-\frac{1}{4p}}
\left(\E\left[|b'(X_u)b(X_u)+ \frac{\sigma^2}{2}b''(X_u) X^{2\alpha}_u|^{4p}\right]\right)^{\frac{1}{4p}}du}\\
&\leq  C\left(\sup_{u\leq s}{\E[\Gamma^{2}_u\e^{4p}_u] }\right)^{1-\frac{1}{4p}}\dt. 
\end{align*}
For the bound of $I_2$, since $b$ is Lipschitz, and $\sgn(x)=1-2\ind{\{x<0\}}$,  we have
\begin{align*}
|I_2| \leq & C\int_{\eta(s)}^s {\E\left[ \Gamma^{2}_u\e^{4p-2}_u|X_{\eta(s)}-X_u||\Xe{s}-\X_{u}|\right]du}\\
& +C\int_{\eta(s)}^s {\E\left[ \Gamma^{2}_u\e^{4p-1}_u|X_{\eta(s)}-X_u|\right]du} + \int_{\eta(s)}^s{\E | R^{(2)}_u\ind{\{\Z{u}<0\}}| du} \\
\leq &  C\left(\sup_{u\leq s}{\E[\Gamma^{2}_u\e^{4p}_u]}\right)^{1-\frac{1}{2p}}\dt^2 +  C\left(\sup_{u\leq s}{\E[\Gamma^{2}_u\e^{4p}_u]}\right)^{1-\frac{1}{4p}}\dt^{3/2} + C\dt^{4p}
\end{align*}
Where again, all the terms multiplied by $\ind{\{\Z{u}<0\}}$ are putted in the rest $R^{(2)}_u$, and the expectation of the product is bounded with Lemma \ref{lem:ZbarRemainsPositive}.

In a similar way for the bound of $I_3$, decomposing $\Sigma_u$ with $\sgn(x)=1-2\ind{\{x<0\}}$,
\begin{align*}
|I_3| \leq &C\int_{\eta(s)}^s {\E\left[ \Gamma^{2}_u\e^{4p-3}_u|X_{\eta(s)}-X_u|D_s(\X)^2\right]du}\\
&+ C\int_{\eta(s)}^s {\E\left[ \Gamma^{2}_u\e^{4p-3}_u|X_{\eta(s)}-X_u|\left(\sigma\X_u-\sigma X_u^\alpha \right) ^2\right] du} +  \int_{\eta(s)}^s{\E | R^{(3)}_u\ind{\{\Z{u}<0\}} | du}.
\end{align*}
For the first term in the right-hand side we have
\begin{align*}
\E\left[ \Gamma^{2}_u\e^{4p-3}_u|X_{\eta(s)}-X_u|D_s(\X)^2\right]\leq &
\left( \E[\Gamma^{2}_u\e^{4p}_u]\right)^{1-\frac{3}{4p}}\left( \E[|X_{\eta(s)}-X_u|^{4p}]\right)^{\frac{1}{4p}} 
\left( \E[D_s(\X)^{4p}]\right)^{\frac{1}{2p}}\\
\leq & \left( \sup_{u\leq s}\E[\Gamma^{2}_u\e^{4p}_u] \right)^{1-\frac{3}{4p}}\dt^{5/2},
\end{align*}
due to the bound for the increments of the exact process and Lemma \ref{lem:CorrectedLocalError}. For the second term,  applying  \eqref{theAlgebraicTrick}, and noting that from Remark \ref{rem:integrabilityOfBeta}-$(ii)$, the exact process has negative moments up to of order $12 p+2$
\begin{align*}
&\E[ \Gamma^{2}_u\e^{4p-3}_u |X_{\eta(s)}-X_u| (\sigma\X_u-\sigma X_u^\alpha) ^2]\\
&\leq C\E[\Gamma^{2}_u\e^{4p-1}_u|X_{\eta(s)}-X_u|\frac{1}{X_u^{2(1-\alpha)}}]\\
&\leq C\left(\E [ \Gamma^{2}_u\e^{4p}_u]\right)^{1-\frac{1}{4p}}\left( \E[|X_{\eta(s)}-X_u|^{8p}]\right)^{1/8p}\left(\E\left[\frac{1}{X_u^{16(1-\alpha)p}}\right]\right)^{1/8p}\\
&\leq  \left(\sup_{u\leq s}\E[\Gamma^{2}_u\e^{4p}_u]\right)^{1-\frac{1}{4p}}\dt^{{1}/{2}}.
\end{align*}
We control the third term in the right-hand side in the bound for $|I_3|$ using  again Lemma \ref{lem:ZbarRemainsPositive}, so
\begin{align*}
|I_3| &\leq C \left(\sup_{u\leq s}\E[ \Gamma^{2}_u\e^{4p}_u]\right)^{1-\frac{3}{4p}}\dt^{7/2}+  \left(\sup_{u\leq s}\E[ \Gamma^{2}_u\e^{4p}_u]\right)^{1-\frac{1}{4p}}\dt^{3/2}+C\dt^{4p}.
\end{align*}

Now we bound $|I_4|$. 
\begin{align*}
| I_4|&\leq C\int_{\eta(s)}^s \E | \Gamma^{2}_u\e^{4p-2}_u D_u(\X) b'(X_u)X_u^\alpha  | du \\
&\quad+ C\int_{\eta(s)}^s {\E | \Gamma^{2}_u\e^{4p-2}_u (\sigma\X_u-\sigma X_u^\alpha) b'(X_u)X_u^\alpha | du}+  \int_{\eta(s)}^s{\E | R^{(4)}_u\ind{\{\Z{u}<0\}} | du}.
\end{align*}
We control the first term in the right-hand side using Hölder's inequality, Lemma \ref{lem:CorrectedLocalError} and the control in the moments of the exact process for all $0\leq u\leq s$
\begin{align*}
 \E|\Gamma^{2}_u\e^{4p-2}_u D_u(\X) b'(X_u)X_u^\alpha| &\leq
 (\E[\Gamma^{2}_u\e^{4p}_u] )^{1-\frac{1}{2p}} (\E[D_s(\X)^{4p}])^{\frac{1}{4p}} (\E[ b'(X_u)^{4p}X_u^{4p\alpha}] )^{\frac{1}{4p}}\\
 \leq& C \left(\sup_{u\leq s}[\Gamma^{2}_u\e^{4p}_u] \right)^{1-\frac{1}{2p}}\dt.
\end{align*}
For the second term in the right-hand side of the bound for $|I_4|$, we use one more time \eqref{theAlgebraicTrick}, and  the existence of negative moments of the exact process $X$, and then 
\begin{align*}
\E | \Gamma^{2}_u\e^{4p-2}_u\left[\sigma\X_u-\sigma X_u^\alpha \right]b'(X_u)X_u^\alpha|   &\leq
C\E[\Gamma^{2}_u\e^{4p-1}_u\frac{1}{X_u^{(1-\alpha)}}b'(X_u)X_u^\alpha] \leq C \left( \sup_{u\leq s}\E[ \Gamma^{2}_u\e^{4p}_u] \right)^{1-\frac{1}{4p}}.
\end{align*}
To control the third term in the right-hand side of the bound for $|I_4|$ we use Lemma \ref{lem:ZbarRemainsPositive} just as before. So
$$|I_4| \leq  C \left(\sup_{u\leq s}\E[ \Gamma^{2}_u\e^{4p}_u]\right)^{1-\frac{1}{2p}}\dt^2 + C \left(\sup_{u\leq s}\E[ \Gamma^{2}_u\e^{4p}_u]\right)^{1-\frac{1}{4p}}\dt + C\dt^{4p}.$$
Finally, 
\begin{align*}
|I_5|
&=\frac{(4p-1)}{2}\E\int_{\eta(s)}^s {\Gamma^{2}_uX^{4p-2}_u |b(X_{\eta(s)})-b(X_u)| dL_u^0(\X)}\\
&\leq C\E\left[ \sup_{u\leq s}\left[1+X_u^{4p-1}\right] L_T^0(\X)\right]\\
&\leq C\sqrt{ \E\left[ \sup_{u\leq s}\left[1+X_u^{4p-1}\right]^2\right]} \sqrt{ \E [L_T^0(\X)^2]} \leq C\dt^{4p}, 
\end{align*}
the last inequality comes from Lemmas \ref{lem:LemmaMomentsOfX} and \ref{lem:Control2ndMomentLocalTime}.

Putting all the last calculations in \eqref{theBadGuy} we obtain 
\begin{align*}
\left| \E[ \Gamma^{2}_s\e^{4p-1}_s (b(X_{\eta(s)})-b(X_s))] \right| &\leq C\left(\sup_{u\leq s}{\E[\Gamma^{ 2}_u\e^{4p}_u] }\right) ^{1-\frac{1}{4p}}\dt^{3/2} + C\left( \sup_{u\leq s}{\E [\Gamma^{ 2}_u\e^{4p}_u] }\right)^{1-\frac{1}{4p}}\dt  \\
&\quad+C\left(\sup_{u\leq s}{\E [\Gamma^{ 2}_u\e^{4p}_u] }\right)^{1-\frac{1}{2p}}\dt^2 +C \left(\sup_{u\leq s}\E[\Gamma^{ 2}_u\e^{4p}_u] \right)^{1-\frac{3}{4p}}\dt^{7/2}\\
&\quad+ C\dt^{4p}.
\end{align*}
Applying Young's Inequality in all terms in the right, we get the desired inequality  \eqref{eq:boundDriftTerm}.  
\end{proof}

\begin{rem}\label{rem:ItoForb}
Proceeding as in the Proof of Lemma \ref{lem:driftScaledError}, if $p=1$ or $p\geq3/2$, is not difficult to prove
\begin{equation}\label{eq:boundForDriftError}
\left| \E\left[ \e^{2p-1}_s (b(X_{\eta(s)})-b(X_s)) \right] \right| \leq C \sup_{u\leq s}{\E[\e^{2p}_u] }+ C\dt^{2p}.
\end{equation}
Notice that \eqref{eq:boundForDriftError} is  similar to \eqref{eq:boundDriftTerm} with the process $\g=1$. 
The restriction on  $p$  comes from the following observation.  Applying the It\^o's Lemma to the function $(x,y) \mapsto x^{2p-1}y$ for $p\geq 1$, with the couple of processes $(\e_\cdot,b(X_{\eta(\cdot)})-b(X_\cdot))$,  in the analogous of the identity  (3.9) with $\Gamma=1$, it will appear a term of the form
$$\E\int_{\eta(s)}^{s}{\e_u^{2p-3}[b(X_{\eta(s)})-b(X_u)]\Sigma_u^2du}.$$
The restriction $p\geq3/2$ avoids  the situation where ${2p-3}$ is negative.
\end{rem}

\subsection{Proof of  Theorem  \ref{mainTheorem} }

We start by controlling the $L^{2p}$-error. First, assume $p=1$ or $p\geq 3/2$. By the It\^o's formula we have
\begin{align*}
\E[ \e^{2p}_t]&=
2p\E\int_{0}^t{\e^{2p-1}_s\left\{\sgn(\Z{s})b(\Xe{s})-b(X_s)\right\}ds}+ 2p\E\int_{0}^t{\e^{2p-1}_sdL_s^0(\X)  } +p(2p-1)\E\int_{0}^t{\e^{2p-2}_s   \Sigma_s^2ds}.
\end{align*}
As we have seen before, $\E\int_{0}^t{ \e_s^{2p-1}dL_s^0(\X)} \leq C\dt^{2p},$ and  $\sgn(x)=1-2\ind{\{x<0\}}$, so
\begin{equation}\label{eq:fisrtStepError}
\begin{split}
 \E[ \e^{2p}_t]&\le
2p\E\int_{0}^t{ \e^{2p-1}_s\left[b(\Xe{s})-b(X_{\eta(s)})\right]ds}   \\
&\quad+2p\E\int_{0}^t{ \e^{2p-1}_s\left[b(X_{\eta(s)})-b(X_s)\right]ds}\\
&\quad+8p(2p-1)\E\int_{0}^t{ \e^{2p-2}_s\left[\sigma\X_s^\alpha-\sigma X_s^\alpha \right]^2ds}\\
&\quad+8p(2p-1)\E\int_{0}^t{ \e^{2p-2}_sD_s(\X)^2ds} + \E\int_{0}^t{R_s\ind{\{\Z{s}<0\}}ds} + C\dt^{2p},
 \end{split}
\end{equation}
where
$R_s = 4p(2p-1)\e_s^{2p-2}\left[ \sigma\Xe{s}^{\alpha}+ \alpha\sigma^2\Xe{s}^{2\alpha-1}\left( W_s - W_{\eta(s)} \right) \right] + 8p\e_s^{2p-1}b(\Xe{s}).$ 
If we use the Lipschitz property of $b$, and Young's inequality in the first term in the right of \eqref{eq:fisrtStepError}, Lemma \ref{lem:CorrectedLocalError} in the fourth one, and Lemma~\ref{lem:ZbarRemainsPositive}  in the fifth one, we have
\begin{equation}\label{eq:SecondStepError}
\begin{split}
 \E[ \e^{2p}_t]&\leq C\int_{0}^t{\sup_{u\leq s} \E[\e^{2p}_u] ds}   +2p\int_{0}^t{ \E[\e^{2p-1}_s (b(X_{\eta(s)})-b(X_s))]ds}  \\
&\quad+8p(2p-1)\int_{0}^t{ \E[\e^{2p-2}_s (\sigma \X_s^\alpha-\sigma X_s^\alpha)^2 ] ds} + C\dt^{2p}.
\end{split}
\end{equation}
And, according to Remark \ref{rem:ItoForb}, we have
$$\left| \E[ \e^{2p-1}_s (b(X_{\eta(s)})-b(X_s)) ] \right|\leq C \sup_{u\leq s}{\E[\e^{2p}_u] } + C\dt^{2p}.$$ 

On the other hand, using again \eqref{theAlgebraicTrick}, we have
$$\E [ \e^{2p-2}_s(\sigma \X_s^\alpha-\sigma X_s^\alpha)^2]\leq C\E[  \e^{2p}_s {X_s^{-2(1-\alpha)}} ] =C\E [ \Gamma_s \e^{2p}_s{X_s^{-2(1-\alpha)}}\Gamma_s^{-1}] ,$$
and applying Cauchy-Schwartz inequality,
\begin{align*}
\E[ \Gamma_s\, \e^{2p}_s\frac{1}{X_s^{2(1-\alpha)}}\Gamma_s^{-1}]
\leq&  (\E[ \Gamma_s^{2} \e^{4p}_s])^{\frac{1}{2}}
\left(\E[\frac{1}{X_s^{4(1-\alpha)}}\Gamma_s^{-2}]\right)^{\frac{1}{2}}.
\end{align*}
The first term in the right-hand side is the weight error controlled by Proposition \ref{LemmaScaledError}. To control the second one, let us recall Remark  \ref{rem:integrabilityOfBeta}. For $\alpha>\tfrac{1}{2}$, the exact process and the weight process $\g$ have negative moments of any order, therefore the second term in the last inequality is bounded by a constant. On the other hand, for $\alpha=\tfrac{1}{2}$ we need a finer analysis. From the second point in Remark  \ref{rem:integrabilityOfBeta}, the $12p+2$-th negative moment of the exact process is finite, and since 
$$16p(4p-1)\,2\frac{6p+1}{6p}  <(12p+2)^2,$$
according with the third point of Remark  \ref{rem:integrabilityOfBeta}, the $2{(6p+1)}/{6p}$-th negative moment of the  weight process $\g$ is also finite. Therefore, when $\alpha=\tfrac{1}{2}$, 
\begin{align*}
\E\left[\frac{1}{X_s^{4(1-\alpha)}}\Gamma_s^{-2}\right]&=\E\left[\frac{1}{X_s^2}\Gamma_s^{-2}\right]\leq \left( \E\left[\frac{1}{X_s^{12p+2}}\right]\right)^{\frac{1}{6p+1}}\left( \E\left[\Gamma_s^{-2\frac{6p+1}{6p}}\right]\right)^{\frac{6p}{6p+1}} \leq C,
\end{align*}
and then in any case
$$\E\left[ \Gamma_s \e^{2p}_s\frac{1}{X_s^{2(1-\alpha)}}\Gamma_s^{-1}\right] \leq C\dt^{2p}.$$
Introducing all the last computations in  \eqref{eq:SecondStepError}  we get
$$\E[  \e^{2p}_t] \leq   C\int_{0}^t \left( \sup_{u\leq s}\E\left[ \e^{2p}_u\right] \right) ds  + C\dt^{2p}. $$
Since the right-hand side is increasing, thanks to Gronwall's Inequality we have, for $p=1$ or $p\geq 3/2$, 
$$\sup_{0\leq t\leq T}\left( \E[ \e^{2p}_t]\right)^{\frac{1}{2p}} \leq C\dt.$$
We extend to $p\in(1,3/2)$, thanks to Jensen's inequality
$$ (\E[ \e^{2p}_t])^{\frac{1}{2p}} \leq  (\E[ \e^{\frac{6}{2}}_t])^{\frac{2}{6}}\leq C\dt,$$
and we conclude that $ (\E[ \e^{2p}_t])^{\frac{1}{2p}} \leq  C\dt$ for all $p\geq 1$ satisfying Remark \ref{rem:hypothesisForB0norm2p}-$(i)$. 

Now we control the $L^p$-error. For $\alpha =\tfrac{1}{2}$ and $p\geq2$, denoting $p'=\frac{p}{2}\geq1$, 
$$\sup_{0\leq t\leq T}\left( \E\left[ \e^{{p}}_t\right]\right)^{\frac{1}{{p}}} =\sup_{0\leq t\leq T}\left( \E[ \e^{2{p'}}_t]\right)^{\frac{1}{2{p'}}} \leq C\dt.$$
Hypothesis  \ref{hypo:2}-$(i)$  gives 
$$b(0)>\frac{3(2p+1)\sigma^2}{2}=\frac{3(4p'+1)\sigma^2}{2}.$$
Since \eqref{hypothesisForB0norm2p} in  Remark \ref{rem:hypothesisForB0norm2p} is satisfied, 
we can control the $L^{2p'}(\Omega)$-norm of the error and then
$$\sup_{0\leq t\leq T}\left( \E[ \e^{{p}}_t]\right)^{\frac{1}{{p}}} =\sup_{0\leq t\leq T}\left( \E[ \e^{2{p'}}_t]\right)^{\frac{1}{2{p'}}} \leq C\dt.$$
If $p\in[1,2)$, Hypothesis  \ref{hypo:2}-$(i)$ is $b(0)>15\tfrac{\sigma^2}{2}$, which is enough to bound the $L^2(\Omega)$-norm of the error, and then from Jensen's inequality
$$\sup_{0\leq t\leq T}\left( \E[ \e^{{p}}]\right)^{\frac{1}{{p}}} \leq \sup_{0\leq t\leq T}\left( \E[ \e^{2}_t]\right)^{\frac{1}{2}} \leq C\dt.$$
The case $\alpha>1/2$ is easier. Since the Hypothesis in the parameters for this case does not depend on $p$, we can conclude for any $p\geq1$ from Jensen's inequality and the control for the $L^{2p}(\Omega)$-norm of the error.

\begin{rem}
Let us mention an example of  extension of our convergence result, based on simple transformation method:  consider the $3/2$-model, namely the solution of 
$$r_t = r_0+\int_{0}^{t}{c_1r_s(c_2-r_s)}ds + \int_{0}^{t}{c_3\,r_s^{3/2}dW_s}.$$
Applying the Itô's Formula to $v_t = f(r_t)$, with $f(x)=x^{-1}$, we have
$$v_t = v_0+\int_{0}^{t}\left(c_1+c_3^2-c_1c_2v_s\right)ds + \int_{0}^{t}{c_3\, v_s^{1/2}dB_s},$$
where $B_s=-W_s$ is a Brownian motion.  We can approximate $v$ with the SMS $\bar{v}$, and then define $\bar{r}_t := 1/\bar{v}_t$. Then we can deduce  the strong convergence with rate one of $\bar{r}_t$ to $r_t$ from our previous results.

Transformation methods can be used in a more exhaustive manner, in the context of CEV-like SDEs and we refer to \cite{Chassagneux:2015aa} for approximation results and examples, using this approach. 
\end{rem}

\subsection{Strong Convergence of the Projected Milstein Scheme}\label{sec:ConvergencePMS}
The Projected Milstein Scheme (PMS) is defined by 
\begin{align*}
\hX_{t_{k}} =
\Big(\hX_{t_{k-1}}+b(\hX_{t_{k-1}})\dt+\sigma\hX_{t_{k-1}}^\alpha(W_{t_k}-W_{t_{k-1}})+\dfrac{\alpha\sigma^2\hX_{t_{k-1}}^{2\alpha-1}}{2}\left[(W_{t_k}-W_{t_{k-1}})^2 -\dt\right]\Big)^+, ~\hX_0=x_0
\end{align*} 
where for all $x\in\R$, $(x)^+=\max(0,x)$. The continuous time version of the (PMS) is given by
\begin{equation}\label{eq:defPMS}
\hX_{t} =
\Big(\hX_{\eta(t)}+b(\hX_{\eta(t)})\dt+\sigma\hX_{\eta(t)}^\alpha(W_{t}-W_{\eta(t)})+\dfrac{\alpha\sigma^2\hX_{\eta(t)}^{2\alpha-1}}{2}\left[(W_{t}-W_{\eta(t)})^2 -\dt\right]\Big)^+.
\end{equation} 
Notice that for all $t\in[0,T]$, $0\leq \hX_t\leq\X_t$, then the positive moments of the PMS are bounded (see Lemma \ref{lem:finitenessOfTheMomentsOfXbar}). 

To obtain a strong convergence rate for the PMS, we first show that the PMS and the SMS coincide with a large probability. 
\begin{lem}\label{lem:PMSandSMSalmostTheSame} Let us consider the stopping time  $\tau= \inf\{s\geq0:\X_s\neq\hX_s\}$. 
Assume that $b(0)>2\alpha(1-\alpha)^2\sigma^2$ and $\dt\leq 1/(2K(\alpha))$. Then for any $p\geq1$,
$$\P\left( \tau \leq T\right) \leq C\dt^p.$$
\end{lem}

\begin{proof}   
Notice that $\tau$ is almost surely strictly positive because both schemes start from the same deterministic initial condition $x_0$. On the other hand
$$\P\left( \tau \leq T\right) =  \sum_{i=0}^{N-1}{\P\left( \tau \in (t_{k},t_{k+1}]\right)},$$
and according to Lemma \ref{lem:ZbarRemainsPositive}
\begin{align*}
\P\left( \tau \in [t_{i},t_{i+1})\right)&=\P\left( \inf_{t_k< s\leq t_{k+1}}\Z{s}\leq0, \X_{t_k} = \hX_{t_k}\right) \leq\P\left( \inf_{t_k< s\leq t_{k+1}}\Z{s}\leq0\right) \leq \exp\left(- \frac{\gamma}{\dt}\right).
\end{align*}
So, 
$$\P\left( \tau \leq T\right) \leq  \frac{T}{\dt}\exp\left(- \frac{\gamma}{\dt}\right).$$
Since for any $p\geq 1$, there exists a constant $C_p$ such that $\exp(-\gamma/\dt)/\dt\leq C_p\dt^{p}$, we have
$$\P\left( \tau \leq T\right) \leq C\dt^p.$$
\end{proof}

\begin{corollary} Assume Hypotheses \ref{hypothesisH0} and \ref{hypo:2}.  Consider a maximum step size $\dmax(\alpha)$ defined in \eqref{eq:def_dt0_alpha}. Let $\XX$ be the process defined on \eqref{exactProcess} and $\XXh$ the  Projecter Milstein scheme given in \eqref{eq:defPMS}. 
Then for any $p\geq1$ that allows Hypotheses \ref{hypo:2}, there exists  a constant $C$ depending on $p$, $T$, $b(0)$, $\alpha$, $\sigma$, $K$,  and $x_0$, but not on $\dt$, such that for all $\dt\leq\dmax(\alpha)$, 
\begin{align}\label{eq:convergence-PMS}
\sup_{0\leq t\leq T}\left( \E[|X_t-\hX_t|^{p}] \right)^{\frac{1}{p}}\leq C\dt.
\end{align}
\end{corollary}

\begin{proof} Notice that for all $t\in[0,T]$, with $\tau= \inf\{s\geq0:\X_s\neq\hX_s\}$, 
\begin{align*}
\E [ |\hX_t-X_t|^p] &= \E[|\hX_t-X_t|^p\ind{\{\tau\leq T\}}] +\E[ |\hX_t-X_t|^p\ind{\{\tau> T\}}]\\
& \quad \leq  \sqrt{\E[ |\hX_t-X_t|^{2p}] \P\left(\tau\leq T\right)} +\E[|\X_t-X_t|^p]\leq C\dt^p 
\end{align*}
where the last inequality comes from Lemma \ref{lem:PMSandSMSalmostTheSame} and Theorem \ref{mainTheorem}.
\end{proof}

\section{Numerical Experiments and Conclusion}\label{sec:numExp}

We start this section with the analysis  of two  numerical experiments. The first one aims to study empirically the strong rate of convergence of the SMS in comparison with other  schemes proposed the literature. The second one aims to study the impact of including the SMS in a Multilevel Monte Carlo application.

\subsection{Empirical study of the strong rate of converge}

In this experiment we compute the error of the  schemes as a function of the step size $\dt$  for different values of the parameters $\alpha$ and $\sigma$. 

For $\alpha>\tfrac{1}{2}$ we compare the SMS with the Symmetrized Euler Scheme (SES) introduced in \cite{BERKAOUI:2008fk},  and with the Balanced Milstein Scheme (BMS) presented in \cite{Kahl:2006ab}.  
Whereas for $\alpha = \tfrac{1}{2}$, in addition to the aforementioned schemes, we will also compare SMS with the Modified Euler Scheme (MES) proposed in \cite{Chassagneux:2015aa}, and with the Alfonsi Implicit Scheme (AIS) proposed in \cite{Alfonsi:2005aa}. 

Let us first, shortly review those different schemes. 

\paragraph{Alfonsi Implicit Scheme (AIS). } 
Proposed in \cite{Alfonsi:2005aa}, the AIS can be applied to equation \eqref{exactProcess} when the drift is a linear function. A priori, the AIS can be applied for $\alpha\in[\tfrac{1}{2},1)$, but is relevant to observe that only when $\alpha=\tfrac{1}{2}$, the AIS is in fact an explicit scheme (also know as drift-implicit square-root Euler approximations ) whereas in any other case is not. This implies that in order to compute the AIS for $\alpha>\tfrac{1}{2}$, at each time step it is necessary to solve numerically a non linear equation. This extra step in the implementation of the scheme brings questions about the impact of the error of this subroutine in the error of the scheme, and about the computing performance of the scheme. Since this questions are beyond the scope of the present work, we include the AIS in the comparison only in the Cox--Ingersol--Ross (CIR) case (linear drift and $\alpha = \tfrac{1}{2}$). In this context, the AIS can be use only if $\sigma^2>4b(0)$, for other values of the parameters the AIS is not defined. In terms of convergence, when $\alpha=\tfrac{1}{2}$, according to Theorem 2 in Alfonsi \cite{Alfonsi:2013aa}, the AIS converges in the  $L^p(\Omega)$-norm, for $p\geq1$, at rate $\dt$ when $(1\lor 3p/4)\sigma^2<b(0)$.
When $\alpha>\tfrac{1}{2}$, the AIS (see Section 3 of \cite{Alfonsi:2013aa}) converge as soon as $b(0)>0$, at rate $\dt$ to the exact solution.

\paragraph{Balanced Milstein Scheme (BMS). }
The BMS was introduced by  Kahl and Schurz in \cite{Kahl:2006ab}, and although its convergence it is not proven for Equation \eqref{exactProcess} (see Remark 5.12 in  \cite{Kahl:2006ab}),  numerical experiments shows a competitive behavior (see \cite{Kahl:2006aa}). Also, the BMS can be easily implemented for $\alpha\in[\tfrac{1}{2},1)$, so we decide to include it in our numerical comparison.  

\paragraph{Modified Euler Scheme (MES). }
Introduced in \cite{Chassagneux:2015aa}, the MES can be applied to  the Equation \eqref{exactProcess} for $\alpha\in[\tfrac{1}{2},1)$ when the drift has the form $b(x)=\mu_1(x)-\mu_2(x)x$ for $\mu_1$ and $\mu_2$ suitables functions. The rate of convergence in the $L^1(\Omega)$-norm of the MES depends on the parameters. For $\alpha =\tfrac{1}{2}$ the rate is 1 if $\sigma^2$ is big enough compared with $b(0)$, and it is ${\rho}<1$ in other case. When $\alpha>\tfrac{1}{2}$, the MES converges at rate 1 as soon as $b(0)>0$   (see Proposition 4.1 in \cite{Chassagneux:2015aa}).  When  $\alpha>\tfrac{1}{2}$, the  implementation of the MES requires some extra tuning which is not explicitly given in \cite{Chassagneux:2015aa}  (see Remark 5.1), so we implement the MES only for $\alpha=\tfrac{1}{2}$.

\paragraph{Symmetryzed Euler Scheme (SES). } The SES, introduced in \cite{BERKAOUI:2008fk}, is an explicit scheme which can be apply to  the equation \eqref{exactProcess} for $\alpha\in[\tfrac{1}{2},1)$ and any Lipschitz drift function $b$. It has the weakest hypothesis over $b$ of all the schemes discussed in this paper. If $\alpha=\tfrac{1}{2}$, according to Theorem 2.2 in \cite{BERKAOUI:2008fk}, the rate of convergence of the SES is $\sqrt{\dt}$ under suitable conditions for $b(0)$, $\sigma^2$ and $K$. When $\alpha>\tfrac{1}{2}$ the SES converge at rate $\sqrt{\dt}$ as soon as $b(0)>0$ (see Theorem 2.2 in \cite{BERKAOUI:2008fk}).

\medskip
We summarize the theoretical analysis of the schemes  above  in Table~\ref{table:ConditionsOverParametersDifferentSchemesAlpha1/2} for $\alpha=\tfrac{1}{2}$, and in Table~\ref{table:ConditionsOverParametersDifferentSchemes} for $\alpha>\tfrac{1}{2}$. 

\begin{table}[h!]
\begin{center}
\def\arraystretch{1.4}
\begin{tabular}{|c|c|c|c|c|}
\cline{1-5}
{Scheme} & {Norm}& {Drift} &{Convergence's Condition}& 
\begin{tabular}{c}
Theoretical \\[-0.2cm]
rate 
\end{tabular} \\ \hline 
\multirow{2}{* }{SMS}  &\multirow{2}{* }{$L^{p},p\geq1$} &$b$ Lipschitz, $b\in\mathcal{C}^2$  &  \multirow{2}{* }{$b(0)>3\left( 2[p\lor2]+1\right)\tfrac{\sigma^2}{2}$} & \multirow{2}{* }{1}   \\

&&$b''$ with polynomial growth && \\\cline{1-5}
\multicolumn{1}{ |c| }{AIS \cite{Alfonsi:2013aa}} & $L^{p},p\in[1,\frac{4b(0)}{3\sigma^2})$ & $b(x)=a-bx$ & $b(0)>(1\lor \frac{3}{4}p)\sigma^2$  & 1    \\ \cline{1-5}
\multicolumn{1}{ |c| }{BMS} &   &  &     & undetermined   \\ \cline{1-5}
\multirow{3}{* }{MES \cite{Chassagneux:2015aa}} & \multirow{3}{* }{$L^{1}$}&{$b(x)=\mu_1(x)-\mu_2(x)x$} & $b(0)>\frac{5\sigma^2}{2}$   & $1$    \\ \cline{4-5}
 &  &$\mu_i\in\mathcal{C}^2_b\cap\mathcal{C}^0_b,\;\mu_1\geq0$ & $b(0)>\frac{3\sigma^2}{2}$   & $\tfrac{1}{2}$  \\ \cline{4-5}
 &  &$\mu_1'\leq0,\;\mu_2'\geq0$& $b(0)>\sigma^2$    & $\left( \frac{1}{6}, \frac{1}{2}-\frac{\sigma^2}{2b(0)+\sigma^2}\right)$  \\ \hline 
{SES \cite{BERKAOUI:2008fk}}&{$L^{p},p\geq1$} &{$b$ Lipschitz}&  
\begin{tabular}{c}
$b(0) >\left[  \sqrt{\tfrac{8}{\sigma^2} \mathcal{K}(\tfrac{p}{2} \vee 1)} +1\right]\tfrac{\sigma^2}{2},$\\
$\mathcal{K}(q)=K(16q-1)$\\
$\;\;\;\;\;\;\;\;\;\;\;\;\;\;\;\;\lor 4\sigma^2(8p-1)^2$
\end{tabular}
& $\tfrac{1}{2}$  
\\ \cline{1-5}
\end{tabular}
\end{center}
\caption{Summary of the condition over the parameters for the convergence of the different schemes for $\alpha=\tfrac{1}{2}$.\label{table:ConditionsOverParametersDifferentSchemesAlpha1/2}}
\end{table}

\begin{table}[h!]
\begin{center}
\def\arraystretch{1.4}
\begin{tabular}{|c|c|c|c|c|}
\hline
{Scheme} & {Norm}& {Drift} &{Convergence's Condition}& 
\begin{tabular}{c}
Theoretical \\[-0.2cm]
rate 
\end{tabular} \\ \hline 
{SMS}  &{$L^{p},p\geq1$}  &
\begin{tabular}{c}
$b$ Lipschitz, $b\in\mathcal{C}^2$\\
$b''$ with polynomial growth
\end{tabular}
& 
{$b(0)>2\alpha(1-\alpha)^2\sigma^2$} &{1}
 \\ \cline{1-5}

\multicolumn{1}{ |c| }{AIS} & $L^{p},p\in[1,\frac{4b(0)}{3\sigma^2})$ & $b(x)=a-bx$ & $b(0)>0$  & 1    \\ \cline{1-5}
\multicolumn{1}{ |c| }{BMS} &    & &   & undetermined   \\ \cline{1-5}

\multirow{3}{* }{MES} & \multirow{3}{* }{$L^{1}$}&{$b(x)=\mu_1(x)-\mu_2(x)x$} & \multirow{3}{*}{$b(0)>0$}   & \multirow{3}{*}{$1$}    \\ 
 &  &$\mu_i\in\mathcal{C}^2_b\cap\mathcal{C}^0_b,\;\mu_1\geq0$ &    &   \\ 
 &  &$\mu_1'\leq0,\;\mu_2'\geq0$&   &   \\ \hline
{SES}&{$L^{p},p\geq1$} &{$b$ Lipschitz}&  $b(0) >0$&$\tfrac{1}{2}$ \\ \hline

\end{tabular}
\end{center}
\caption{Summary of the condition over the parameters for the convergence of the different schemes when $\alpha>\tfrac{1}{2}$. \label{table:ConditionsOverParametersDifferentSchemes}}
\end{table}

\subsubsection*{Simulation setup}

In our simulations we consider a time horizon $T=1$, and $x_0=1$. 
In order to include as many schemes as possible we consider for all simulations a linear drift 
$$b(x)=10-10x.$$ 
To measure the error of each scheme,  we estimate its $L^1(\Omega)$-norm for which a  theoretical rate is proposed for all the selected schemes. 

Let $\E| \e_T^{\text{\tiny{SMS}}}|$, $\E| \e_T^{\text{\tiny{BMS}}}|$, $\E| \e_T^{\text{\tiny{SES}}}|$, $\E| \e_T^{\text{\tiny{MES}}}|$, and $\E| \e_T^{\text{\tiny{AIS}}}|$ be the $L^1(\Omega)$-norm of the error for the SMS, BMS, SES, MES and AIS respectively.  To estimate these quantities, we consider as a reference solution the AIS approximation for $\dt=\dmax(\alpha)/2^{12}$ when $\alpha=\tfrac{1}{2}$, and the SMS for $\dt~ = ~\dmax(\alpha)/2^{12}$ when $\alpha>\tfrac{1}{2}$.
 Then for each 
 $$\dt \in \left\{\frac{\dmax(\alpha)}{2^{n}}, n=1,\ldots 9\right\},$$
we estimate  $\E| \e_T^{{\tiny{\cdot\cdot\cdot}}}|$ by computing $5\times 10^4$ trajectories of the corresponding scheme, and comparing them  with the reference solution. The results of these simulations are reported  in Figures \ref{figure:simulations} ($\alpha=\tfrac{1}{2}$) and \ref{figure:simulations2} ($\alpha>\tfrac{1}{2}$).  
The graphs  plot  the $\text{Log} \E| \e_T^{{\tiny{\cdot\cdot\cdot}}}|$ in terms of the $\text{Log}  \dt$,  and we have added the plot of  the identity map  to serve as reference for rate of order 1. The  schemes with a slope smaller than the slope of the reference line have an order of convergence smaller than one. To obtain a more quantitative comparison of the schemes, we also perform a regression analysis on the model
$$\log(\E| \e_T^{{\tiny{\cdot\cdot\cdot}}}|) = \rho \log(\dt)+ C.$$
Notice that $\hat{\rho},$ the estimated value for $\rho$, corresponds to the empirical rate of convergence of the different schemes. We present the result of this regression analysis in Tables \ref{table:NumericalResultsAlphaOneHalf} and \ref{table:NumericalResultsAlphaBiggerThanOneHalf}.

\paragraph{Empirical results for  $\alpha=\tfrac{1}{2}$. } Figure \ref{figure:simulations} and Table \ref{table:NumericalResultsAlphaOneHalf} present the result for the CIR case. From Table \ref{table:ConditionsOverParametersDifferentSchemesAlpha1/2},  we  observe that we can distinguish five cases for the parameters. 

The first case $(\sigma^2 =1)$ is such that  $b(0)>6\sigma^2$: the SMS, the AIS, and the MES have a theoretical rate of convergence equal to $\dt$, whereas the SES has a theoretical rate of convergence equal to $\sqrt{\dt}$. In Figure \ref{figure:simulations-a},  we observe that the graphs of the SMS, the AIS, the BMS, and the MES seem parallel to the reference line, which is expected, while the SES has a smaller slope. This is also confirms in the first line of the Table \ref{table:NumericalResultsAlphaOneHalf}, where we observe that the empirical rates of convergence are close to the theoretical ones. Notice that the BMS has a competitive empirical rate of convergence, although the theoretical one is not known.
\begin{table}[!h!]
\begin{center}
\bgroup
\def\arraystretch{2}
\begin{tabular}{|c|c|c|c|c|c|c|c|c|c|c|}
\cline{1-11}
\multirow{3}{* }{$\sigma^2$}  & \multicolumn{10}{ |c| }{Observed $L^1(\Omega)$ convergence rate $\hat{\rho}$   (and its $R^2$ value)} \\[-0.2cm] \cline{2-11} 
& \multicolumn{2}{|c|}{SMS}& \multicolumn{2}{|c|}{AIS} & \multicolumn{2}{|c|}{BMS} & \multicolumn{2}{|c|}{MES}  & \multicolumn{2}{|c|}{SES} \\[-0.2cm]  \cline{2-11}
&$\hat{\rho}$ 
&{\small \cellcolor{gray!20} ($R^2$)}&$\hat{\rho}$ 
&{\small \cellcolor{gray!20} ($R^2$)}&$\hat{\rho}$ 
&{\small \cellcolor{gray!20} ($R^2$)}&$\hat{\rho}$ 
&{\small \cellcolor{gray!20} ($R^2$)}&$\hat{\rho}$ 
&{\small \cellcolor{gray!20} ($R^2$)}\\[-0.1cm] \cline{1-11}
1&\cellcolor{red!20}0.9956&\small\cellcolor{gray!20}(99.9\%)
&\cellcolor{red!20}1.0060&\small\cellcolor{gray!20}(99.9\%)
&\cellcolor{red!20}1.0055&\small\cellcolor{gray!20}(99.9\%) 
&\cellcolor{red!20}0.9955&\small\cellcolor{gray!20}(99.9\%)
&\cellcolor{blue!25}0.5941&\small\cellcolor{gray!20}(99.3\%)\\ \cline{1-11}
4&\cellcolor{red!20}0.9969&\small\cellcolor{gray!20}(99.9\%)
&\cellcolor{red!20}1.0054&\small\cellcolor{gray!20}(99.9\%)
&\cellcolor{red!20}1.0037&\small\cellcolor{gray!20}(99.9\%)
&\cellcolor{red!20}0.9961&\small\cellcolor{gray!20}(99.9\%)
&\cellcolor{blue!25}0.5344&\small\cellcolor{gray!20}(99.8\%)\\ \cline{1-11}
6.25&\cellcolor{red!20}0.9976&\small\cellcolor{gray!20}(99.9\%)
&\cellcolor{red!20}1.0043&\small\cellcolor{gray!20}(99.9\%)
&\cellcolor{red!20}1.0002&\small\cellcolor{gray!20}(99.9\%)
&\cellcolor{red!20}0.9941&\small\cellcolor{gray!20}(99.9\%)
&\cellcolor{blue!25}0.5237&\small\cellcolor{gray!20}(99.9\%)\\ \cline{1-11}
9&\cellcolor{red!20}0.9984&\small\cellcolor{gray!20}(99.9\%)
&\cellcolor{red!20}1.0015&\small\cellcolor{gray!20}(99.9\%)
&\cellcolor{red!20}0.9859&\small\cellcolor{gray!20}(99.9\%)
&\cellcolor{red!20}0.7891&\small\cellcolor{gray!20}(99.9\%)
&\cellcolor{blue!25}0.5164&\small\cellcolor{gray!20}(99.9\%)\\ \cline{1-11}
36&\cellcolor{blue!25}0.6410&\small\cellcolor{gray!20}(99.7\%)
&\cellcolor{blue!25}0.6282&\small\cellcolor{gray!20}(99.8\%)
&\cellcolor{blue!25}0.4538&\small\cellcolor{gray!20}(99.3\%)
&\cellcolor{blue!25}0.3575&\small\cellcolor{gray!20}(99.4\%)
&\cellcolor{blue!25}0.4718&\small\cellcolor{gray!20}(99.9\%)\\ \cline{1-11}
\end{tabular}
\egroup
\end{center}
\caption{Empirical rate of convergence $\hat{\rho}$ for the $L^1(\Omega)$-error of the schemes when $\alpha=\tfrac{1}{2}$ for different values of $\sigma^2$. \label{table:NumericalResultsAlphaOneHalf}}
\end{table}

The second case $(\sigma^2=4)$ is such that $b(0)\in (5\frac{\sigma^2}{2},6\sigma^2)$, now only the AIS and the MES have a theoretical rate of convergence equal to $\dt$. However, how we can see in Figure \ref{figure:simulations-b} the SMS still shows a linear behavior in this case. Recall that the condition over the parameters is a sufficient condition and we believe that could be improved. Notice that also the BMS shows a linear behavior. In the second line of Table \ref{table:NumericalResultsAlphaOneHalf}, we can observe that empirical rates of convergence are close to one for  all the scheme but the SES.

In Figure \ref{figure:simulations-c}, we illustrate the third case 
$(\sigma^2=6.25)$ and then $b(0)\in (3\sigma^2/2,5\sigma^2/2)$. In this case, only the AIS has a theoretical rate of convergence equal to one. For the MES is $\sqrt{\dt}$, but in the graphics we still  observe a linear behavior for the MES, and also for the SMS and the BMS. This is confirm in the third line of Table \ref{table:NumericalResultsAlphaOneHalf}.

The fourth case is $(\sigma^2=9)$ and $b(0)\in (\sigma^2,3\sigma^2/2)$, which we display in Figure \ref{figure:simulations-d}. For this values of the parameters the theoretical rate of convergence is known only for the AIS and the MES. Nevertheless, we observe in the graphs and in the fourth line of Table \ref{table:NumericalResultsAlphaOneHalf}  that all the schemes seems to reach their optimal convergence rates. 

Finally, the fifth case is $(\sigma^2=36)$ and then $b(0)<\sigma^2$. In this case all the schemes have a sublinear behavior as we can see in
Figure \ref{figure:simulations-d} and the fifth line of  Table \ref{table:NumericalResultsAlphaOneHalf}.  This case illustrate the necessity of some condition over the parameters of the model to obtain the optimal rate of convergence for the SMS.

\paragraph{Empirical results for  $\alpha>\tfrac{1}{2}$. } 
In Figure \ref{figure:simulations2} and Table \ref{table:NumericalResultsAlphaBiggerThanOneHalf} we present the results of the simulations for $\alpha=0.6$, and $\alpha = 0.7$. 

In these cases, it can be observed in numerical experiments  that the MES  needs smaller $\dt$ to achieve its theoretical order one convergence rate, unless one tunes the projection operator in the manner of Remark 5.1 in \cite{Chassagneux:2015aa}.  Since this tuning is not explicitly given we do not include the MES in these simulations. 

\begin{table}[h!]
\begin{center}
\bgroup
\def\arraystretch{2}
\begin{tabular}{|c|c|c|c|c|c|c|c|}
\cline{1-8}
\multicolumn{2}{|c|}{Parameters} & \multicolumn{6}{ |c| }{Observed $L^1(\Omega)$ convergence rate $\hat{\rho}$ {\small (and its $R^2$ value)}} \\[-0.2cm]  \cline{1-8} 
\multirow{2}{*}{$\alpha$}&\multirow{2}{*}{$\sigma^2$}
&\multicolumn{2}{|c|}{SMS}
&\multicolumn{2}{|c|}{BMS}
&\multicolumn{2}{|c|}{SES}\\[-0.2cm]  \cline{3-8}
&
&$\hat{\rho}$&\small\cellcolor{gray!20}($R^2$)
&$\hat{\rho}$&\small\cellcolor{gray!20}($R^2$)
&$\hat{\rho}$&\small\cellcolor{gray!20}($R^2$)\\[-0.2cm]  \cline{1-8}
\multirow{3}{*}{$0.6$}&49
&\cellcolor{red!20}0.9819&\small\cellcolor{gray!20}(99.9\%)
&\cellcolor{blue!20}0.7296&\small\cellcolor{gray!20}(99.1\%)
&\cellcolor{blue!20}0.5273&\small\cellcolor{gray!20}(99.8\%)\\ \cline{2-8}
&53.29
&\cellcolor{red!20}0.9766&\cellcolor{gray!20}(99.9\%)
&\cellcolor{blue!20}0.7788&\small\cellcolor{gray!20}(99.3\%)
&\cellcolor{blue!20}0.5133&\small\cellcolor{gray!20}(99.9\%)\\ \cline{2-8}
&144
&\cellcolor{blue!20}0.6609 &\small\cellcolor{gray!20}(98.9\%)
&\cellcolor{blue!20}0.4336&\small\cellcolor{gray!20}(97.3\%)
&\cellcolor{blue!20}0.5074&\small\cellcolor{gray!20}(99.9\%)\\ \cline{1-8}
\multirow{3}{*}{$0.7$}&64
&\cellcolor{red!20}1.004&\small\cellcolor{gray!20}(99.9\%)
&\cellcolor{red!20}0.9022&\small\cellcolor{gray!20}(99.7\%)
&\cellcolor{blue!20}0.5242&\small\cellcolor{gray!20}(99.8\%)\\ \cline{2-8}
&81
&\cellcolor{red!20}0.9991&\small\cellcolor{gray!20}(99.9\%)
&\cellcolor{blue!20}0.8813&\small\cellcolor{gray!20}(99.7\%)
&\cellcolor{blue!20}0.5327&\small\cellcolor{gray!20}(99.7\%)\\ \cline{2-8}
&225&\cellcolor{red!20}0.9146&\cellcolor{gray!20}(99.7\%)
&\cellcolor{blue!20}0.6497&\small\cellcolor{gray!20}(97.6\%)
&\cellcolor{blue!20}0.6410&\small\cellcolor{gray!20}(99.2\%)\\ \cline{1-8}
\end{tabular}
\egroup
\end{center}
\caption{Empirical rate of convergence $\hat{\rho}$ for the $L^1(\Omega)$-error,  when $\alpha>\tfrac{1}{2}$ for different values of $\alpha$ and $\sigma^2$. \label{table:NumericalResultsAlphaBiggerThanOneHalf}}
\end{table}
We have observed in the numerical experiments three cases for the parameters. The first one  is when $b(0)>2\alpha(1-\alpha)^2\sigma^2$ ($\sigma^2=49$ and $\sigma^2=64$).  In this case, Theorem \ref{mainTheorem} holds and we observe the order one convergence (see Figures \ref{figure:simulations2-a},  \ref{figure:simulations2-b}, and first and fourth row in Table  \ref{table:NumericalResultsAlphaBiggerThanOneHalf}). The second case is when the parameters do not satisfy  $b(0)>2\alpha(1-\alpha)^2\sigma^2$ ($\sigma^2=53.29$ and $\sigma^2=81$), and then we can not apply Theorem \ref{mainTheorem}, but in the numerical simulations we still observe the order one convergence (see Figures \ref{figure:simulations2-c},  \ref{figure:simulations2-d}, and second and fifth row in Table  \ref{table:NumericalResultsAlphaBiggerThanOneHalf}). Finally the third case, is when $\sigma\gg b(0)$, and then we do not observe a linear convergence anymore (see Figures \ref{figure:simulations2-e},  \ref{figure:simulations2-f}, and third and six row in Table  \ref{table:NumericalResultsAlphaBiggerThanOneHalf}).  Notice that in the three cases the SMS performs better than the BMS, specially when $\sigma^2$ grows. (see Table \ref{table:NumericalResultsAlphaBiggerThanOneHalf}). 

The second and third case show us that some restriction has to be impose on the parameters to observe the convergence of  order one. But our restriction, although sufficient, it seems to be too strong, specially for $\alpha$ close to one.

\begin{figure}[t!]
    \centering
    \begin{subfigure}[t]{0.47\textwidth}
        \centering
        \includegraphics[width=\textwidth,height=6cm]{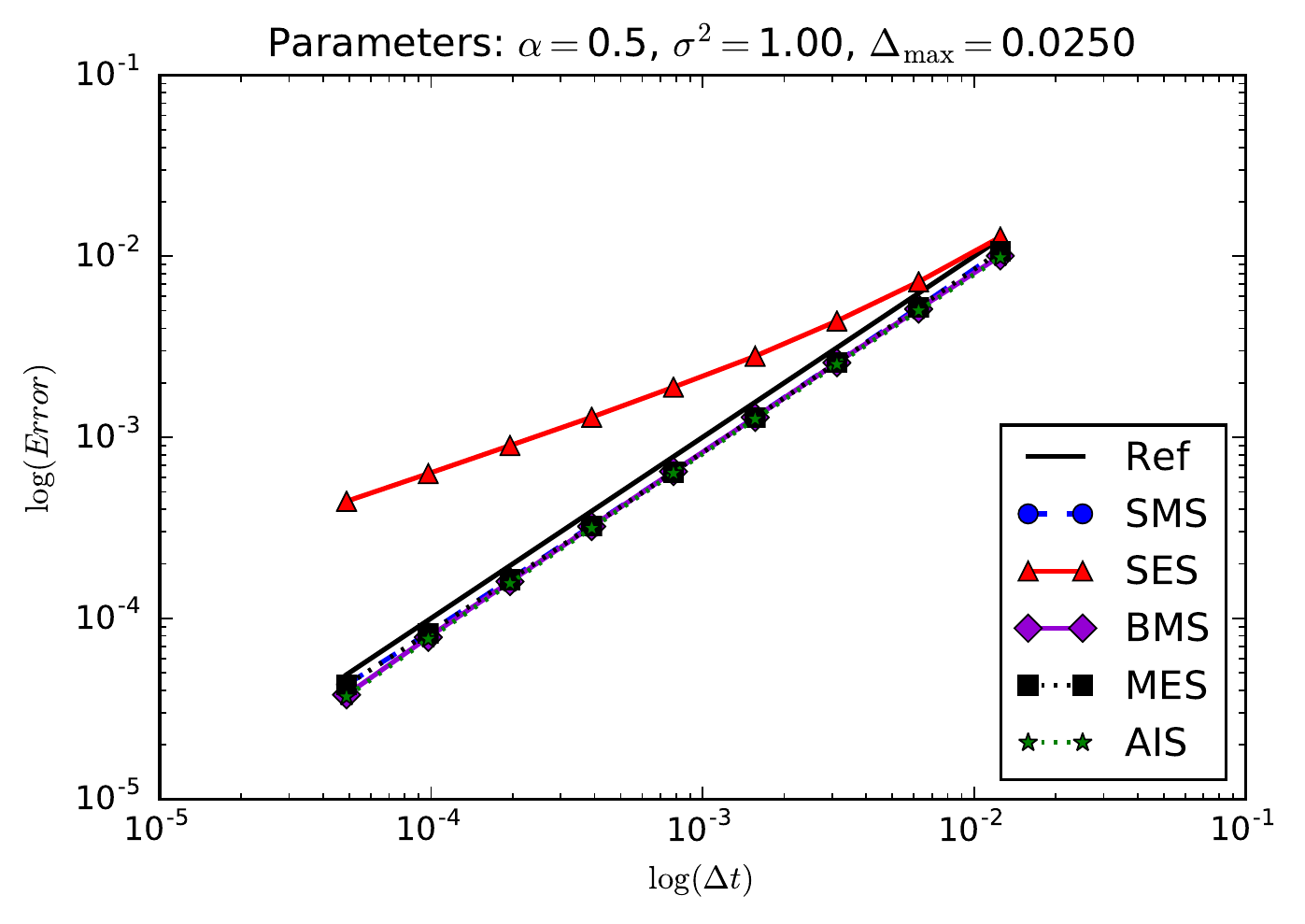}
        \caption{Parameters in case 1: $b(0)>6\sigma^2$. \label{figure:simulations-a}}
    \end{subfigure}
    \quad
    \begin{subfigure}[t]{0.47\textwidth}
        \centering
        \includegraphics[width=\textwidth,height=6cm]{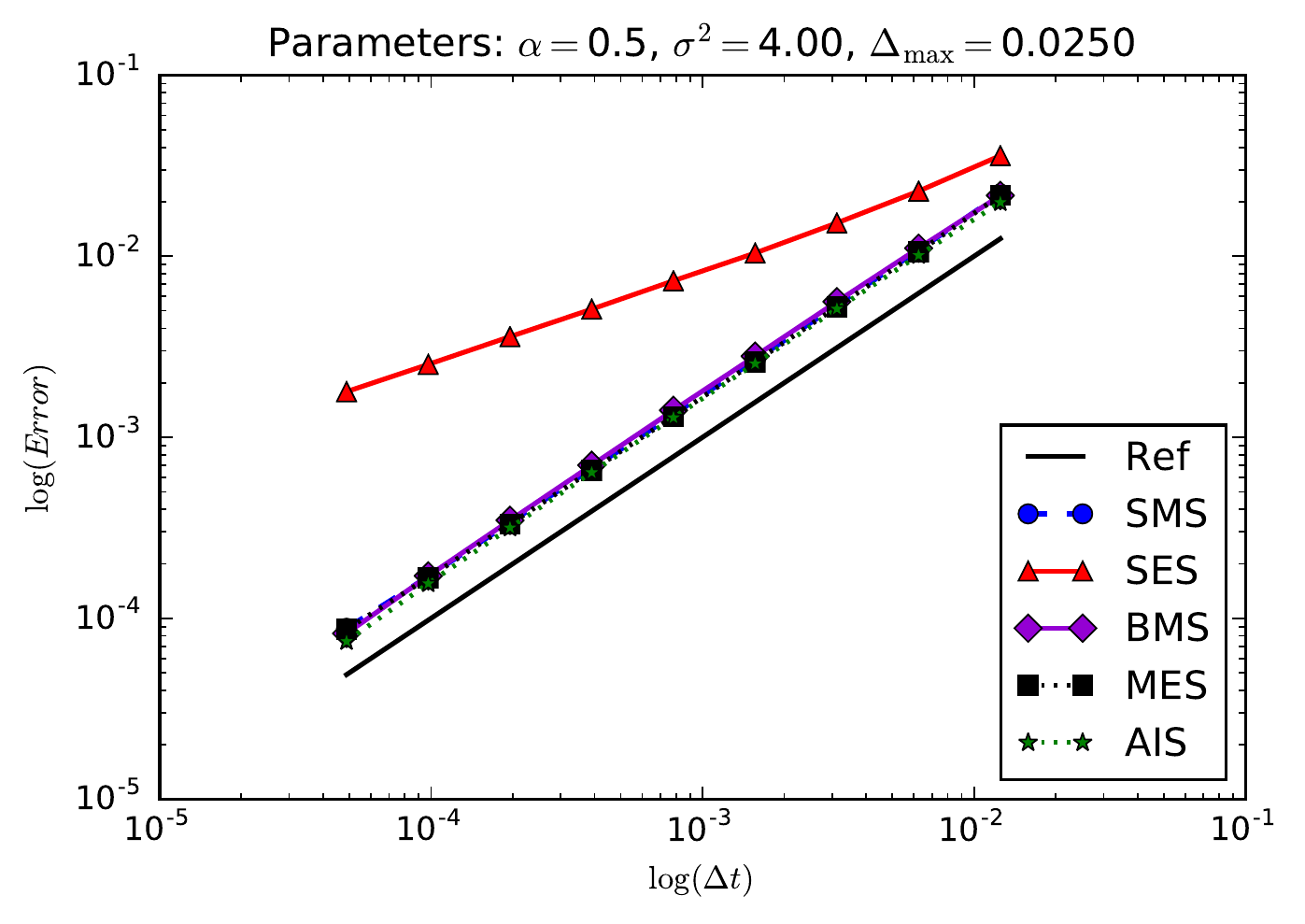}
        \caption{Parameters in case 2: $b(0)\in (5\sigma^2/2, 6\sigma^2)$.\label{figure:simulations-b}}
    \end{subfigure}
    \\
    \begin{subfigure}[t]{0.47\textwidth}
        \centering
        \includegraphics[width=\textwidth,height=6cm]{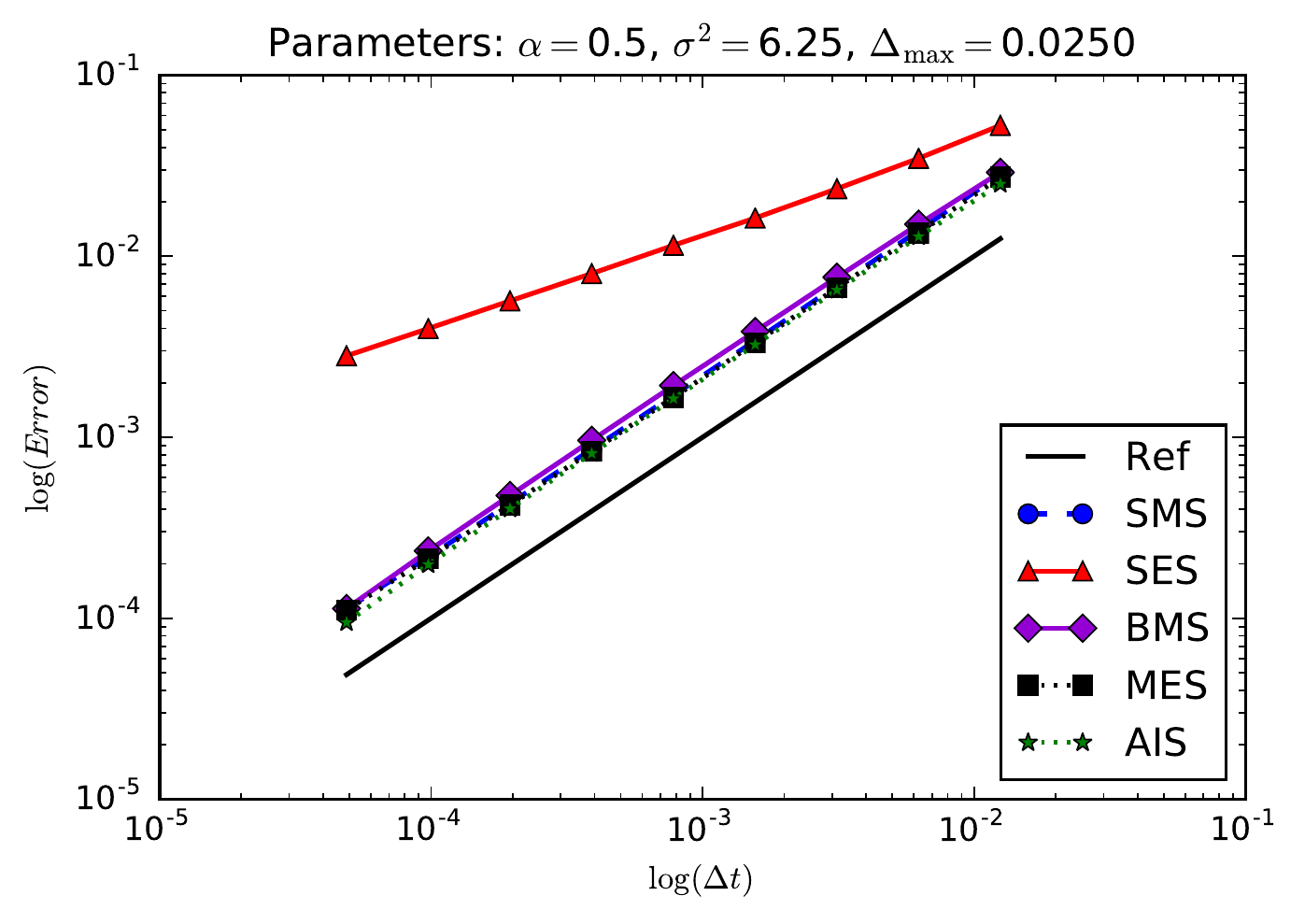}
        \caption{Parameters in case 3: $b(0)\in (3\sigma^2/2, 5\sigma^2/2)$.\label{figure:simulations-c}}
    \end{subfigure}
    \quad
    \begin{subfigure}[t]{0.47\textwidth}
        \centering
        \includegraphics[width=\textwidth,height=6cm]{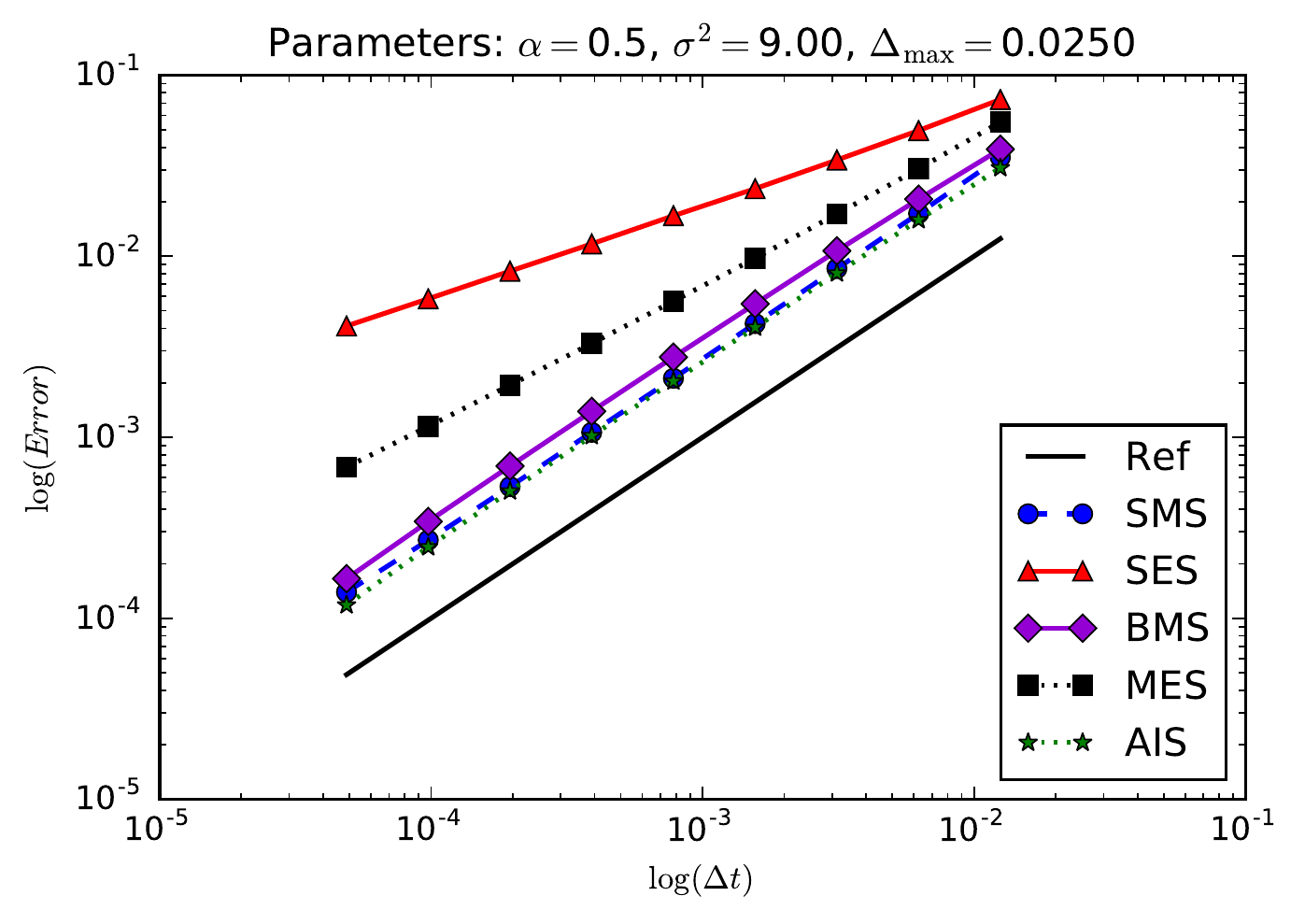}
        \caption{Parameters in case 4: $b(0)\in (\sigma^2, 3\sigma^2/2)$.\label{figure:simulations-d}}
    \end{subfigure}
    \\
    \begin{subfigure}[t]{0.47\textwidth}
        \centering
        \includegraphics[width=\textwidth,height=6cm]{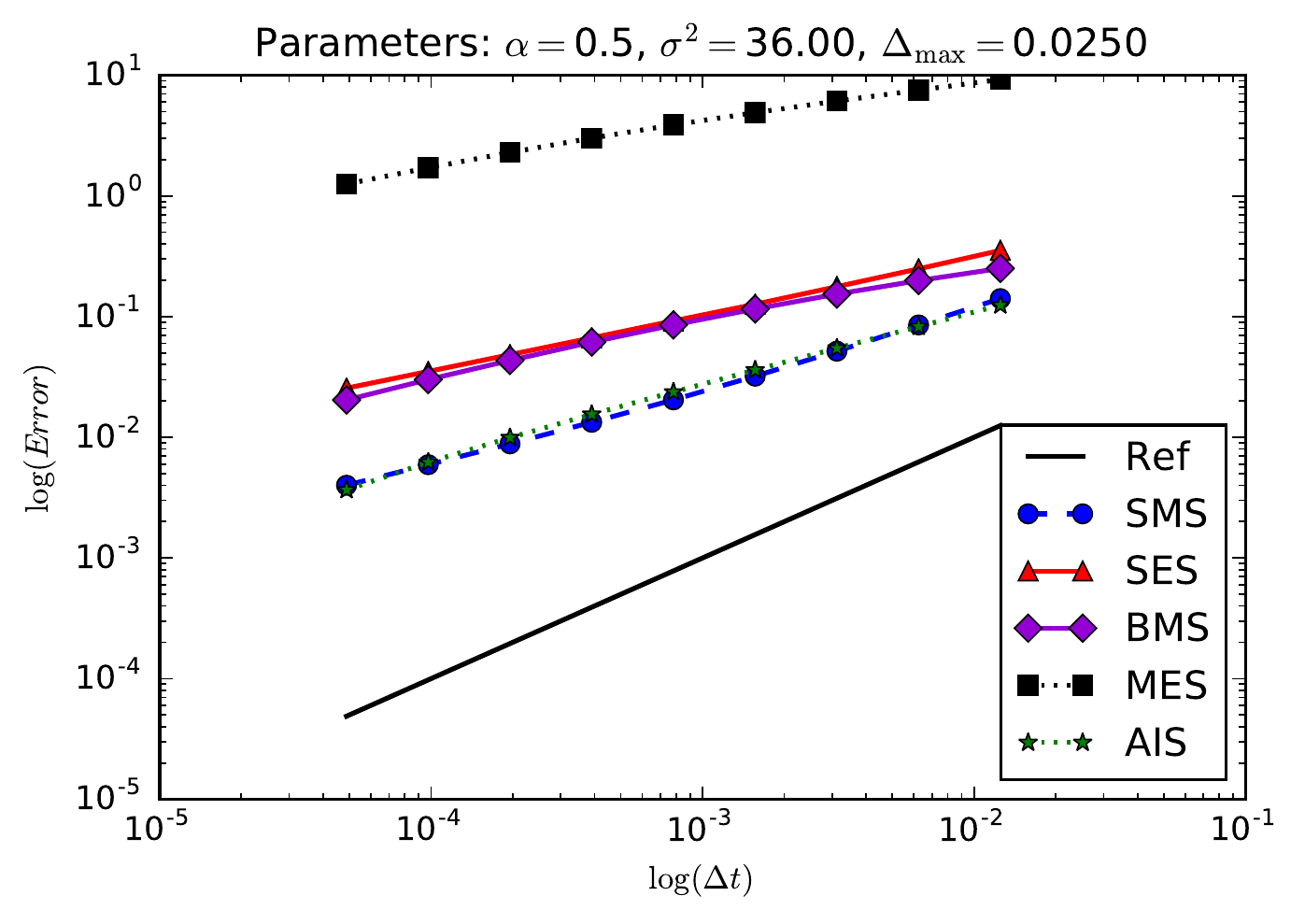}
        \caption{Parameters in case 5: $b(0)<\sigma^2$.\label{figure:simulations-e}}
    \end{subfigure}
    \caption{Step size $\dt$  versus the estimated $L^1(\Omega)$-strong error for the CIR Process  (Log-Log scale). The  identity map  serves as a reference line of rate one.    \label{figure:simulations}}
\end{figure}

\begin{figure}[t!]
    \centering
    \begin{subfigure}[t]{0.47\textwidth}
        \centering
        \includegraphics[width=\textwidth,height=6cm]{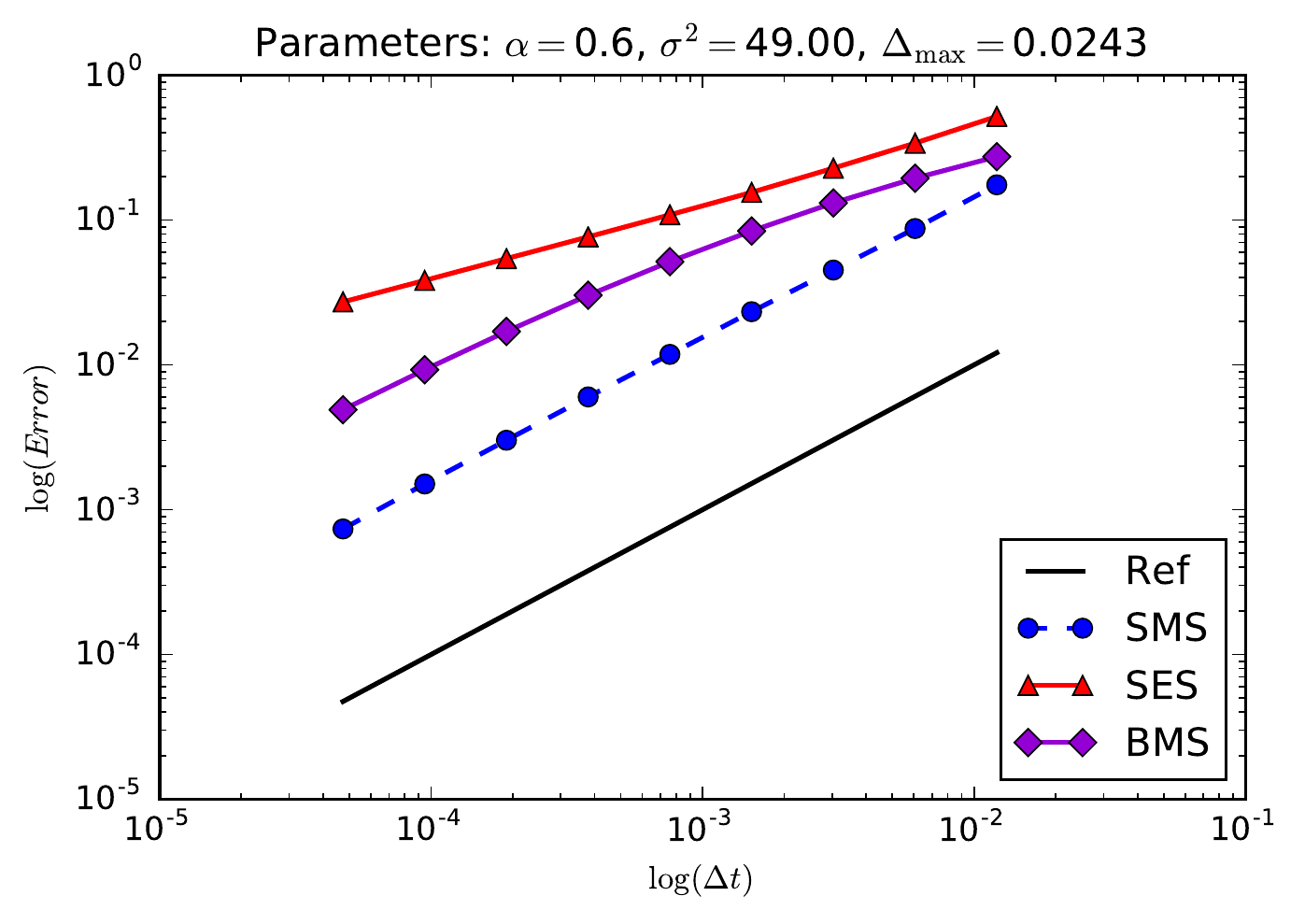}
        \caption{Parameters in case 1: $\alpha=0.60$. \label{figure:simulations2-a}}
    \end{subfigure}
    \quad
    \begin{subfigure}[t]{0.47\textwidth}
        \centering
        \includegraphics[width=\textwidth,height=6cm]{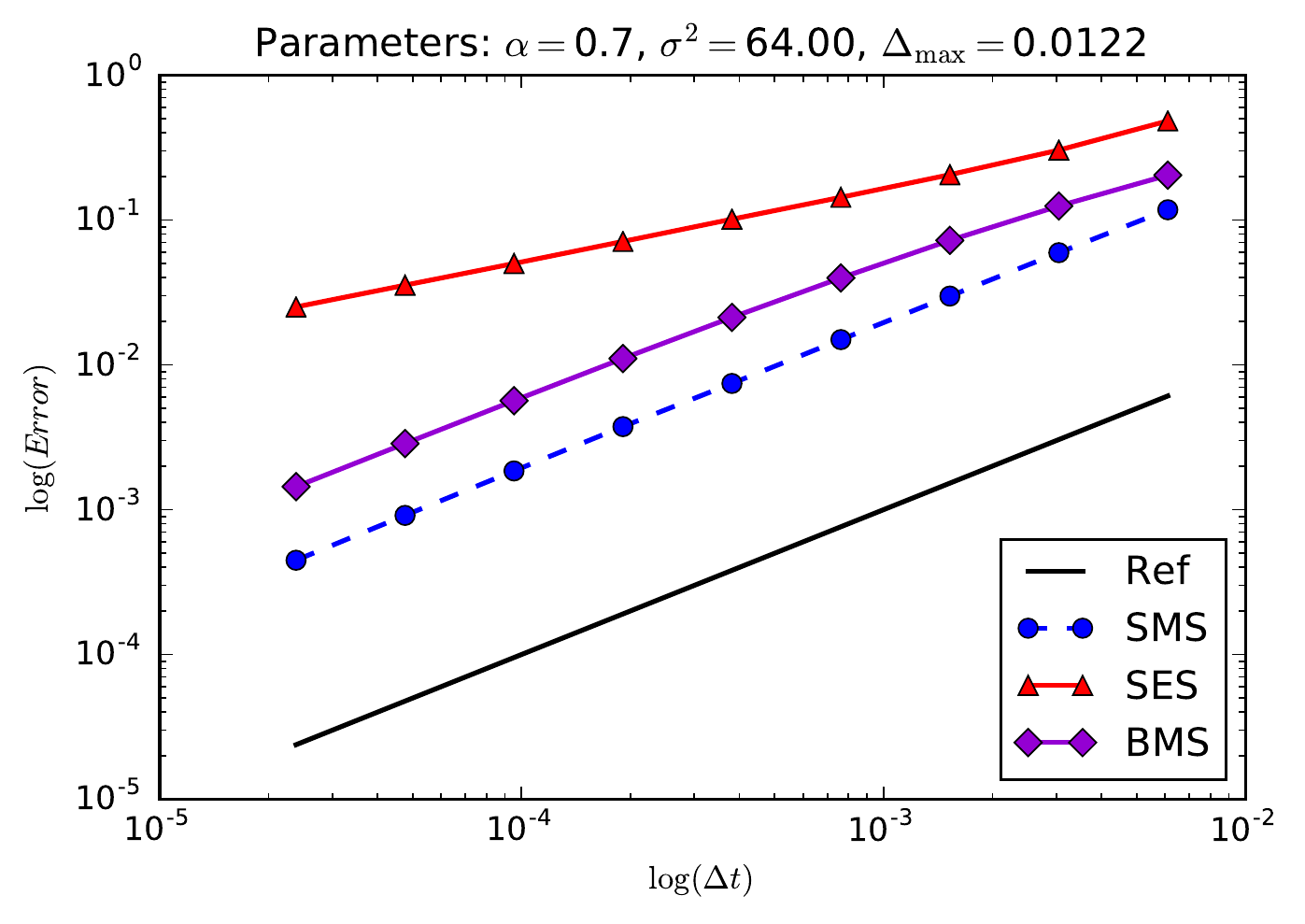}
        \caption{Parameters in case 1: $\alpha=0.70$.\label{figure:simulations2-b}}
    \end{subfigure}
    \\
    \begin{subfigure}[t]{0.47\textwidth}
        \centering
        \includegraphics[width=\textwidth,height=6cm]{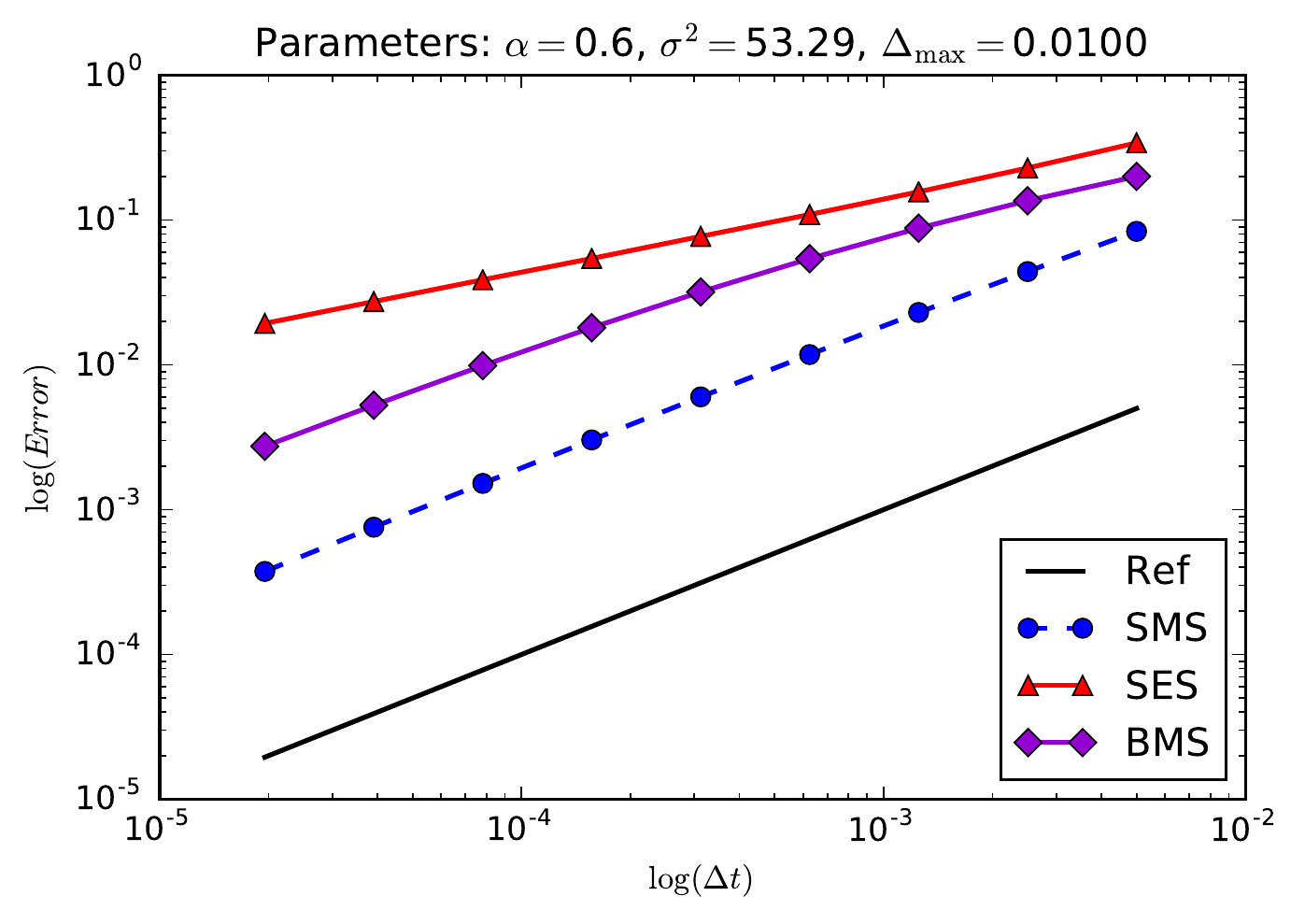}
        \caption{Parameters in case 2: $\alpha=0.60$.\label{figure:simulations2-c}}
    \end{subfigure}
    \quad
    \begin{subfigure}[t]{0.47\textwidth}
        \centering
        \includegraphics[width=\textwidth,height=6cm]{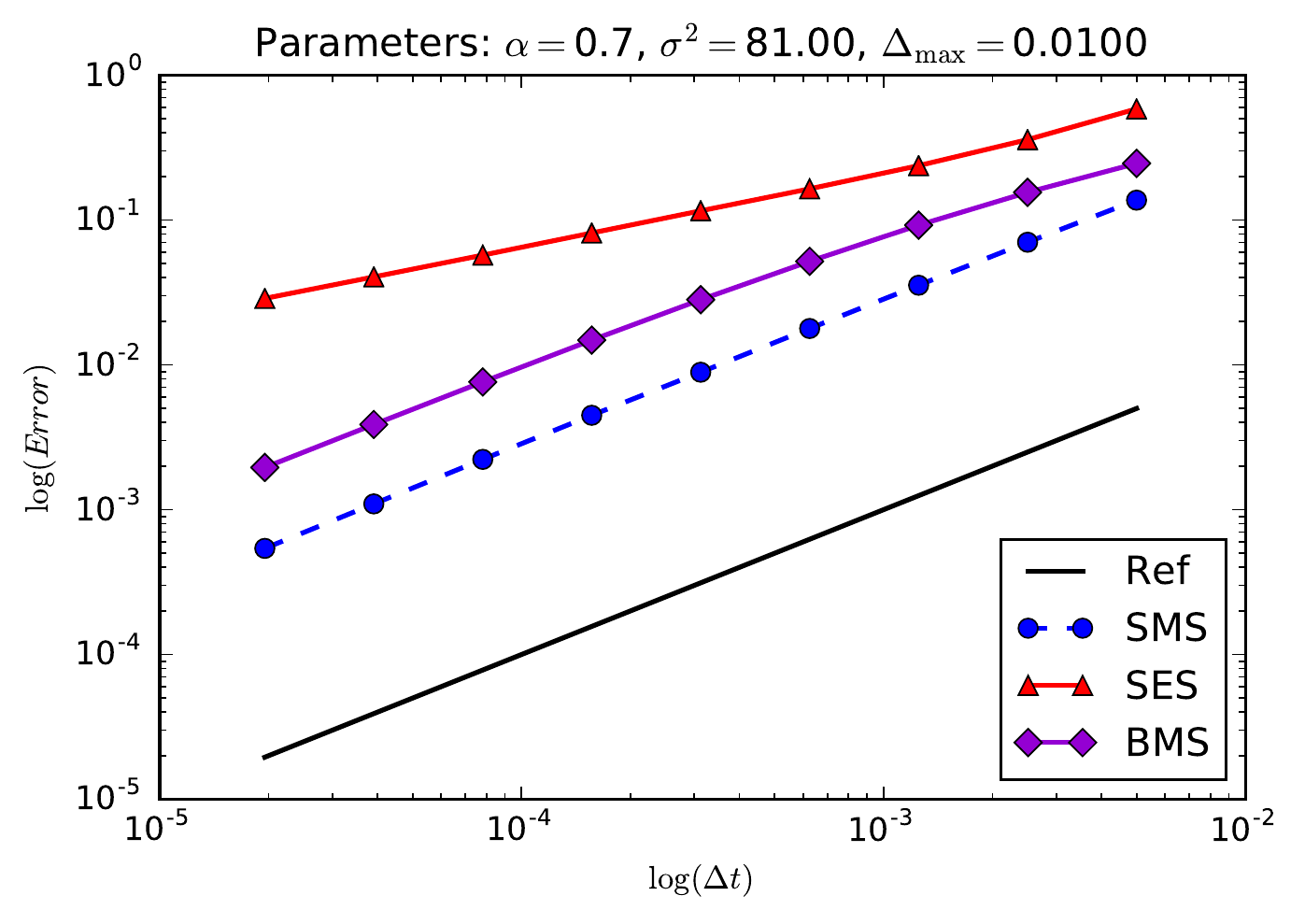}
        \caption{Parameters in case 2: $\alpha=0.70$.\label{figure:simulations2-d}}
    \end{subfigure}
    \\
    \begin{subfigure}[t]{0.47\textwidth}
        \centering
        \includegraphics[width=\textwidth,height=6cm]{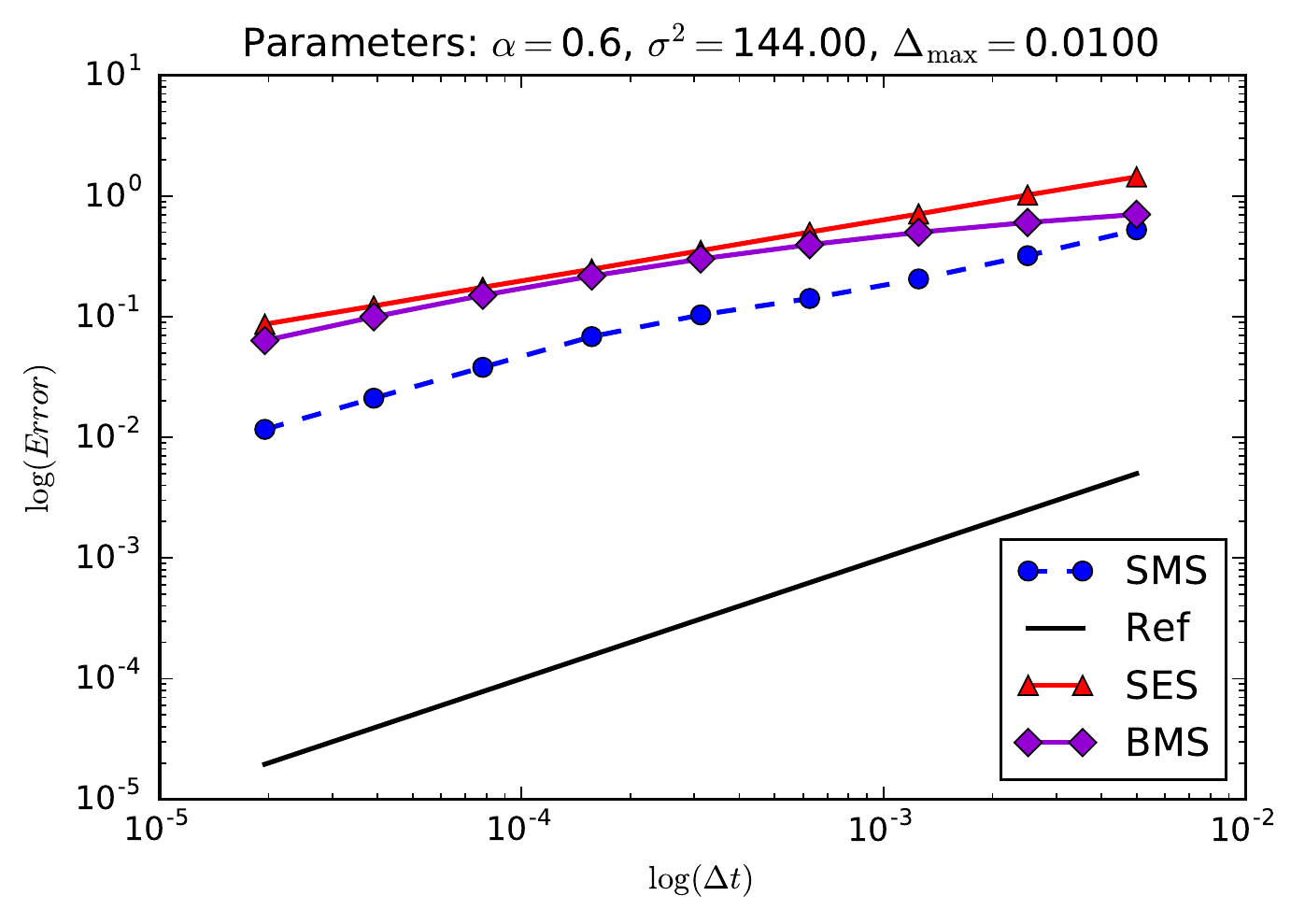}
        \caption{Parameters in case 3: $\alpha=0.60$\label{figure:simulations2-e}}
    \end{subfigure}
    \quad
    \begin{subfigure}[t]{0.47\textwidth}
        \centering
        \includegraphics[width=\textwidth,height=6cm]{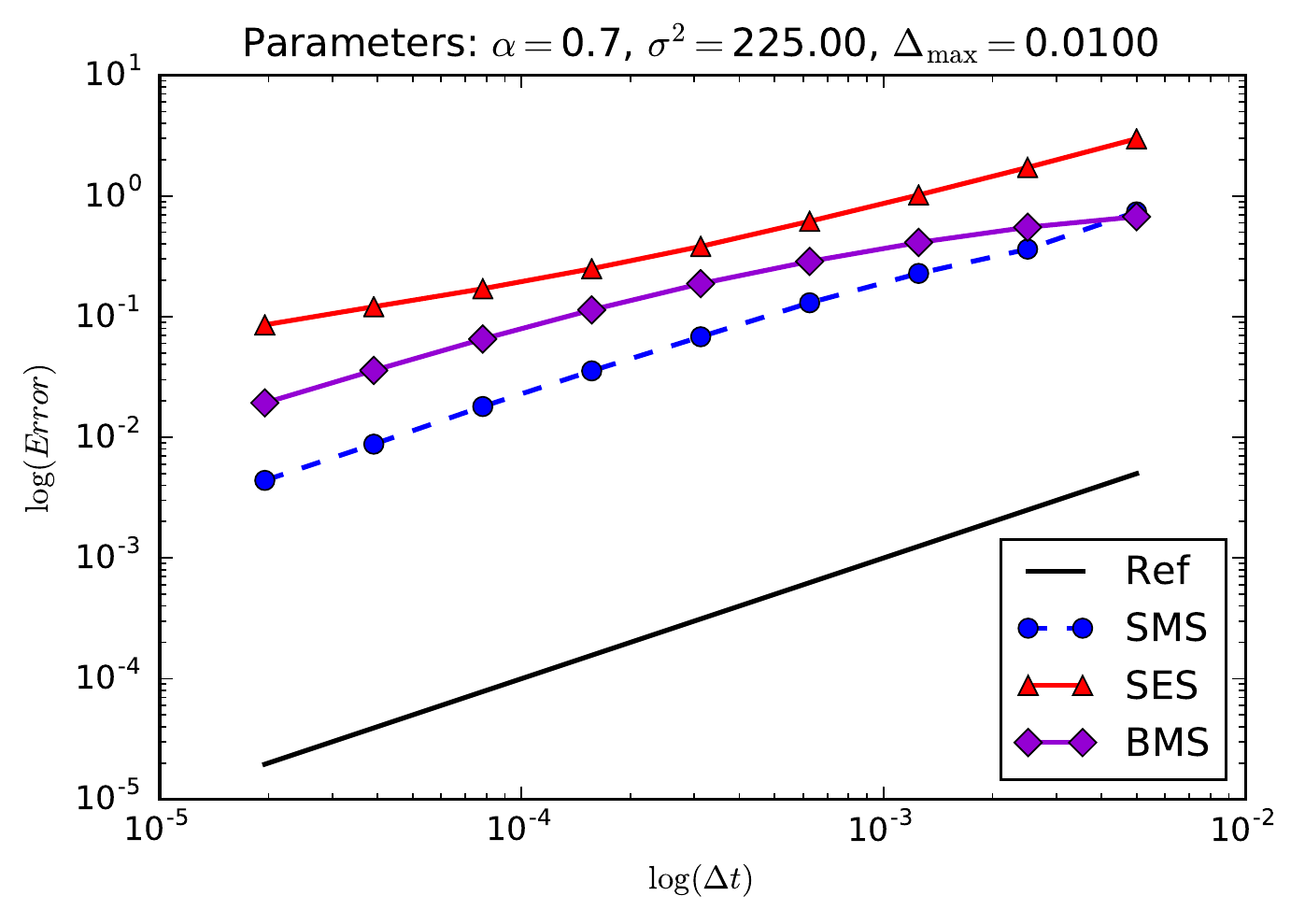}
        \caption{Parameters in case 3: $\alpha=0.70$  \label{figure:simulations2-f}}
    \end{subfigure}
    \caption{Step size $\dt$  versus the estimated $L^1(\Omega)$-error for $\alpha>\tfrac{1}{2}$ and different values for $\sigma^2$ (Log-Log scale). The  identity map  serves as a reference line of rate one.    . \label{figure:simulations2}}
\end{figure}

\subsection{Application of the SMS in Multilevel Monte Carlo}

We continue this section  by testing the SMS in the context of a multilevel Monte Carlo application widely used nowadays in computational finance (see e.g. \cite{Higham:2015aa} and  references therein).   
Multilevel Monte Carlo is an efficient technique introduced by Giles \cite{Giles:2008aa} to decrease the computational complexity of an estimator combining Monte Carlo simulation and time discretisation scheme for a given threshold in the accuracy. For details we refer to  \cite{Giles:2008aa,Giles:2008ab,Higham:2015aa}.  

For this experiment, we consider the classical but non trivial  test-case of  the Zero Coupon Bound (ZCB) pricing of maturity $T$, 
$$ B(0,T) = \E\left[ \exp\Big(-\int_0^T r_s ds\Big) \right], $$
under the hypothesis that the short interest rate dynamics $(r_t,t\geq 0)$ is  modeled with a CIR process  ($\alpha=\tfrac{1}{2}$ and $b(x)=a-bx$) : 
$$dr_t  = (a-br_t)dt + \sigma \sqrt{r_t}dW_t.$$
In this context, the price of the ZCB admits a wellknown closed-form solution given by (see e.g \cite{Cox:1985fk, lamberton1996introduction}) 
$$B(0,T)=A(T)e^{-B(T)r_0},\quad A(T)=\left[ \frac{2\lambda e^{(b+\lambda)T/2}}{(\lambda+b)(e^{\lambda T }-1)+2\lambda}\right]^{\frac{2a}{\sigma^2}},\quad B(T)=\frac{2(e^{\lambda T }-1)}{(\lambda+b)(e^{\lambda T }-1)+2\lambda}.$$
where $r_0$ is the initial value of the interest rate, and $\lambda = \sqrt{b^2+2\sigma^2}$. 
Let $\E\widehat{B}(\dt_{(l)})$  a discrete-time weak approximation of $B(0,T)$ with the time step $\dt_{(l)}$.  We consider the $L$-level Monte Carlo estimator : 
 $$\widehat{Y}_T = \frac{1}{N_0}  \sum_{i=1}^{N_0} \widehat{B}^{(i)}(\dt_{(0)}) + \sum_{l=1}^L   \frac{1}{N_l}  \sum_{i=1}^{N_l} \Big(\widehat{B}^{(i)}(\dt_{(l)}) - \widehat{B}^{(i)}(\dt_{(l-1)})\Big).$$ 
For a targeted mean-square error $\epsilon^2$ on the computation of the quantity $B(0,T)$
$$
\E[(\widehat{Y}_T - B(0,T))^2] = \mathcal{O}(\epsilon^2), 
$$
one can choose the following a priori parametrization of the MLMC method   in order to minimize the computational time (complexity) (see \cite{Giles:2008aa,Giles:2008ab,Higham:2015aa}):  we use the estimation $L = \dfrac{\log \epsilon ^{-1}}{\log{2}}$; from one level to the next one the time step is divided by $2$,  ${\dt}_{(l)}  = \tfrac{1}{2^{(l+1)}}$;  the number of trajectories to simulate is estimated with Giles formula \cite{Giles:2008aa}
$$N_l = \frac{2}{\epsilon^2} \sqrt{V_l \dt_l} \left( \sum_{l=0}^L \sqrt{V_l / \dt_{(l)}}\right),$$ 
with $V_l  = \overline{\mathbb{V}ar}\Big(\widehat{B}^{(1)}(\dt_{(l)}) - \widehat{B}^{(1)}(\dt_{(l-1)})\Big)$. 
As an estimator for the bias variance, the strong rate of convergence of the discretisation scheme enters  as a key ingredient in the $N_l$ a priori estimation. 
{A scheme with a reduced strong bias will then allow a smaller $N_l$.  We apply the MLMC computation for the SMS, the PMS, the AIS and the BMS. }

We summarize the results of the performance comparison between the four schemes in Table \ref{table:CPUTimeMLMC}.  
The computation have been run using  a initial interest rate $r_0 = 1$, the maturity of the bond $T=1$, the drift parameters $a=b=10$, the volatility $\sigma=1$. For the MLMC simulation, we fix a minimum number of trajectories equal to $500$, and a minimal number of levels equal to $6$. 

In Table \ref{table:CPUTimeMLMC}, we give the measures of the CPU time for a set of three decreasing targeted errors, as long as the effective measured error and the total number of simulated trajectories. 
As expected, the required threshold error has been reached by the MLMC strategy. 
As also expected (see Giles \cite{Giles:2008ab}), Milstein schemes perform  better than their Euler versions.  Finally, as $\epsilon$ decreases, the SMS clearly performs better than his PMS version. 

\begin{table}[h!]
\begin{center}
\bgroup
\def\arraystretch{2}
\begin{tabular}{|c|c|c|c|c|}
\hline
\begin{tabular}{c}
{$\epsilon$ = 1.0e-03}\\[-0.45cm]
{\footnotesize{($L=\,\,9$, $\dt_{(L)}=1/ 2^{10}$)}} 
\end{tabular}  
& SMS & PMS & AIS & BMS \\  \hline 
\begin{tabular}{c}
CPU time\\[-0.15cm]
($N_0 + \cdots + N_L$) \\[-0.45cm]
(observed error)
\end{tabular}  
&
\begin{tabular}{c}
0.2304\\[-0.15cm]
(792 651)  \\[-0.45cm]
(1.970e-05)
\end{tabular}  
& 
\begin{tabular}{c}
0.2657  \\[-0.15cm]
(950 838)  \\[-0.45cm]
(3.347e-04)
\end{tabular}  
&    
\begin{tabular}{c}
0.264    \\[-0.15cm]
(990 769)  \\[-0.45cm]
(3.132e-04 )
\end{tabular}  
& 
\begin{tabular}{c}
0.274  \\[-0.15cm]
(992 432)  \\[-0.45cm]
(3.292e-04)
\end{tabular}  \\ \hline   
\end{tabular}
\begin{tabular}{|c|c|c|c|c|}
\hline
\begin{tabular}{c}
{$\epsilon$ = 1.0e-04}\\[-0.45cm]
{\footnotesize{($L=13$,  $\dt_{(L)} = 1/ 2^{14}$)}}
\end{tabular}  
& SMS & PMS & AIS & BMS \\  \hline 
\begin{tabular}{c}
CPU time\\[-0.15cm]
($N_0 + \cdots + N_L$) \\[-0.45cm]
(observed error)
\end{tabular}  
&
\begin{tabular}{c}
16.871\\[-0.15cm]
{\footnotesize{(56 229 224)}} \\[-0.45cm]
(4.870e-05)
\end{tabular}  
& 
\begin{tabular}{c}
20.843 \\[-0.15cm]
{\footnotesize{(70 876 600)}}  \\[-0.45cm]
(1.091e-04)
\end{tabular}  
&  
\begin{tabular}{c}
17.311   \\[-0.15cm]
{\footnotesize{(73 824 621)}}  \\[-0.45cm]
(9.538e-06)
\end{tabular}  
& 
\begin{tabular}{c}
16.95   \\[-0.15cm]
{\footnotesize{(73 668 115)}}  \\[-0.45cm]
(2.203e-06)
\end{tabular}  \\ \hline   
\end{tabular}
\begin{tabular}{|c|c|c|c|c|}
\hline
\begin{tabular}{c}
{$\epsilon$ = 1.0e-05}\\[-0.45cm]
{\footnotesize{($L=16$,  $\dt_{(L)} = 1/ 2^{17}$)}}
\end{tabular}  
& SMS & PMS & AIS & BMS \\  \hline 
\begin{tabular}{c}
CPU time\\[-0.15cm]
($N_0 + \cdots + N_L$) \\[-0.45cm]
(observed error)
\end{tabular}  
&
\begin{tabular}{c}
1589.6 \\[-0.15cm]
{\scriptsize{(5 531 879 264)}} \\[-0.45cm]
(4.752e-06)
\end{tabular}  
& 
\begin{tabular}{c}
1910.7 \\[-0.15cm]
{\scriptsize{(6 913 546 698)}}  \\[-0.45cm]
(5.889e-06)
\end{tabular}  
&  
\begin{tabular}{c}
1576.4   \\[-0.15cm]
{\scriptsize{(7 368 734 119)}}  \\[-0.45cm]
(3.912e-06)
\end{tabular}  
& 
\begin{tabular}{c}
1540.2   \\[-0.15cm]
{\scriptsize{(7 333 474 098)}}  \\[-0.45cm]
(5.653e-07)
\end{tabular}  \\ \hline   
\end{tabular}
\egroup
\end{center}
\caption{CPU time to achieve the target error for the different schemes. The observed error is $|\widehat{Y}_T - B(0,T)|$.  \label{table:CPUTimeMLMC}}
\end{table}
\subsection{Conclusion}

In this paper we have recovered the classical rate of convergence of the Milstein scheme in a context of non smooth diffusion coefficient, although we have  to impose some restrictions over the parameters of the SDE  \eqref{exactProcess} to ensure the theoretical order one of convergence. Typically,  if the quotient $b(0)/\sigma^2$ is big enough we will observe the optimal convergence rate.
 
In the numerical simulations we have observed that, despite the fact it is necessary to impose some restriction over the parameters of SDE \eqref{exactProcess} to obtain the order one convergence, Hypothesis \ref{hypo:2} seems  to be not optimal, specially for $\alpha=\tfrac{1}{2}$.   Also, through numerical simulations, we have observed that the use of SMS  could improve the computation times in a Multilevel Monte Carlo framework, at least as well as the (CIR specialized) one-order  schemes. 

Although our  result seems more restrictive in term of hypotheses on the set of parameters, in particular if we compare SMS  with  Lamperti's transformation-based schemes  (see the recent works in \cite{Alfonsi:2013aa} and  \cite{Chassagneux:2015aa}, the SMS  can be applied to a more general class of drifts functions and  in various  contexts. It is thus a useful complement  of the existing literature. 

\section{Proofs for preliminary lemmas} \label{sec:ProofForPreliminaryResultsOnXbar}

\subsection{On the Local Error of the SMS}

\begin{proof}[Proof of Lemma  \ref{lem:localErrorOfTheScheme}]
From the definition of $\X$, and the algebraic inequality for positive real numbers $(a_1+\ldots+a_n)^p\leq n^p(a_1^p+\ldots+a_n^p)$ we have
\begin{align*}
\left|\X_t  - \Xe{t} \right|^{2p} 
&\leq 3^{2p}\left(b(\Xe{t})^{2p} (t-\eta(t))^{2p} +\sigma^{2p} \Xe{t}^{2\alpha p} (W_t-W_{\eta(t)})^{2p}\right.\\
&\qquad\qquad \left. +\frac{\alpha^{2p} \sigma^{4p} }{2^{2p} }\Xe{t}^{(2\alpha-1)2p}\left[(W_t-W_{\eta(t)})^2 -(t-\eta(t)) \right]^{2p} \right). 
\end{align*}
Thanks to the linear growth of $b$, Lemma \ref{lem:finitenessOfTheMomentsOfXbar} and the properties of the Brownian Motion it is quite easy to conclude the existence of a constant $C$ such that 
$$\E\left[|\X_t  - \Xe{t}|^{2p}\right]  \leq C\dt^p,$$
from where the result follows.
\end{proof}

\subsection{On the Probability of SMS being close to zero}
From $b_\sigma(\alpha)$ and $K(\alpha)$ defined in \eqref{def:b_alpha},  let us recall the notation 
\begin{align*}
\xx :=  \dfrac{ b_\sigma(\alpha)}{K(\alpha)}
\end{align*}
introduced in Lemma \ref{lem:ZbarPositiveSmallx}. As $b_\sigma(\alpha)>0$ under Hypothesis \ref{hypo:2}-$(i)$,  $\xx$ is bounded away from $0$. In particular, 
\begin{align*}
\lim_{\alpha\to \tfrac{1}{2}} \xx= \frac{(b(0)-\sigma^2/4)}{K},\quad\mbox{whereas}\quad \lim_{\alpha\to 1}  \xx= \frac{b(0)}{K+\sigma^2/2}.
\end{align*}
\begin{proof}[Proof of Lemma \ref{lem:ZbarPositiveSmallx}]
Denoting $\dW_s = (W_s-W_{\eta(s)})$, and $\ds=s-\eta(s)$, we have for all $s\in[0,T]$,
\begin{equation*}
\begin{split}
\Z{s}= \frac{\alpha\sigma^2}{2}\Xe{s}^{2\alpha-1}\dW_s^2 
 + \sigma \Xe{s}^\alpha \dW_s+ \Xe{s} +\left(b(\Xe{s}) - \frac{\alpha\sigma^2}{2}\Xe{s}^{2\alpha-1}\right)\ds.
\end{split}
\end{equation*}
From the Lipschitz property of $b$ and the following bound  
for any $x>0$
\begin{equation}\label{eq:boundForBadTerm}
x^{2\alpha-1} \leq 4(1-\alpha)^2  + (2\alpha-1)[2(1-\alpha)]^{-\frac{2(1-\alpha)}{2\alpha-1}}x,
\end{equation}
we have
\begin{equation}\label{eq:linearboundForZbar}
\Z{s}\geq \frac{\alpha\sigma^2}{2}\Xe{s}^{2\alpha-1}\dW_s^2 
 + \sigma \Xe{s}^\alpha\dW_s+  \Xe{s} +\left(b_\sigma(\alpha)  - K(\alpha)\Xe{s}   \right)\ds.
\end{equation}
So,
\begin{align*}
\P\Big[\Z{s}&\leq (1-\rho ){b_\sigma}(\alpha)\ds ,\; \Xe{s}< {\rho} \xx \Big] \\
&\leq \P\Big[\frac{\alpha\sigma^2}{2}\Xe{s}^{2\alpha-1}\dW_s^2 + \sigma \Xe{s}^\alpha\dW_s  
+\Xe{s} +\left(\rho {b_\sigma}(\alpha)  - K(\alpha)\Xe{s}   \right)\ds\leq 0 ,\; \Xe{t}< {\rho} \xx \Big].
\end{align*}
From the independence of $\dW_s$ with respect to $\F_{\eta(s)}$, if we denote by $\mathcal{N}$ a standard Gaussian variable, we have
\begin{align*}
\P\Big[\Z{s}&\leq (1-\rho ){b_\sigma}(\alpha)\ds ,\; \Xe{s}< {\rho} \xx \Big] \\
   &\leq \E\Big[ \P\Big(  \frac{\alpha\sigma^2}{2}x^{2\alpha-1}\ds \mathcal{N}^2 + \sigma\sqrt{\ds} x^\alpha \mathcal{N} 
   +x+\left[\rho {b_\sigma}(\alpha)  - K(\alpha)x \right]\ds\leq 0  \Big)\Big|_{x=\Xe{s}}\ind{\{\Xe{s}< {\rho} \xx \}} \Big].
\end{align*}
Notice that in the right-hand side we have a quadratic polynomial of a standard Gaussian random variable. Let us compute its discriminant:
\begin{align*}
\Delta(x,\alpha) &= \sigma^2x^{2\alpha}\ds  - 2\alpha\sigma^2x^{2\alpha-1}\ds \left(x + \left(\rho {b_\sigma}(\alpha)  -K(\alpha)x \right)\ds \right)\\
& =  -(2\alpha-1)\sigma^2x^{2\alpha} \ds  - 2\alpha\sigma^2x^{2\alpha-1} \ds^2 \left(\rho {b_\sigma}(\alpha)  - K(\alpha)x\right).  
\end{align*}
Since ${b_\sigma}(\alpha)>0$,  we have $\Delta(x,\alpha)<0$ for all $\alpha\in[\tfrac{1}{2},1)$, and $x\leq {\rho} \xx$. So, for all $\ds\leq\dt$ we have
$$\P\left[\Z{s}\leq (1-\rho ){b_\sigma}(\alpha)\ds ,\; \Xe{s}< {\rho} \xx  \right]=0,$$
taking $\ds=\dt$ we conclude on the Lemma. 
\end{proof}

\begin{proof}[Proof of Lemma \ref{lem:ZbarRemainsPositive}]
We have
\begin{align}\label{eq:firstBoundReflectionProof}
\P\left( \inf_{t_k\leq s<t_{k+1}}\Z{s}\leq0\right)
=\P\left( \inf_{t_k\leq s<t_{k+1}}\Z{s}\leq0, \X_{t_k}\geq \xx \right)
+\P\left( \inf_{t_k\leq s<t_{k+1}}\Z{s}\leq0, \X_{t_k}< \xx \right)
\end{align}
We start with the second term in the right hand of the last inequality. By continuity of the path of $\Z{}$ and Lemma \ref{lem:ZbarPositiveSmallx}, we have
$$\P\left( \inf_{t_k\leq s<t_{k+1}}\Z{s}\leq0, \X_{t_k}< \xx \right) = \sum_{s\in\Q\cap(t_{k},t_{k+1}]}{\P\left( \Z{s}\leq0, \X_{t_k}< \xx \right)}=0.$$
On the other hand,  from  \ref{eq:linearboundForZbar} we have
\begin{equation}\label{eq:BrownianLoweboundForZbar}
\Z{s} \geq \sigma \Xe{s}^\alpha\dW_s +\left(1  - K(\alpha) \ds  \right)\Xe{s}+{b_\sigma}(\alpha)\ds.
\end{equation}
Then
\begin{align*}
&\P\Big(  \inf_{t_k\leq s<t_{k+1}}\Z{s}\leq0,\;   \X_{t_k}\geq \xx \Big)  \\
&\quad\leq \P \Big( \inf_{t_k<s\leq t_{k+1} } \frac{\Xe{s}^{1-\alpha}}{\sigma}   +\frac{\left(b_\sigma(\alpha)  - K(\alpha)\Xe{s}   \right)\ds}{ \sigma \Xe{s}^\alpha}+\dW_s \leq  0 ,\;\; \X_{t_k}\geq \xx \Big) \\
&\quad = \E\Big[ \psi(\X_{t_k}) \ind{\left\{\X_{t_k}\geq\xx\right\}} \Big], 
\end{align*}
where the last equality holds thanks to the Markov Property of the Brownian motion, for
$$\psi(x)= \P\Big( \inf_{0<u\leq \dt}\frac{x^{1-\alpha}}{\sigma} + \frac{{b_\sigma}(\alpha)-K(\alpha)x}{\sigma x^\alpha}u + B_u    \leq 0\Big),$$
where $(B_t)$ denotes a Brownian Motion independent of $(W_t)$.

If $(B_t^\mu,{0\leq t\leq T})$ is a Brownian motion with drift $\mu$, starting at $y_0$, then for all $y\leq y_0$, we have (see \cite{borodin2002handbook}): 
\begin{equation}\label{eq:BoundBM}
\begin{split}
\P\left( \inf_{0<s\leq t}B_s^\mu \leq y \right) &= \frac{1}{2}\erfc\left( \frac{y_0-y}{\sqrt{2t}} + \frac{\mu\sqrt{t}}{\sqrt{2}} \right)\\
&\quad+ \frac{1}{2}\exp\left(- 2\mu(y_0-y) \right)\erfc\left( \frac{y_0-y}{\sqrt{2t}} - \frac{\mu\sqrt{t}}{\sqrt{2}} \right),
\end{split}
\end{equation}
where for $z\in\R$, $\erfc{z} = \sqrt{2/\pi}\int_{\sqrt{2}z}^\infty{\exp\left( -{u^2}/{2} \right)du}.$ In our case $y_0 = x^{1-\alpha}/\sigma$, $\mu = (b_\sigma(\alpha)-K(\alpha)x)/\sigma x^\alpha$, and $y=0$. Then
\begin{align*}
\psi(x) &= \frac{1}{2}\erfc\left( \frac{\left[(1-K(\alpha)\dt)x +{b_\sigma}(\alpha)\dt  \right]}{\sqrt{2\dt}\sigma x^\alpha}  \right)   \\
& \quad+\frac{1}{2}\exp\left(-\frac{2[{b_\sigma}(\alpha)-K(\alpha)x]x}{\sigma^2x^{2\alpha}}  \right)\erfc\left( \frac{x -[b_\sigma(\alpha)-K(\alpha)x]\dt }{\sqrt{2\dt}\sigma x^\alpha}  \right).
\end{align*}
Since $\dt\leq 1/2K(\alpha)$, for any $x\geq\xx$, the arguments in the $\erfc$ function in the last equality are both positives, and then recalling that for all $z>0$ $\erfc(z)\leq \exp(-z^2)$, we obtain 
\begin{align*}
\psi(x) &\leq \frac{1}{2}\exp\left( -\frac{\left[(1-K(\alpha)\dt)x +{b_\sigma}(\alpha)\dt  \right]^2}{{2\dt}\sigma^2 x^{2\alpha}}  \right)   \\
& \quad+\frac{1}{2}\exp\left(-\frac{2[{b_\sigma}(\alpha)-K(\alpha)x]x}{\sigma^2x^{2\alpha}}  \right)\exp\left(- \frac{[x -[b_\sigma(\alpha)-K(\alpha)x]\dt]^2 }{{2\dt}\sigma^2 x^{2\alpha}}  \right)\\
&\leq \exp\left( -\frac{\left[(1-K(\alpha)\dt)x +{b_\sigma}(\alpha)\dt  \right]^2}{{2\dt}\sigma^2 x^{2\alpha}}  \right).
\end{align*}
So, for all  $x\geq\xx$
$$\psi(x)\leq \exp\left(- \frac{(1-K(\alpha)\dt)^2x^{2(1-\alpha)} }{2\sigma^2\dt}  \right).$$
Then
\begin{align*}
 \P\Big(  \inf_{t_k\leq s<t_{k+1}}\Z{s}\leq0,\;   \X_{t_k}\geq \xx \Big)  
&\leq  \E\Big[ \exp\left( -\frac{(1-K(\alpha)\dt)^2\X_{t_k}^{2(1-\alpha)} }{2\sigma^2\dt}  \right) \ind{\left\{\X_{t_k}\geq\xx\right\}} \Big]\\
&\leq  \exp\left( -\frac{(1-K(\alpha)\dt)^2\xx^{2(1-\alpha)} }{2\sigma^2\dt}  \right),
\end{align*}
and finally, choosing  $\gamma={\xx^{2(1-\alpha)} }/{8\sigma^2}$, we get 
 $$\P\Big(  \inf_{t_k\leq s<t_{k+1}}\Z{s}\leq0 \Big)  \leq \exp\left( -\frac{\gamma }{\dt}  \right).$$
\end{proof}

\subsection{On the Local Time of the SMS at Zero}

\subsubsection*{The Stopping Times $(\stalpha,\tfrac{1}{2}\leq\alpha<1)$}
In what follows, we consider
\begin{equation}\label{defStoppingTime}
\stalpha = \inf\left\{s>0: \X_s< (1-\sqrt{\alpha})b_\sigma(\alpha)\dt \right\}.
\end{equation}

\begin{lem}\label{lem:ProbabilityOfThetaLessThanT}
Assume $b(0)>2\alpha(1-\alpha)^2\sigma^2$, and $\dt \leq 1/(2K(\alpha)) \wedge x_0/[(1-\sqrt{\alpha})b_\sigma(\alpha)]$. Then there exists a positive constant $\gamma$ depending on $\alpha$, $b(0)$, $K$ and $\sigma$ but not on $\dt$ such that
\begin{equation}\label{boundPtauLessT}
\P(\stalpha\leq T) \leq \frac{T}{\dt}\exp\left(\frac{-\gamma}{\dt}\right).
\end{equation}
\end{lem}
\begin{proof} First, notice that the condition $\dt< x_0/[(1-\sqrt{\alpha})b_\sigma(\alpha)]$ ensures that the stopping time $\stalpha$ is almost surely strictly positive. 

To enlighten the notation along this proof, let us call $\lo :=(1-\sqrt{\alpha})b_\sigma(\alpha)$, and $\zeta_k = \inf_{t_k<s\leq t_{k+1}}{\Z{s}}$. We split the proof in three steps.
\paragraph{Step 1.}  Let us prove that for a suitable function $\psi:\R\to[0,1]$ and the set $A_k = \{\X_{t_k}>\xx{\sqrt{\alpha}}\} \in\F_{t_k}$:
\begin{equation}\label{eq:boundForPtheta}
\P\left( \Theta_\alpha\leq T \right) \leq \sum_{k=0}^{N-1}{\E\left( \psi( \X_{t_k})\ind{A_k}  \right)}
\end{equation}
Indeed,
$$\P\left( \Theta_\alpha\leq T \right) \leq \sum_{k=0}^{N-1}{\P\left( \zeta_k \leq  \lo\dt ,\;\; \X_{t_k}>\lo\dt  \right)}.$$
But, for each $k=0,\ldots,N-1$
\begin{align*}
\P\left(  \zeta_k \leq  \lo\dt ,\;\; \X_{t_k}>\lo\dt  \right) &=  \P\left( \zeta_k  \leq \lo\dt ,\;\; \X_{t_k}>\lo\dt, \;\; \X_{t_k} <  \xx{\sqrt{\alpha}}\right)  \\
&\quad +  \P\left(   \zeta_k \leq \lo\dt ,\;\;   \X_{t_k}>\lo\dt,\;\; \X_{t_k}\geq \xx{\sqrt{\alpha}} \right).  
\end{align*}
Since $\lo \dt= (1-\sqrt{\alpha}){b_\sigma}(\alpha)\dt\leq \xx{\sqrt{\alpha}},$  we have
\begin{align*}
  \P\Big( \zeta_k\leq  \lo\dt ,\;\;   \X_{t_k}>  \lo\dt,\;\;& \X_{t_k}<  \xx{\sqrt{\alpha}} \Big)\\
    &\leq  \P\left(  \zeta_k\leq  \lo\dt ,\;\; \X_{t_k}<  \xx{\sqrt{\alpha}} \right)  \\
 &\leq  \sum_{s \in \Q\cap (t_k,t_{k+1}]}\P\left( {\Z{s}}\leq \lo\dt ,\;\; \X_{t_k}<  \xx{\sqrt{\alpha}} \right) = 0,
\end{align*}
thanks to Lemma \ref{lem:ZbarPositiveSmallx}. On the other hand, we have
\begin{align*}
 \P\Big(  \zeta_k  \leq \lo\dt& ,\;  \X_{t_k} >\lo\dt,\; \X_{t_k}\geq \xx{\sqrt{\alpha}} \Big)  \\
&= \P\left( \zeta_k \leq  \lo\dt ,\;\; \X_{t_k}\geq \xx{\sqrt{\alpha}} \right) \\
&\leq \P \Big( \inf_{t_k<s\leq t_{k+1} } \frac{\Xe{s}^{1-\alpha}}{\sigma}   +\frac{\left(b_\sigma(\alpha)  - K(\alpha)\Xe{s}   \right)\ds}{ \sigma \Xe{s}^\alpha}+\dW_s \leq  \frac{\lo\dt}{ \sigma \Xe{s}^\alpha} ,\;\; \X_{t_k}\geq \xx{\sqrt{\alpha}}  \Big) \\
& \quad = \E\Big[ \psi(\X_{t_k}) \ind{\left\{\X_{t_k}>\xx{\sqrt{\alpha}}\right\}} \Big], 
\end{align*}
where the inequality comes from \eqref{eq:linearboundForZbar}, and the last equality holds thanks to the Markov Property of the Brownian motion, for
$$\psi(x)= \P\Big( \inf_{0<u\leq \dt}\frac{x^{1-\alpha}}{\sigma} + \frac{{b_\sigma}(\alpha)-K(\alpha)x}{\sigma x^\alpha}u + B_u    \leq  \frac{\lo\dt}{\sigma x^\alpha}\Big),$$
where $(B_t)$ denotes a Brownian Motion independent of $(W_t)$. Summarizing 
$$\P\left(  \zeta_k \leq  \lo\dt ,\;\; \X_{t_k}>\lo\dt  \right)  \leq \E\Big[ \psi(\X_{t_k}) \ind{\left\{\X_{t_k}>\xx{\sqrt{\alpha}}\right\}} \Big], $$
and we have  \eqref{eq:boundForPtheta} for $A_k = \left\{\X_{t_k}>\xx{\sqrt{\alpha}}\right\}$.

\paragraph{Step 2.} Let us prove that for all $x\geq \xx{\sqrt{\alpha}}$:
\begin{equation}\label{eq:boundForPsi}
\psi(x)\leq \exp\left( -  \frac{(1-K(\alpha)\dt)^2 (\xx\sqrt{\alpha})^{2(1-\alpha)} }{2\sigma^2\dt } \right).
\end{equation}

Applying again \eqref{eq:BoundBM}, we have
\begin{align*}
\psi(x) &= \frac{1}{2}\erfc\left( \frac{\left[(1-K(\alpha)\dt)x + \sqrt{\alpha}{b_\sigma}(\alpha)\dt  \right]}{\sqrt{2\dt}\sigma x^\alpha}  \right)   \\
& \quad+\frac{1}{2}\exp\left(-\frac{2[{b_\sigma}(\alpha)-K(\alpha)x][x-(1-\sqrt{\alpha}){b_\sigma}(\alpha)\dt]}{\sigma^2x^{2\alpha}}  \right)  \\
 &\qquad \times \erfc\left( \frac{1}{\sqrt{2\dt}\sigma x^\alpha} \left[(1+K(\alpha)\dt)x -(2-\sqrt{\alpha}){b_\sigma}(\alpha)\dt   \right] \right)  \\
&=:A(x)+B(x).
\end{align*}
Since $\dt\leq 1/(2K(\alpha))$, and  $\erfc(z)\leq \exp(-z^2)$ for all $z>0$ we have
\begin{align*}
A(x)&\leq \frac{1}{2}\exp\left( -  \frac{\left[(1-K(\alpha)\dt)x + \sqrt{\alpha}{b_\sigma}(\alpha)\dt  \right]^2}{2\sigma^2\dt x^{2\alpha}} \right) \leq \frac{1}{2}\exp\left( -  \frac{(1-K(\alpha)\dt)^2x^{2(1-\alpha)} }{2\sigma^2\dt } \right).
\end{align*}
On the other hand, for $x\geq\xx{\sqrt{\alpha}}$, and $\dt\leq 1/(2K(\alpha))$, it follows 
$$x> (2-\sqrt{\alpha}){b_\sigma}(\alpha)\dt/(1+K\dt),$$ 
so the argument of the function $\erfc$ in $B$ is positive, and then
\begin{align*}
B(x)&\leq \frac{1}{2}\exp\left(-\frac{2[{b_\sigma}(\alpha)-K(\alpha)x][x-(1-\sqrt{\alpha}){b_\sigma}(\alpha)\dt]}{\sigma^2x^{2\alpha}}  \right)  \\
&\quad\times \exp\left(- \frac{ \left[(1+K(\alpha)\dt)x -(2-\sqrt{\alpha}){b_\sigma}(\alpha)\dt   \right]^2}{2\dt\sigma^2 x^{2\alpha}} \right)  \\
 &\qquad  =\frac{1}{2}\exp\left( -\frac{\left[(1-K(\alpha)\dt)x + \sqrt{\alpha}{b_\sigma}(\alpha)\dt  \right]^2}{2\dt\sigma^2 x^{2\alpha}}  \right)\leq \frac{1}{2}\exp\left( -  \frac{(1-K(\alpha)\dt)^2x^{2(1-\alpha)} }{2\sigma^2\dt } \right).
\end{align*}
So 
$$\psi(x) = A(x)+B(x) \leq \exp\left( -  \frac{(1-K(\alpha)\dt)^2x^{2(1-\alpha)} }{2\sigma^2\dt } \right),$$
and since the right-hand side is decreasing on $x$, we have \eqref{eq:boundForPsi}.

\paragraph{Step 3. } Let us conclude. Putting together \eqref{eq:boundForPtheta} and \eqref{eq:boundForPsi} we have
\begin{align*}
\P\left( \Theta_\alpha\leq T \right) &\leq \sum_{k=0}^{N-1}{\E\left( \psi( \X_{t_k})\ind{\left\{\X_{t_k}>\xx{\sqrt{\alpha}}\right\}}  \right)}\\
&\leq  \sum_{k=0}^{N-1}{ \exp\left( -  \frac{(1-K(\alpha)\dt)^2 (\xx\sqrt{\alpha})^{2(1-\alpha)} }{2\sigma^2\dt } \right)\P\left(\X_{t_k}>\xx{\sqrt{\alpha}}  \right)}\\
&\leq \frac{C}{\dt}\exp\left( -\frac{\gamma}{\dt}\right)
\end{align*}
with $\gamma = (\xx\sqrt{\alpha})^{2(1-\alpha)} / (8 \sigma^2)$. 
\end{proof}

\begin{proof}[Proof of Lemma \ref{lem:Control2ndMomentLocalTime}]
From \eqref{XbarIsSemiMartingale}, standard arguments show that $\E[ L_T^0(\X)^4]\leq C(T)$. On the other hand, thanks to Corollary VI.1.9 on Revuz and Yor \cite[p. 212]{Revuz:1991aa}, we have almost surely
\begin{align*}
L_{T\wedge\stalpha}^0(\X)&=\lim_{\eps\downarrow0}\frac{1}{\eps}\int_{0}^{T\wedge\stalpha}{\ind{[0,\eps)}(\X_s)d\langle \X \rangle_s}=0,
\end{align*}
because for $\eps<(1-\sqrt{\alpha})b_\sigma(\alpha)\dt$, and $s\leq T\wedge\stalpha$, $\ind{[0,\eps)}(\X_s) = 0.$ a.s.
Now, since 
\begin{align*}
L_T^0(\X) = L_T^0(\X)\ind{\{\stalpha<T\}}+  L_{T\wedge\stalpha}^0(\X)\ind{\{T\leq \stalpha\}} =L_T^0(\X)\ind{\{\stalpha<T\}},
\end{align*}
we can conclude that
\begin{align*}
\E[  L_T^0(\X)^2] =  \E[L_T^0(\X)^2\ind{\{\stalpha<T\}}] \leq \sqrt{\E[ L_T^0(\X)^4] \P\left( \stalpha<T\right)}\leq C\sqrt{\frac{1}{\dt}\exp\left( -\frac{\gamma}{\dt}\right)}.
\end{align*}
\end{proof}

\subsection{On the negative moments of the stopped increment process $(\Z{t\wedge\stalpha})$}

To prove Lemma \ref{lem:CorrectedLocalError} (see section \ref{sec:CorrectedLocalErrorProof} below), we need to control the negative moments of the stopped increment 
  process $\{\Z{t\wedge \stalpha}\}_{0\leq t\leq T}$.  This is the object of the following lemmas, that can be summarize in the following
\begin{lem}\label{lem:Lemmas}
Let $q\geq 1$.  Let $\stalpha$  be the stopping time defined in \eqref{defStoppingTime}.   Let us assume $\dt\leq \dmax(\alpha)$.  
Moreover, let us assume $b(0)>2\alpha(1-\alpha)^2$ when $\alpha\in(\tfrac{1}{2},1)$,  and  $b(0)>\frac{3}{2} \sigma^2(q+1)$ when $\alpha=\tfrac{1}{2}$. 
Then there exists a constant $C$ depending on $b(0)$, $\sigma$, $\alpha$, $T$ and $q$ but not on $\dt$, such that
$$\forall t\in[0,T],\quad\E\left[\Z{t\wedge \stalpha}^{\,-q}\right]\leq C\left(1 + \frac{1}{x_0^q}\right).$$
\end{lem}

\subsubsection*{Existence of Negative Moments. Case $\alpha=\tfrac{1}{2}$}
The proof of the existence of Negative Moments of  $\Z{t\wedge \stalpha}$ has two parts. First we study the quotient   $\Xe{s}/\Z{s}$, and then we proof the main result of the section.

\begin{lem}\label{lem:BoundForZbeingLessThanHalfXCIR}
For $\alpha=\tfrac{1}{2}$, and $\dt \leq  1/(4K)\wedge x_0$ we have
\begin{equation}\label{eq:BoundForZbeingLessThanHalfXCIR}
\sup_{0\leq s\leq T}\P\left(  \Z{s}\leq \frac{\Xe{s}}{2}\right)\leq C\dt^{\frac{15}{8\sigma^2} \,b_\sigma(1/2)}.
\end{equation}
\end{lem}

To prove this lemma, we  need the following auxiliary result, the proof of which is postponed in Appendix \ref{sec:appendix} as a straightforward adaptation of the Lemma 3.6 in \cite{BOSSY:2013fk}.

\begin{lem}\label{LemmaBossyDiop}
 Assume Hypothesis \ref{hypothesisH0} holds, and $b(0)>\sigma^2/4$. Assume also that $\dt \leq   1/(4K)\wedge x_0$. Then, for any $\gamma\geq1$ there exists a constant $C$ depending on the parameters $b(0)$, $K$, $\sigma$, $x_0$, $T$, and also on $\gamma$, such that
$$\sup_{k=0,\ldots,N}\E\exp\left(  -\frac{\X_{t_k}}{\gamma\sigma^2\dt}\right)\leq C\left(  \frac{\dt}{x_0}\right)^{\frac{2}{\sigma^2}b_\sigma(1/2)\left(  1-\frac{1}{2\gamma}\right)}.$$
\end{lem}
\begin{proof}[Proof of Lemma \ref{lem:BoundForZbeingLessThanHalfXCIR}]
We start by proving
\begin{equation}\label{eqBoundForZbeingLessThanHalfX}
\sup_{0\leq s\leq T}\P\left(  \Z{s}\leq \frac{\Xe{s}}{2}\right)\leq \sup_{k=0,\ldots,N}\E\exp\left(  -\frac{\X_{t_k}}{\gamma\sigma^2\dt}\right).
\end{equation}
Indeed, if we call $\ds=s-\eta(s)$, and $\dW_s = (W_s-W_{\eta(s)})$, then
\begin{align*}
\P\left( \Z{s}\leq   \frac{\Xe{s}}{2}\right)&\leq \P\left( \sigma\sqrt{\Xe{s}}\dW_s+b_\sigma(1/2)\ds+(1-K\ds)\Xe{s} \leq   \frac{\Xe{s}}{2}\right)\\
&\leq \E\left[  \P\left( \left.\frac{\dW_s}{\sqrt{\ds}}\leq \frac{b_\sigma(1/2)\ds+(\tfrac{1}{2}-K\ds)\Xe{s}}{\sigma\sqrt{\Xe{s}}}\right|\F_{\eta(s)}\right)\right]\\
&\leq \E \exp\left(  -\frac{(b_\sigma(1/2)\ds+ (\tfrac{1}{2}-K\ds)\Xe{s})^2}{2\sigma^2\ds\Xe{s}}\right)\\
&\leq \E \exp\left(  -\frac{(1-2K\dt)^2\Xe{s}}{8\sigma^2\dt}\right).
\end{align*}
From here, the bound \eqref{eqBoundForZbeingLessThanHalfX} follows easily, and then  we conclude using Lemma \ref{LemmaBossyDiop}.
\end{proof}

\begin{lem}\label{lem:Lemma-Negative-Moments-Zbar-CIR}
Let $\stonehalf$  be the stopping time defined in \eqref{defStoppingTime}, and $q\geq1$. If $\dt\leq \dmax(1/2)$, and
\begin{equation}\label{eq:condParametersNegativeMomentsCIR}
b(0)>\frac{3}{2} \sigma^2(q+1).
\end{equation}
Then there exists a constant $C$ depending on $b(0)$, $\sigma$, $\alpha$, $T$ and $q$ but not on $\dt$, such that
$$\forall t\in[0,T],\quad\E\left[\Z{t\wedge \stonehalf}^{\,-q}\right]\leq C\left(1 + \frac{1}{x_0^q}\right).$$
\end{lem}
\begin{proof}
Let us call $\Delta W_s :=(W_s-W_{\eta(s)})$, and $\Delta s :=(s-\eta(s))$.  By Ito's formula 
\begin{equation}\label{eq:ExpectationItoFormulaNegativeMomentsCIR}
\begin{split}
\E\left[\Z{t\wedge \stonehalf}^{\,-q}\right]=&  \frac{1}{x_0^q}-q\,\E\left[ \int_0^{t\wedge \stonehalf}{\frac{b(\Xe{s})}{\Z{s}^{q+1}}ds}\right]\\
&\quad  + \frac{q(q+1)}{2}\E\left[ \int_0^{t\wedge \stonehalf}{\frac{1}{\Z{s}^{q+2}}\left(\sigma\sqrt{\Xe{s}}+{\frac{\sigma^2}{2}}\dW_s\right)^2  ds}\right].
\end{split}
\end{equation}
But,
\begin{equation}\label{eq:boundCuadraticVariationOfZbarCIR}
\left(\sigma\sqrt{\Xe{s}}+{\frac{\sigma^2}{2}}\dW_s\right)^2 \leq \sigma^2\Xe{s} + \sigma^2\Z{s},\;\;\P-\text{a.s.} 
\end{equation}
Indeed,
\begin{align*}
\left(\sigma\sqrt{\Xe{s}}+{\frac{\sigma^2}{2}}\dW_s\right)^2=& \sigma^2\left( \Xe{s} + \sigma\sqrt{\Xe{s}}\Delta W_s + \frac{\sigma^2}{4} \Delta W_s \right)\\
= & \sigma^2\left( \Xe{s}+\Z{s}- (\Xe{s}+ b(\Xe{s})\Delta s -  \frac{\sigma^2}{4} \ds)\right).
\end{align*}
But, thanks to the Lipschitz property of $b$,
\begin{align*}
\Xe{s} + b(\Xe{s})\ds-  \frac{\sigma^2}{4} \ds &\geq \Xe{s} + \left(  b(0) - K\Xe{s}\right)\ds  -\frac{\sigma^2}{4} \Delta s \\
& \quad = b_\sigma(1/2)\ds +(1- K\ds)\Xe{s} \geq0,
\end{align*}
since $\ds \leq \dt \leq  1/(2K)$, and $b_\sigma(1/2)>0$. So we have \eqref{eq:boundCuadraticVariationOfZbarCIR}. Introducing \eqref{eq:boundCuadraticVariationOfZbarCIR} in \eqref{eq:ExpectationItoFormulaNegativeMomentsCIR}, and using $b(x)\geq b(0)-Kx$, we have
\begin{equation}\label{eq:ItoNegativeMomentsAfterQVCIR}
\begin{split}
\E\left[\Z{t\wedge \stonehalf}^{\,-q}\right]\leq& \frac{1}{x_0^q}-q\,\E\left[\int_0^{t\wedge \stonehalf}{\frac{b(0)}{\Z{s}^{q+1}}ds}\right]+qK\E\left[\int_0^{t\wedge \stonehalf}{\frac{\Xe{s}}{\Z{s}^{q+1}}ds}\right]\\
& \quad + \frac{q(q+1)}{2}\sigma^2\E\left[\int_0^{t\wedge \stonehalf}{\frac{1}{\Z{s}^{q+2}}\left\{\Xe{s} + \Z{s}\right\} ds}\right].
\end{split}
\end{equation}
Since 
\begin{align*}
\frac{\Xe{s}}{\Z{s}}&\leq \frac{\Xe{s}}{\Z{s}}\ind{\{\Z{s}\leq \Xe{s}/2\}}+2,
\end{align*}
and applying H\"older's Inequality for some $\eps>0$, we have 
\begin{equation*}
\begin{split}
\E\left[\Z{t\wedge \stonehalf}^{\,-q}\right]\leq& \frac{1}{x_0^q}-q\E\left[\int_0^{t\wedge \stonehalf}{\frac{b(0)}{\Z{s}^{q+1}}ds}\right]+2qK\E\left[\int_0^{t\wedge \stonehalf}{\frac{1}{\Z{s}^{q}}ds}\right]\\
& + \frac{3q(q+1)}{2}\sigma^2\E\left[\int_0^{t\wedge \stonehalf}{\frac{1}{\Z{s}^{q+1}} ds}\right]\\
&+ \frac{C}{\dt^{q+2}}\int_0^{T}{\left(\E[\Xe{s}^{1/\varepsilon}]\right)^\varepsilon\P\left(\Z{s}\leq \Xe{s}/2\right)^{1-\varepsilon} ds}.
\end{split}
\end{equation*}
Since $b(0)>3\sigma^2(q+1)/2$, we have $15b_\sigma(1/2)/8\sigma^2 > 2q+2$, so choosing $\eps = q/(2q+2)$, and applying Lemma \ref{lem:BoundForZbeingLessThanHalfXCIR}
we have
$$\P\left(\Z{s}\leq \Xe{s}/2\right)^{1-\varepsilon}\leq C\dt^{q+2},$$
and then
\begin{align*}
\E\left[\Z{t\wedge \stonehalf}^{\,-q}\right]\leq& \frac{1}{x_0^q}+2qK\E\left[\int_0^{t\wedge \stonehalf}{\frac{1}{\Z{s}^{q}}ds}\right]
 + q\left(\frac{3(q+1)}{2}\sigma^2- b(0)\right)\E\left[\int_0^{t\wedge \stonehalf}{\frac{1}{\Z{s}^{q+1}} ds}\right] + C.
\end{align*}
Since from the Hypotheses, the third term in the right-hand side is negative, we can conclude thanks to Gronwall's Lemma.
\end{proof}

\subsubsection*{Existence of Negative Moments.  Case $\alpha>\tfrac{1}{2}$}

\begin{lem}\label{lem:boundPZlessFracX}
For $\alpha\in(\tfrac{1}{2},1)$, if $b(0)>2\alpha(1-\alpha)^2\sigma^2$ and $\dt \leq  1/(4\alpha K(\alpha))$, there exists $\gamma>0$ such that
\begin{equation}\label{boundPZlessFracX}
\sup_{0\leq s\leq T}\P\left(\Z{s}\leq \left(1-\frac{1}{2\alpha}\right)\Xe{s}\right) \leq \exp\left(-\frac{\gamma}{\dt}\right).
\end{equation}
\end{lem}
\begin{proof}
Let us call $\Delta W_s :=(W_s-W_{\eta(s)})$, $\Delta s :=(s-\eta(s))$, and 
\begin{equation*}
q(\Xe{s},\dW_s)=\frac{\alpha\sigma^2}{2}\Xe{s}^{2\alpha-1}\Delta W_s^2 + \sigma \Xe{s}^\alpha\Delta W_s+ 
  \frac{\Xe{s}}{2\alpha}  +\left({b_\sigma}(\alpha)  - K(\alpha)\Xe{s}   \right)\Delta s.
\end{equation*}
Notice that for fix $x\in\R$, $q(x,\cdot)$ is a quadratic polynomial. Using \eqref{eq:linearboundForZbar}, we have
\begin{align*}
 \P\Big(\Z{s} \leq \left(1-\frac{1}{2\alpha}\right)\Xe{s}\Big) & \leq \P\Big( q(\Xe{s},\dW_s)   \leq 0 , \Xe{s}\leq \xx \Big) \\
& \quad + \P\Big(q(\Xe{s},\dW_s)    \leq 0 , \Xe{s}\geq \xx \Big),
\end{align*}
where recall, $\xx={b_\sigma}(\alpha)/K(\alpha)$. But
\begin{equation*}
\P\Big[ q(\Xe{s},\dW_s)      \leq 0 , \Xe{s}\leq \xx \Big]=\E\left[  \P\left(   q(x,\sqrt{\ds}\mathcal{N})     \leq 0 \right)\Big|_{x=\Xe{s}}\ind{\left\{ \Xe{s} \leq \xx\right\}}\right],
\end{equation*}
where $\mathcal{N}$ stands for a standard Gaussian random variable. As in the Lemma \ref{lem:ZbarPositiveSmallx}, we have a quadratic polynomial in $\mathcal{N}$, its discriminant is
\begin{align*}
\Delta&= \sigma^2x^{2\alpha}{\Delta s} - 2\alpha\sigma^2x^{2\alpha-1}\Delta s\left[\frac{x}{2\alpha} +\left({b_\sigma}(\alpha)  - K(\alpha)x   \right)\Delta s \right] = -2\alpha\sigma^2x^{2\alpha-1}\Delta s^2\left({b_\sigma}(\alpha)  - K(\alpha)x   \right), 
\end{align*}
so if $x\leq \xx$, $\Delta<0$ and the quadratic form in $\mathcal{N}$ has not real roots, and in particular is non negative almost surely. Then
\begin{equation*}
\P\Big(q(\Xe{s},\dW_s)      \leq 0 , \Xe{s}\leq \xx \Big)=0.
\end{equation*}
On the other hand, 
\begin{align*}
\P\Big(q(\Xe{s},\dW_s)    &\leq 0 , \Xe{s}\geq \xx \Big)	 \\
&\leq \E\left[ \left.\P\left( 
  \mathcal{N}  \leq  - \frac{ {b_\sigma}(\alpha)\Delta s   +\left(\tfrac{1}{2\alpha}- K(\alpha)\Delta s   \right)x  }{ \sigma x^\alpha\sqrt{\Delta s}} \right)\right|_{x=\Xe{s}}\ind{\left\{ \Xe{s} \geq \xx\right\}}\right],
\end{align*} 
and since $\dt\leq 1/(4\alpha K(\alpha))$ we can apply the exponential bound for Gaussian tails and get
\begin{align*}
\P\Big(q(\Xe{s},\dW_s)    &\leq 0 , \Xe{s}\geq \xx \Big)\leq \E\left[ \left.\exp\left( 
  - \frac{ \left(\tfrac{1}{2\alpha}- K(\alpha)\Delta s  \right)^2x^{2(1-\alpha)}  }{ \sigma^2 {\Delta s}} \right)\right|_{x=\Xe{s}}\ind{\left\{ \Xe{s} \geq \xx\right\}}\right].
\end{align*} 
We conclude by taking $\gamma = \xx^{2(1-\alpha)}/(16 \sigma^2)$. 
\end{proof}

\begin{lem}\label{lem:Lemma-Negative-Moments-Zbar}
Let $\stalpha$  be the stopping time defined in \eqref{defStoppingTime}. Let us assume for $\alpha\in(\tfrac{1}{2},1)$, $b(0)>2\alpha(1-\alpha)^2$, and $\dt\leq \dmax(\alpha)$, then for all $q\geq 1$, there exists a constant $C$ depending on $b(0)$, $\sigma$, $\alpha$, $T$ and $p$ but not on $\dt$, such that
$$\forall t\in[0,T],\quad\E\left[\Z{t\wedge \stalpha}^{\,-q}\right]\leq C\left(1 + \frac{1}{x_0^q}\right).$$
\end{lem}
\begin{proof}
Let us call $\dW_s:=W_s-W_{\eta(s)}$. By Ito's formula and the Lipschitz property of $b$,
\begin{multline}\label{eq:ItoFormulaNegativeMoments}
\E\left[\Z{t\wedge \stalpha}^{\,-q}\right]\leq \frac{1}{x_0^q}-q\E\left[\int_0^{t\wedge \stalpha}{\frac{b(0)}{\Z{s}^{q+1}}ds}\right] +qK\E\left[\int_0^{t\wedge \stalpha}{\frac{\Xe{s}}{\Z{s}^{q+1}}ds}\right]\\
 + \frac{q(q+1)}{2}\E\left[ \int_0^{t\wedge \stalpha}{\frac{1}{\Z{s}^{q+2}}\left(\sigma\Xe{s}^{\alpha}+{\alpha\sigma^2\Xe{s}^{2\alpha-1}}\dW_s\right)^2  ds}\right].
\end{multline}
Following the same ideas to prove \eqref{eq:boundCuadraticVariationOfZbarCIR},  for all $s\in[0,t]$  we can easily prove that almost surely  
\begin{equation*}
\left(\sigma\Xe{s}^{\alpha}+{\alpha\sigma^2\Xe{s}^{2\alpha-1}}\dW_s\right)^2 \leq \sigma^2\Xe{s}^{2\alpha} + 2\alpha\sigma^2\Xe{s}^{2\alpha-1}\Z{s}.
\end{equation*}
Introducing this bound in the previous inequality, we have
\begin{equation}\label{eq:ItoNegativeMomentsAfterQV}
\begin{split}
\E\left[ \Z{t\wedge \stalpha}^{\,-q}\right]\leq& \frac{1}{x_0^q}-q\E\left[\int_0^{t\wedge \stalpha}{\frac{b(0)}{\Z{s}^{q+1}}ds}\right]+qK\E\left[\int_0^{t\wedge \stalpha}{\frac{\Xe{s}}{\Z{s}^{q+1}}ds}\right]\\
& + \frac{q(q+1)}{2}\sigma^2\E\left[\int_0^{t\wedge \stalpha}{\frac{1}{\Z{s}^{q+2}}\left\{\Xe{s}^{2\alpha} + 2\alpha\Xe{s}^{2\alpha-1}\Z{s}\right\} ds}\right].
\end{split}
\end{equation}
since for $r\in\{1,2\alpha-1,2\alpha\}$, 
\begin{align*}
\left(\frac{\Xe{s}}{\Z{s}}\right)^r
&\leq \left(\frac{\Xe{s}}{\Z{s}}\right)^r\ind{\{\Z{s}\leq \Xe{s}(1-\tfrac{1}{2}\alpha)\}}+\left(\frac{2\alpha}{2\alpha-1}\right)^r.
\end{align*}
we get 
\begin{align*}
\E\left[\Z{t\wedge \stalpha}^{\,-q}\right] \leq  & \frac{1}{x_0^q}-q\E\left[\int_0^{t\wedge \stalpha}{\frac{b(0)}{\Z{s}^{q+1}}ds}\right]+ \frac{2\alpha}{2\alpha-1}qK\E\left[\int_0^{t\wedge \stalpha}{\frac{1}{\Z{s}^{q}}ds}\right]\\
& \quad + \frac{q(q+1)}{2}\sigma^2\frac{(2\alpha)^{2\alpha+1}}{(2\alpha-1)^{2\alpha}}\E\left[\int_0^{t\wedge \stalpha}{\frac{1}{\Z{s}^{q+2(1-\alpha)}}  ds}\right]\\
&\quad + C\E\left[\int_0^{t\wedge \stalpha}{\left\{\frac{\Xe{s}}{\Z{s}^{q+1}}+\frac{\Xe{s}^{2\alpha}}{\Z{s}^{q+2}} + \frac{\Xe{s}^{2\alpha-1}}{\Z{s}^{q+1}}\right\}\ind{\{\Z{s}\leq \Xe{s}(1-\tfrac{1}{2}\alpha)\}} ds}\right].
\end{align*} 
The last term in the previous inequality is bounded because of the definition of $\stalpha$ and the Lemma \ref{lem:boundPZlessFracX}. Indeed,
\begin{align*}
\E\Big[\int_0^{t\wedge \stalpha}&\left\{\frac{\Xe{s}}{\Z{s}^{q+1}}+\frac{\Xe{s}^{2\alpha}}{\Z{s}^{q+2}} + \frac{\Xe{s}^{2\alpha-1}}{\Z{s}^{q+1}}\right\}\ind{\{\Z{s}\leq \Xe{s}(1-\tfrac{1}{2}\alpha)\}}  ds\Big]  \\
&\leq\frac{C}{\dt^{q+2}}\int_0^T\sqrt{\E\left[\left(\Xe{s}+\Xe{s}^{2\alpha}+\Xe{s}^{2\alpha-1}\right)^2\right]\P\left(\Z{s}\leq \Xe{s}\Big(1-\frac{1}{2\alpha}\Big)\right)}ds\\
&\leq \frac{C}{\dt^{q+2}}\exp\left(\frac{\gamma}{\dt}\right)\leq C.
\end{align*}
So, \eqref{eq:ItoNegativeMomentsAfterQV} becomes
\begin{equation}\label{eq:negativeMomentsAlpha}
\begin{split}
\E\left[\Z{t\wedge \stalpha}^{\,-q}\right]\leq& \frac{1}{x_0^q}-q\E\left[\int_0^{t\wedge \stalpha}{\frac{b(0)}{\Z{s}^{q+1}}ds}\right]+\frac{2\alpha}{2\alpha-1}qK\E\left[\int_0^{t\wedge \stalpha}{\frac{1}{\Z{s}^{q}}ds}\right]\\
& + \frac{q(q+1)}{2}\sigma^2\frac{(2\alpha)^{2\alpha+1}}{(2\alpha-1)^{2\alpha}}\E\left[\int_0^{t\wedge \stalpha}{\frac{1}{\Z{s}^{q+2(1-\alpha)}}  ds}\right] + C.
\end{split}
\end{equation} 
But, for any $A_1,A_2>0$, the mapping $z \mapsto \frac{A_1}{z^{q+2(1-\alpha)}}- \frac{A_2}{z^{q+1}}$ 
is bounded, and \eqref{eq:negativeMomentsAlpha} becomes
\begin{equation*}
\E\left[\Z{t\wedge \stalpha}^{\,-q}\right]\leq \frac{1}{x_0^q}+2qK\E\left[\int_0^{t\wedge \stalpha}{\frac{1}{\Z{s}^{q}}ds}\right]+ C,
\end{equation*} 
from where we can conclude applying Gronwall's Lemma.
\end{proof}

\subsection{On the corrected local error process}\label{sec:CorrectedLocalErrorProof}
\begin{proof}[Proof of Lemma \ref{lem:CorrectedLocalError}]
Let us recall the notation in the proof of the main Theorem
\begin{equation} \label{defDeltaSigma}
\dS_s(\X):= \sigma \X_s^\alpha-\sigma \Xe{s}^\alpha-\alpha\sigma^2\Xe{s}^{2\alpha-1}(W_s-W_{\eta(s)}),  
\end{equation}
and also introduce
$$S_{u\wedge \stalpha}(\X):=\sigma \Xe{s\wedge \stalpha}^\alpha+\alpha\sigma^2\Xe{s\wedge \stalpha}^{2\alpha-1}(W_{u\wedge \stalpha}-W_{\eta(s\wedge \stalpha)}),$$
and $\dW_s:= (W_{s}-W_{\eta(s)})$.

Using Lemma \ref{lem:ProbabilityOfThetaLessThanT}, and the finiteness of the moments of $\dS$, is easy to prove
$$\E\left[\dS_s(\X)^{2p}\right] \leq  C\E\left[ \dS_{s\wedge \stalpha}(\X)^{2p}\ind{\{\stalpha\geq\eta(s)\}} \right] +C\dt^{2p}.$$
Then we only have to prove
\begin{equation}
\E\left[  \dS_{s\wedge \stalpha}(\X)^{2p}\ind{\{\stalpha\geq\eta(s)\}} \right]  \leq C\dt^{2p}.
\end{equation}
Notice that $\X_{s\wedge \stalpha} = \Z{s\wedge \stalpha}$, so
$$\dS_{s\wedge \stalpha}(\X)\ind{\{\stalpha\geq\eta(s)\}} =  \left\{\sigma \Z{s\wedge \stalpha}^\alpha-\sigma \Xe{s\wedge \stalpha}^\alpha-\alpha\sigma^2\Xe{s\wedge \stalpha}^{2\alpha-1}\dW_{s\wedge \stalpha} \right\} \ind{\{\stalpha\geq\eta(s)\}}.$$
Then applying Itô's Formula to the function $\sigma|x|^\alpha$ which is $\mathcal{C}^2$ for $x\geq C\dt$, we have
\begin{equation}\label{eq:ItoCeAlpha}
\begin{split}
\dS_{s\wedge \stalpha}(\X)\ind{\{\stalpha\geq\eta(s)\}} = & \Big\{\int_{\eta(s\wedge \stalpha)}^{s\wedge \stalpha}{ \left( \frac{\alpha\sigma}{\Z{u\wedge \stalpha}^{1-\alpha}}- \frac{\alpha\sigma}{\Xe{s\wedge \stalpha}^{1-\alpha}}  \right)\sigma \Xe{s\wedge \stalpha}^\alpha dW_u}\\
&\qquad+ \int_{\eta(s\wedge \stalpha)}^{s\wedge \stalpha}{\frac{ \alpha^2\sigma^3 \Xe{s\wedge \stalpha}^{2\alpha-1}}{\Z{u\wedge \stalpha}^{1-\alpha}}\dW_{u\wedge\stalpha} dW_u}  \\
&\qquad+\int_{\eta(s\wedge \stalpha)}^{s\wedge \stalpha}{\frac{\alpha\sigma}{\Z{u\wedge \stalpha}^{1-\alpha}}b(\Xe{s\wedge \stalpha}) du}\\
&\qquad- \int_{\eta(s\wedge \stalpha)}^{s\wedge \stalpha}{\frac{1}{2}\frac{\alpha(1-\alpha)\sigma}{\Z{u\wedge \stalpha}^{2-\alpha}}S_{u\wedge \stalpha}(\X)^2 du}\Big\}\ind{\{\stalpha\geq\eta(s)\}}\\
&=: J_1+J_2+J_3-J_4.
\end{split}
\end{equation}
Notice that on the event $\{\eta(s)\leq\stalpha\}$ we have $\eta(s)=\eta(s\wedge\stalpha)$, and then
\begin{equation*}
 \E[ |J_1|^{2p}]  
= \E\left[\left|\int_{\eta(s)}^{s\wedge \stalpha}{ \ind{\{\stalpha\geq\eta(s)\}} \left( \frac{\alpha\sigma}{\Z{u\wedge \stalpha}^{1-\alpha}}- \frac{\alpha\sigma}{\Xe{s\wedge \stalpha}^{1-\alpha}}  \right)\sigma \Xe{s\wedge \stalpha}^\alpha dW_u} \right|^{2p}\right].
\end{equation*}
By the Burkholder-Davis-Gundy inequality,  
there exists a constant $C_p$ depending only on $p$ such that
\begin{multline*}
\E\left[\left|\int_{\eta(s)}^{s\wedge \stalpha}{ \ind{\{\stalpha\geq\eta(s)\}} \left( \frac{\alpha\sigma}{\Z{u\wedge \stalpha}^{1-\alpha}}- \frac{\alpha\sigma}{\Xe{s\wedge \stalpha}^{1-\alpha}}  \right)\sigma \Xe{s\wedge \stalpha}^\alpha dW_u} \right|^{2p}\right] \\
\leq(\alpha\sigma^2)^{2p}C_p \E\left[\left(\int_{\eta(s)}^{s\wedge \stalpha}{ \left( \frac{\Xe{s\wedge \stalpha}^{1-\alpha}-\Z{u\wedge \stalpha}^{1-\alpha}}{\Z{u\wedge \stalpha}^{1-\alpha}\Xe{s\wedge \stalpha}^{1-\alpha}} \right)^2\ \Xe{s\wedge \stalpha}^{2\alpha} \ind{\{\stalpha\geq\eta(s)\}} du}\right)^{p}\right],
\end{multline*}
observing that the integrand in the right-hand side is positive. And  we have
\begin{align*}
\E[ |J_1|^{2p}] &\leq  (\alpha\sigma^2)^{2p}C_p \E\left[\left(\int_{\eta(s)}^{s}{ \left( \frac{\Xe{s\wedge \stalpha}^{1-\alpha}-\Z{u\wedge \stalpha}^{1-\alpha}}{\Z{u\wedge \stalpha}^{1-\alpha}} \right)^2\ \Xe{s\wedge \stalpha}^{4\alpha-2} \ind{\{\stalpha\geq\eta(s)\}} du}\right)^{p}\right]\\
&\leq  C \E\left[\left(\int_{\eta(s)}^{s}\left( \frac{\left(\Xe{s\wedge \stalpha}^{1-\alpha}-\Z{u\wedge \stalpha}^{1-\alpha}\right)}{\Z{u\wedge \stalpha}^{1-\alpha}} \frac{\left(\Xe{s\wedge \stalpha}^{\alpha}+\Z{u\wedge \stalpha}^{\alpha}\right)}{\Z{u\wedge \stalpha}^{\alpha}} \right)^2   \Xe{s\wedge \stalpha}^{4\alpha-2} \ind{\{\stalpha\geq\eta(s)\}} du\right)^{p}\right].
\end{align*}
But for $x,y\geq0$, and $\beta\in[0,\tfrac{1}{2})$ it holds $|x^\beta-y^\beta|(x^{1-\beta}+y^{1-\beta})\leq 2|x-y|$, so
\begin{align*}
\E[ |J_1|^{2p}] &\leq C\E\left[\left(\int_{\eta(s)}^{s}{  \left(\Xe{s\wedge \stalpha}-\Z{u\wedge \stalpha}\right)^2\ind{\{\stalpha\geq\eta(s)\}} \frac{\Xe{s\wedge \stalpha}^{4\alpha-2}}{\Z{u\wedge \stalpha}^{2}}\   du}\right)^{p}\right]\\
&\leq C\dt^{p-1} \int_{\eta(s)}^{s}{ \E\left[ \left(\Xe{s\wedge \stalpha}-\Z{u\wedge \stalpha}\right)^{2p} \ind{\{\stalpha\geq\eta(s)\}} \frac{ \Xe{s\wedge \stalpha}^{2p(2\alpha-1)}}{\Z{u\wedge \stalpha}^{2p}} \right]du}.
\end{align*}
Let $a>1$. Thanks to H\"older's inequality we have
\begin{multline*}
\E\left[ \left(\Xe{s\wedge \stalpha}-\Z{u\wedge \stalpha}\right)^{2p} \ind{\{\stalpha\geq\eta(s)\}} \frac{ \Xe{s\wedge \stalpha}^{2p(2\alpha-1)}}{\Z{u\wedge \stalpha}^{2p}} \right] \\ \leq
 \left(\E\left[ \left[\Xe{s\wedge \stalpha}-\Z{u\wedge \stalpha}\right]^{\frac{2ap}{(a-1)}} \ind{\{\stalpha\geq\eta(s)\}}  \right]\right)^{1-1/a}\left(\E\left[\frac{ \Xe{s\wedge \stalpha}^{2ap(2\alpha-1)}}{\Z{u\wedge \stalpha}^{2ap}} \right]\right)^{1/a}.
\end{multline*}
We  use Lemma \ref{lem:localErrorOfTheScheme} to bound the Local Error of the scheme
$$\E\left[ \left(\Xe{s\wedge \stalpha}-\Z{u\wedge \stalpha}\right)^{\frac{2ap}{(a-1)}} \ind{\{\stalpha\geq\eta(s)\}}  \right] \leq  C \dt^{\frac{ap}{(a-1)}},$$
On the other hand, when $\alpha>\tfrac{1}{2}$, we have control of any negative moment of $\Z{u\wedge\stalpha}$, so
$$\E\left[\frac{ \Xe{s\wedge \stalpha}^{4ap(2\alpha-1)}}{\Z{u\wedge \stalpha}^{2ap}} \right] \leq \sqrt{\E\left[\Xe{s\wedge \stalpha}^{2ap(2\alpha-1)} \right]\E\left[\frac{ 1}{\Z{u\wedge \stalpha}^{4ap}} \right]}\leq C,$$
whereas when $\alpha=\tfrac{1}{2}$,  we choose $a>1$, such that
${2b(0)}/{\sigma^2}> 3(2ap+1),$
so we have control of the $2ap$-th negative moment of $\Z{u\wedge\stalpha}$. And then
$$\E\left[\frac{ \Xe{s\wedge \stalpha}^{4ap(2\alpha-1)}}{\Z{u\wedge \stalpha}^{2ap}} \right] =\E\left[\frac{1}{\Z{u\wedge \stalpha}^{2ap}} \right] \leq C.$$
So, in any case we have
$$\E\left[ \left(\Xe{s\wedge \stalpha}-\Z{u\wedge \stalpha}\right)^{2p} \ind{\{\stalpha\geq\eta(s)\}} \frac{ \Xe{s\wedge \stalpha}^{2p(2\alpha-1)}}{\Z{u\wedge \stalpha}^{2p}} \right] \leq C\dt^p.$$
And then we can conclude $\E[ |J_1|^{2p}] \leq C\dt^{2p}.$

Using the same arguments for $\E[ |J_2|^{2p}]$, we have
\begin{align*}
\E[ |J_2|^{2p}] &\leq C_p \E\left[\left(\int_{\eta(s)}^{s\wedge \stalpha}{\ind{\{\stalpha\geq\eta(s)\}}\frac{ \alpha^2\sigma^6 \Xe{s}^{2(2\alpha-1)}}{\Z{u}^{2(1-\alpha)}}\dW_{u}^2 du}\right)^{p} \right] \\
&\leq C\dt^{p-1}\int_{\eta(s)}^s{\E\left[\ind{\{\stalpha\geq\eta(s)\}}\frac{ \Xe{s}^{2(2\alpha-1)p}}{\Z{u\wedge\stalpha}^{2(1-\alpha)p}}\dW_{u\wedge\stonehalf}^{2p} \right] du}  \\
&\leq C\dt^{p-1}\int_{\eta(s)}^s{\sqrt{\E\left(\frac{ \Xe{s}^{4(2\alpha-1)p}}{\Z{u\wedge \stalpha}^{4(1-\alpha)p}} \right)}\sqrt{\E\left(\ind{\{\stalpha\geq\eta(s)\}}\dW_{u\wedge\stonehalf}^{4p} \right)}du} \leq C\dt^{2p}. 
\end{align*}

To bound $\E[| J_3|^{2p}]$ we proceed as follows
\begin{align*}
\E[ |J_3|^{2p}]   &= \E\left[\left(\int_{\eta(s)}^{s\wedge \stalpha}{\ind{\{\stalpha\geq\eta(s)\}}\frac{\alpha\sigma}{\Z{u\wedge \stalpha}^{1-\alpha}}b(\Xe{s\wedge \stalpha}) du}\right)^{2p}\right]\\
&\leq (\alpha\sigma)^{2p}\dt^{2p-1}\int_{\eta(s)}^{s}{\E\left(\frac{1}{\Z{u\wedge \stalpha}^{2(1-\alpha)p}}b\left(\Xe{s\wedge \stalpha}\right)^{2p}\right) du}\\
&\leq (\alpha\sigma)^{2p}\dt^{2p-1}\int_{\eta(s)}^{s}{\E\left(\frac{1}{\Z{u\wedge \stalpha}^{2p}}\right)^{1-\alpha} \E\left(b\left(\Xe{s\wedge \stalpha}\right)^{\frac{2p}{\alpha}}\right)^\alpha du}\\
&\leq C\dt^{2p}.
\end{align*}

Finally for $\E[ |J_4|^{2p}]$ we consider first $\alpha>\tfrac{1}{2}$. In this case we have control of any negative moment of $\Z{u\wedge\stalpha}$.  So proceeding as before
\begin{align*}
\E[| J_4|^{2p}]  &= \E\left[\left(\int_{\eta(s)}^{s\wedge \stalpha}{\ind{\{\stalpha\geq\eta(s)\}}\frac{1}{2}\frac{\alpha(1-\alpha)\sigma}{\Z{u\wedge \stalpha}^{2-\alpha}}S_{u\wedge \stalpha}(\X)^{2} du}\right)^{2p}\right] \\
&\leq C\dt^{2p-1} \int_{\eta(s)}^{s}{\E\left(\frac{1}{\Z{u\wedge \stalpha}^{2p(2-\alpha)}}S_{u\wedge \stalpha}(\X)^{4p} \right)du}\\
&\leq C\dt^{2p}.
\end{align*}
The case $\alpha = \tfrac{1}{2}$ is a little more delicate. Let us recall the identity used in the  proof of \eqref{eq:boundCuadraticVariationOfZbarCIR} 
$$\left(\sigma\Xe{s\wedge\stonehalf}^{1/2}+\frac{\sigma^2}{2}\dW_{u\wedge\stonehalf}\right)^2
= \sigma^2\Z{u\wedge\stonehalf}
 - \sigma^2\left(b(\Xe{s\wedge\stonehalf})-  \frac{\sigma^2}{4} \right)(u\wedge\stonehalf - \eta(s\wedge\stonehalf)),$$
so, we have from the definition of $\stonehalf$
\begin{align*}
S_{u\wedge \stonehalf}(\X)^{4p}  = \Big(\sigma\Xe{s\wedge\stonehalf}^{\alpha}+\frac{\sigma^2}{2}\dW_{u\wedge\stonehalf}\Big)^{4p}
&\leq C\left( \Z{u\wedge\stonehalf}^{2p}
 + \left(b(\Xe{s\wedge\stonehalf})-  \frac{\sigma^2}{4} \right)^{2p}\dt^{2p}\right)\\
 &\leq C\left( 1
 + \left(b(\Xe{s\wedge\stonehalf})-  \frac{\sigma^2}{4} \right)^{2p}\right)\Z{u\wedge\stonehalf}^{2p}.
\end{align*}
Then
\begin{align*}
\E[ |J_4|^{2p}]  
& \leq  C\dt^{2p-1} \int_{\eta(s)}^{s}{\E\left[\frac{1}{\Z{u\wedge \stalpha}^{3p}}\left( \sigma\Xe{s\wedge\stonehalf}^{\alpha}+\frac{\sigma^2}{2}\dW_{u\wedge\stonehalf}   \right)^{4p} \right]du}\\
& \leq   C \dt^{2p-1} \int_{\eta(s)}^{s}{\E\left[\frac{1}{\Z{u\wedge \stalpha}^{p}}\left( 1
  + \left(b(\Xe{s\wedge\stonehalf})-  \frac{\sigma^2}{4} \right)^{2p}\right) \right]du}\\
& \leq C \dt^{2p-1} \int_{\eta(s)}^{s}{\sqrt{\E\left(\frac{1}{\Z{u\wedge \stalpha}^{2p}} \right)}\sqrt{\E\left(  1
 + \left(b(\Xe{s\wedge\stonehalf})-  \frac{\sigma^2}{4} \right)^{4p}\right)}du}\\
& \leq C \dt^{2p}.
\end{align*}
So for every $\alpha\in[\tfrac{1}{2},1)$, 
$\E[J_4^{2p}]\leq C\dt^{2p}$,  from where we conclude on the Lemma.
\end{proof}

\paragraph{Acknowledgements. }
{\it The authors are  grateful to the anonymous Referees for their useful suggestions and comments on the early version of this work. 

\medskip
The second author acknowledges the following institutions for their financial support:  Proyecto Mecesup UCH0607, the Direcci{\'o}n de Postgrado y Post{\'i}tulo de la Vicerrector{\'i}a de Asuntos Acad{\'e}micos de la Universidad de Chile, the Instituto Franc{\'e}s de Chile - Embajada de Francia en Chile, and the Center for Mathematical Modeling CMM.
}

\appendix
\section{Appendix} \label{sec:appendix}
\begin{proof}[Proof of Lemma \ref{lem:finitenessOfTheMomentsOfXbar}] 
Let us recall the notations $\ds = s-\eta(s)$,  and $\dW_s= W_s-W_{\eta(s)}$. Let us  define $\tau_m = \inf\{t\geq0:\X_t\geq m\}$. Then by Itô's Formula,
 Young's inequality and the Lipschitz property of $b$, we have
\begin{equation}\label{eq:eq1}\begin{split}
\E[ \X_{t\wedge \tau_m}^{2p}] & \leq   x_0^{2p} + C\E\left[ \int_0^{t\wedge \tau_m}{\X_s^{2p} + C+ \Xe{s}^{2p}  ds}\right] \\
 &\quad+C\E\left[\int_0^{t\wedge \tau_m}{ \left(\sigma\Xe{s}^\alpha+\alpha\sigma^2\Xe{s}^{2\alpha-1}\dW_s\right)^{2p} ds}\right].
\end{split}
\end{equation}
From the definition of $\X$, a straightforward computation shows that for all $s\in[0,t]$ almost surely 
\begin{equation*}\label{eqMomentOfXbar1}
\X_s^{2p}    \leq C \left(1 +   \Xe{s}^{2p}+ \dW_s^{\frac{2p}{1-\alpha}}+\left(\dW_s^2 -\ds \right)^{\frac{2p}{2(1-\alpha)}} \right).
\end{equation*}
Putting this in \eqref{eq:eq1}, we have
\begin{align*}
\E[ \X_{t\wedge \tau_m}^{2p}] 
&  \leq  x_0^{2p} + C\E\left[ \int_0^{t\wedge \tau_m}{ 1+ \Xe{s}^{2p}  + \left(\sigma\Xe{s}^\alpha+\alpha\sigma^2\Xe{s}^{2\alpha-1}\dW_s\right)^{2p} ds}\right] \\
&\quad +  C\E\left[ \int_0^{t\wedge \tau_m}{\dW_s^{\frac{2p}{1-\alpha}}+\left(\dW_s^2 -(s-\eta(s)) \right)^{\frac{2p}{2(1-\alpha)}}ds}\right] \\
&  \leq x_0^{2p} + C\E\left[\int_0^{t\wedge \tau_m}{ 1+ \Xe{s}^{2p}  + \Xe{s}^{2p\alpha}+\Xe{s}^{2p(2\alpha-1)}\dW_s^{2p} ds}\right] \\
& \quad +  C\int_0^{T}{\E\left[ \dW_s^{\frac{2p}{1-\alpha}}\right]+\E\left[ \left(\dW_s^2 -\ds \right)^{\frac{2p}{2(1-\alpha)}}\right]ds}.
\end{align*}
Since $\alpha\in[\tfrac{1}{2},1)$ we have $\Xe{s}^{2p\alpha} \leq 1+\Xe{s}^{2p},$ and then, using Young's Inequality and the finiteness of the moments of Gaussian random variables, we conclude
\begin{align*}
\E[ \X_{t\wedge \tau_m}^{2p}]&\leq   Cx_0^{2p} + C\E\left[\int_0^{t\wedge \tau_m}{ \Xe{s}^{2p}ds}  \right] 
\leq Cx_0^{2p} + C\int_0^{t}{\sup_{u\leq s}\E[ \X_{u\wedge \tau_m}^{2p}] ds}. 
\end{align*}
Since the right-hand side is increasing, we can take supremum in the left-hand side  and from here, applying Gronwall's inequality, and  taking $m\to \infty$   we get
$$\sup_{t\leq T}\E[ \X_{t}^{2p}] \leq Cx_0^{2p}.$$
From here, following  standard argument using Burkholder-Davis-Gundy inequality we can conclude on Lemma~\ref{lem:finitenessOfTheMomentsOfXbar}. 
\end{proof}

\begin{proof}[Proof of lemma \ref{LemmaBossyDiop}.]
First, from the definition of $\X_{t_k}$ we have
$$\X_{t_k}\geq \X_{t_{k-1}}+ ({b_\sigma}(1/2) - K\X_{t_{k-1}})\dt + \sigma\sqrt{\X_{t_{k-1}}}\left(  W_{t_{k}}-W_{t_{k-1}}\right),$$
then
$$ \E\exp\left(  -\mu_0\X_{t_k} \right) \leq \E\exp\Big( -\mu_0\Big[\X_{t_{k-1}}+
 ({b_\sigma}(1/2) - K\X_{t_{k-1}})\dt + \sigma\sqrt{\X_{t_{k-1}}}\left(  W_{t_{k}}-W_{t_{k-1}}\right)\Big]\Big),$$
where $\mu_0 = 1/\gamma\sigma^2\dt$. From here, just as in Lemma 3.6 in \cite{BOSSY:2013fk}, we conclude
\begin{equation}\label{firstBoundForExponentialMomentAlpha1/2}
 \E\exp\left(  -\mu_0\X_{t_k} \right) \leq \exp\left(  -\mu_0{b_\sigma}(1/2)\dt\right)
 \E\exp\left( -\mu_0\X_{t_{k-1}}\left[1- K\dt - \frac{\sigma^2\dt}{2}\mu_0\right]\right).
\end{equation}
Then if we introduce the same sequence $(\mu_j)_j\geq0$ of Lemma 3.6 in \cite{BOSSY:2013fk}, given by
\begin{displaymath}
\mu_j = \left\{
\begin{array}{ll}
\frac{1}{\gamma\sigma^2\dt},&j=0,\\
\mu_{j-1}\left[1- K\dt - \frac{\sigma^2\dt}{2}\mu_{j-1}\right],&j\geq1.
\end{array}
\right.
\end{displaymath}
We can repeat the proof in \cite{BOSSY:2013fk} and find out that if $\dt\leq 1/(2K)$ then, the sequence $(\mu_j)_j\geq0$ is nonnegative, decreasing and satisfies the following bound
$$\mu_j\geq \mu_1\left(  \frac{1}{1+\frac{\sigma^2}{2}\dt(j-1)\mu_0}\right)- K\left(  \frac{\dt(j-1)\mu_0}{1+\frac{\sigma^2}{2}\dt(j-1)\mu_0}\right),\;\forall j\geq1.$$
On the other hand making the same calculations to obtain \eqref{firstBoundForExponentialMomentAlpha1/2} we can get for any $j\in\{0,\ldots,k-1\}$,
\begin{equation*}
 \E\exp\left(  -\mu_j\X_{t_{k-j}} \right) \leq \exp\left(  -\mu_j{b_\sigma}(1/2)(\tfrac{1}{2})\dt\right)\E\exp\left( -\mu_j\X_{t_{k-j-1}}\left[1- K\dt - \frac{\sigma^2\dt}{2}\mu_{j+1}\right]\right),
 \end{equation*}
from where, by an induction argument we have
$$\E\left(  -\mu_0\X_{t_k}\right)\leq \exp\left(  -{b_\sigma}(1/2)\sum_{j=0}^{k-1}{\mu_j\dt}\right)\exp{\left( x_0\mu_k \right)}.$$
From here, and the bound for the sequence $(\mu_j)_j\geq0$, we have 
$$\E\left(  -\mu_0\X_{t_k}\right)\leq C\left(  \frac{\dt}{x_0}\right)^{\frac{2{b_\sigma}(1/2)}{\sigma^2}\left(  1-\frac{1}{2\gamma}\right)}.$$
From where we see immediately
$$\sup_{k=0,\ldots,N}\E\exp\left(  -\frac{\X_{t_k}}{\gamma\sigma^2\dt}\right)\leq C\left(  \frac{\dt}{x_0}\right)^{\frac{2{b_\sigma}(1/2)}{\sigma^2}\left(  1-\frac{1}{2\gamma}\right)}.\qedhere$$
\end{proof}

\end{document}